\documentclass[11pt]{article}

\usepackage[]{graphicx,float,latexsym,times}
\usepackage{amsfonts,amstext,amsmath,amssymb,amsthm,bm}
\usepackage{caption}
\usepackage{pgf}
\usepackage{tikz}
\usetikzlibrary{arrows,automata}
\usepackage[latin1]{inputenc}
\usepackage{listings}
\theoremstyle{plain} 
\usepackage{titlesec}
\usepackage[authoryear, round]{natbib}
\usepackage[titletoc]{appendix}
\long\def\symbolfootnote[#1]#2{\begingroup%
\def\thefootnote{\fnsymbol{footnote}}\footnote[#1]{#2}\endgroup}

\usepackage{titlesec} % Very fancy section title control: Sets sections as SQA requires

\titleformat{\section}{\large\bfseries}{\thesection.}{.5em}{}
\titlespacing*{\section}{0pt}{*3}{*2}
\titleformat{\subsection}{\normalfont\bfseries}{\thesubsection.}{.5em}{}
\titlespacing*{\subsection} {0pt}{*3}{*2}
\titleformat{\subsubsection}{\normalfont\bfseries}{\thesubsubsection.}{.5em}{}
\titlespacing*{\subsubsection} {0pt}{*3}{*2}

\newtheorem{claim}{Claim}

\newtheorem{corollary}{Corollary}[section]

\newtheorem{lemma}{Lemma}[section]

\newtheorem{proposition}{Proposition}[section]

\theoremstyle{definition} 
\newtheorem{remark}{Remark}[section]
\newtheorem{definition}{Definition}[section]

\numberwithin{theorem}{section}
\numberwithin{equation}{section}

\newcommand{\mc}{\mathcal}
\newcommand{\ra}{\rightarrow}
\newcommand{\mb}{\mathbb}

\DeclareMathOperator{\ei}{ess\,inf \ }

\newcommand{\bo}{\textbf}

\begin{document}

\title{\textbf{\Large Quickest Detection with Discretely Controlled Observations}}

\date{}

\maketitle

%%%%%%%%% Authors, affiliations %%%%%%%%%%%%%%%%%%%%%%%%%%

\author{
\begin{center}
\vskip -1cm

\textbf{\large Erhan Bayraktar and Ross Kravitz}

Department of Mathematics, University of Michigan,\\
Ann Arbor, Michigan, USA.

\end{center}
}

\symbolfootnote[0]{\normalsize Address correspondence to E. Bayraktar,
Department of Mathematics, University of Michigan, Ann Arbor, MI 48109, USA; E-mail: erhan@umich.edu}

{\small \noindent\textbf{Abstract:} 
 We study a continuous time Bayesian quickest detection problem in which observation times are a scarce resource.  The agent, limited to making a finite number of discrete observations, must adaptively decide his observation strategy to minimize detection delay and the probability of false alarm.  Under two different models of observation rights, we establish the existence of optimal strategies, and formulate an algorithmic approach to the problem via jump operators.  We describe algorithms for these problems, and illustrate them with some numerical results.  As the number of observation rights tends to infinity, we also show convergence to the classical continuous observation problem of Shiryaev.
}\\ \\

\\ \\
%%%%%%%%% Subject Classifications %%%%%%%%%

\maketitle

\section{\uppercase{Introduction}}

Problems in quickest detection, also known as online change point analysis and disorder detection, have been studied for much of the twentieth century.  In such problems, one observes a channel of information whose statistical properties abruptly change at some unknown instant in time.  The basic problem is to determine when this change occurs, using only the observations from the channel.  The mathematical theory of quickest detection begins with \cite{shiryaev1963optimum}, which studies both discrete and continuous time problems.

To describe this problem more precisely, we suppose that on some probability space we have a random variable $\Theta$ taking nonnegative integer values, modeling the disorder time, and a sequence of random variables $Z_1,Z_2,\ldots$ modelling the observations.  Conditionally on the set $\{\Theta = i\}$, the random variables $X_1,X_2,\ldots,X_{i-1}$ have the distribution $P_1$, and the random variables $X_i,X_{i+1},\ldots$ have the distribution $P_2$, where $P_1$ and $P_2$ are distinct but equivalent probabilities.  The random variable $\Theta$ is supposed to have a geometric distribution.  With probability $\pi$, $\Theta = 0$, and $P(\Theta = n) = (1-\pi)(1-p)^{n-1}p$, for some constant $p$ between zero and one.  Let $\mb{F}$ be the filtration generated by $Z_1,Z_2,\ldots$.  For a $\mb{F}$-stopping time $\tau$, the risk $\rho^\pi(\tau)$ is defined by
\[
\rho^\pi(\tau) = P^\pi(\tau < \Theta) + cE \left[ \left(\tau - \Theta \right)^+ \right].
\]
Here, $P^\pi(\tau < \Theta)$ represents the probability of false alarm, $E \left[ \left(\tau - \Theta \right)^+ \right]$ is the expected delay time, and $c$ is a constant which weighs the relative importance of these two terms.  The goal in the quickest detection problem is to find a stopping time $\tau$ which minimizes risk.

In the continuous version of the problem, one observes a process $X$ with $dX_t = dB_t + \alpha 1_{\{t \geq \Theta\}} dt$, where $B_t$ is a standard Brownian motion, and $\alpha$ is a constant.  $\Theta$ is a random variable which is zero with probability $\pi$, and with probability $1-\pi$, it is exponentially distributed with parameter $\lambda$. Both $B_t$ and $\Theta$ are not directly observable. As in the discrete-time formulation, the goal is to estimate the value of $\Theta$ as a result of observing $X$. Our objective is to minimize a weighted average of the probability of false alarm and the detection delay time. The set over which we optimize is the set of stopping times for the filtration generated by $X$.

Since their original formulation, there has been an extensive literature on quickest detection problems which modify or generalize the classical assumptions of \cite{shiryaev1963optimum}.  We mention a few of these papers, although it is impossible to describe the entire literature on the subject.  In \cite{MR1483699}, the deterministic drift term $\alpha$ received after disorder is replaced with a random variable.  In \cite{MR1835037}, the ``linear" penalty for delay, $E \left[ \left(\tau - \Theta \right)^+ \right]$, is replaced with an exponential one, $E \left[e^{\left(\tau - \Theta \right)^+} -1 \right]$.  In \cite{MR2307058}, the continuous time problem is solved  on a finite time horizon, in comparison with the classical infinite horizon case.  In continuous time, it is also natural to study a Poisson process whose intensity rate suddenly changes, and this problem, known as Poisson disorder, has been extensively studied in, for example, \cite{MR1929384}, \cite{MR2158013, MR2260062}, and \cite{MR2233993}.

In all of the above problems, we observe all values of $X$.  Another interesting generalization of the quickest detection problem occurs when we place restrictions on our ability to make observations.  The literature on this problem is relatively sparser, but there is a strand of research from the 1960's and 1970's in which observations are available at all times, but must be purchased.  In \cite{MR0189833}, the costly information quickest detection problem for Brownian Motion is first studied; solutions are not obtained, but some qualitative properties of the value function are established. In \cite{MR0402919}, a problem with both costly information and imperfect observations is studied; again, no explicit solutions are found.  In \cite{MR0431582}, a more thorough analysis of a costly information problem is attempted, but there is a problem with the author's analysis: when the process $X$ is not continuously observed, the posterior process (e.g. $\pi_t \triangleq P \left(\Theta \leq t | \mc{F}_t^X \right)$) ceases to become a sufficient statistic, and the time elapsed since an observation was last made must also be tracked.  
More recent and complete results in this direction are given by \cite{DalangShiryaev}.
On the other hand, \cite{banerjee2012data, banerjee2012data2} consider a discrete-time formulation, in which observations are costly only if they occur before the alarm time.

Alternatively, one may formulate a quickest detection problem in which observations of the process $X$ can be made only at discrete time periods, and such that the agent possesses, no matter what, a fixed amount of observation rights.  For example, we could imagine a remote battery powered sensor which has enough power to make a fixed number of observations and must be active for an extended amount of time.  Such a problem was first studied in \cite{MR2777513}, in which it is assumed that observations fall on a grid which is determined exogenously. If observations can be made only at discrete time periods, it makes sense to consider the case when there is control over when observations can be used: if observations are a limited resource, then the judicious use of them should increase efficiency significantly. In such a scenario, the observation times will no longer be exogenously given, but will be determined adaptively within the problem as part of the optimal strategy. In \cite{MR2777513}, an infinite sequence of of observation times is given. If we allow the controller to choose when observations are made, it does not make sense to allow him infinitely many observation rights, because as we will see, such a problem is degenerate, and equivalent to the classical continuously observed case. Therefore, if observation times can be chosen, there must be some limit on how they can be spent.  In this paper, building off the discrete-time analysis in \cite{bayraktar-lai} and \cite{6862025}, we study such a problem: an agent tries to determine when a disorder occurs, but his ability to observe is constrained.  In the first variant of the problem, we assume that the agent receives a lump sum of $n$ observation rights which he may use as he sees fit.  In the second variant, we assume that an independent Poisson process regulates the times at which new observation rights become available.

We now outline the structure of the paper.  In Section \ref{nobs}, we formulate the lump sum $n$-observation problem, and establish the theoretical existence of optimal strategies.  In Section \ref{nobs_cont}, we demonstrate that as $n \ra \infty$, this $n$-observation problem converges to the classical continuous observation problem.  In Section \ref{stobs}, we formulate the stochastic arrival rate $n$-observation problem, and establish the theoretical existence of optimal strategies.  In Sections \ref{nalg} and \ref{nalg_res}, we give a numerical algorithm for computing the value functions and optimal strategies in the lump sum $n$-observation problem, and illustrate some results from the implementation of this algorithm.  In Section \ref{stalg}, we describe a heuristic algorithm for computing the value functions and optimal strategies in the stochastic arrival rate problem, and illustrate a result from a partial implementation of this algorithm.  Sections \ref{nobs_proof}, \ref{nobs_cont_proof}, and \ref{stobs_proof} contain the technical proofs of the results in Sections \ref{nobs}, \ref{stobs}, and \ref{nobs_cont}.  Finally, Sections 11 and 12 are the appendices of this paper where we establish the dynamics of the posterior process under discrete observations, give selected figures, respectively.

\section{\uppercase{The Lump Sum $n$-Observation Problem: Setup, Existence of Optimal Strategies}}\label{nobs}

Our basic setup is a probability space $(\Omega,\mc{F},P')$, which supports a Wiener process $X = \{X_t\}_{t \geq 0}$ and an independent random variable $\Theta$, which has the same distribution as before: with probability $p$ it is zero, and with probability $1-p$ it is exponentially distributed with parameter $\lambda$.

In \cite{MR2260062}, observation times are determined exogenously, and therefore the information flow is a fixed aspect of the problem.  In contrast, when the agent must decide when to make observations, the information flow is itself variable.  In other words, the filtration is dependent on the observation strategy used.  We therefore have to be somewhat technical in our definition of observation strategies.  We will now inductively define elements in the set of allowed observation strategies, denoted by $\mathfrak{O}^n$.

\begin{definition} We say that a sequence of random variables $\Psi = \{\psi_1,\psi_2,\ldots,\psi_n\} \in \mathfrak{O}^n$ if $\psi_1 \leq \psi_2 \leq \cdots \leq \psi_n$, $\psi_1$ deterministic, and for $1 \leq j \leq n$, $\psi_j \in \\ m \ \sigma(X_{\psi_1},\ldots,X_{\psi_{j-1}},\psi_1,\ldots,\psi_{j-1})$, i.e. $\psi_j$ is measurable with respect to the sigma algebra generated by $X_{\psi_1},\ldots,X_{\psi_{j-1}},\psi_1,\ldots,\psi_{j-1}$.  We set $\psi_0 = 0$, and for convenience take $\psi_{n+1} = \infty$.
\end{definition} 

For each $\Psi \in \mathfrak{O}^n$, let $\mc{F}^\Psi_{\psi_j} = \sigma(X_{\psi_1},\ldots,X_{\psi_j},\psi_1,\ldots,\psi_j)$.  $\Psi$ generates a continuous time filtration $\mb{F}^\Psi = (\mc{F}_t^\Psi)_{t \geq 0}$ in the following way.  We say that $A \in \mc{F}^\Psi_t$ if and only if for each $1 \leq j \leq n$, $A \cap \{\psi_j \leq t\} \in \mc{F}^\Psi_{\psi_j}$.  Intuitively, this means that the set $A$ is known at time $t$ if, for any $j$, it is known at the time of the $j^{th}$ observation, when this observation comes before $t$.  Let $\mc{T}^\Psi$ be the set of $\mb{F}^\Psi$-stopping times which are a.s. finite.

Let $\Phi^\Psi$ be the conditional odds-ratio process that the disorder has occurred, supposing that the observation strategy $\Phi$ has been used.  In other words
\[
\Phi^\Psi_t \triangleq \frac{P(\Theta \leq t | \mc{F}^\Psi_t)}{P(\Theta > t | \mc{F}^\Psi_t)}.
\]

The posterior process $\Phi^\Psi$ can be calculated recursively by the following formula, starting from $\Phi^\Psi_0 = \frac{p}{1-p}$: for more details, please see Appendix \ref{jfun}, which follows the derivation on p. 32-33 of \cite{MR2260062}.
\begin{equation}\label{stateprocess}
\Phi^\Psi_t =
\begin{cases}
\varphi(t - \psi_{n-1},\Phi^\Psi_{\psi_{n-1}}) & \text{if } \psi_{n-1} \leq t < \psi_n
\\ j \left( \Delta \psi_n, \Phi^\Psi_{\psi_{n-1}},\frac{\Delta X_{\psi_n}}{\sqrt{\Delta \psi_n}} \right) & \text{if } t = \psi_n,
\end{cases}
\end{equation}
where $\Delta \psi_n = \psi_n - \psi_{n-1}$, $\Delta X_{\psi_n} = X_{\psi_{n}} - X_{\psi_{n-1}}$, $\varphi(t,\phi) = e^{\lambda t}(\phi + 1) - 1$, and
\begin{eqnarray}\label{jdef}
&& \ \ \ \lefteqn{j(\Delta t, \phi, z)} \\
&& \ \ \ \ \ \ \ \ = \exp \left \{\alpha z \sqrt{\Delta t} + \left( \lambda - \frac{\alpha^2}{2} \right) \Delta t \right\}\phi + \int_0^{\Delta t} \lambda \exp \left\{ \left(\lambda + \frac{\alpha z}{\sqrt{\Delta t}} \right)u - \frac{\alpha^2 u^2}{2 \Delta t} \right \} du.
\end{eqnarray}

According to Lemma $3.1$ of \cite{MR2260062}, the minimum Bayes risk equals $R_n(p) = 1 - p + (1-p)c V_n(p/(1-p))$, where
\begin{equation}\label{nrights}
  \begin{split}
    V_n(\phi) 
    & \triangleq \underset{\Psi \in \mathfrak{O}^n}{\inf} \ \underset{\tau \in \mc{T}^\Psi}{\inf} \ E^\phi \left[ \int_0^\tau e^{-\lambda t} \left( \Phi^\Psi_t -      \frac{\lambda}{c} \right) dt \right],
  \end{split}
\end{equation}
and the expectation $E^\phi[\cdot]$ is with respect to a probability measure $P$ under which $X$ is a standard Weiner process and $\Phi^\Psi_0 = \phi$.  Consequently, \eqref{nrights} is the problem we will focus on.

\begin{proposition}\label{prop1} Let $\Psi \in \mathfrak{O}^n$, and let $\tau$ be an $\mb{F}^\Psi$-stopping time.  Then for each $0 \leq j \leq n$, $\tau 1_{\{\psi_j \leq \tau < \psi_{j+1}\}}$ and $\{\psi_j \leq \tau < \psi_{j+1}\}$ are both $\mc{F}^\Psi_{\psi_j}$-measurable.
\end{proposition}

\begin{proof} The proof is done by a basic modification of Proposition $3.1$ and Theorem $3.2$ of \cite{MR2777513}.  The essential property here is that between observations, there is no flow of new information.
\end{proof}

Define
\[
\mc{T}^\Psi_o \triangleq \left \{ \tau \in \mc{T}^\Psi : \text{for } \omega \in \Omega \text{ with } \tau(\omega) \leq \psi_n(\omega), \tau(\omega) = \psi_j(\omega) \text{ for some } 0 \leq j \leq n \right\},
\]
i.e. those $\mb{F}^\Psi$-stopping times that do not stop between observations.  The following proposition says that, in contrast with \cite{MR2260062}, it is never optimal to stop between observations: if one has a total of $n$ observations at their disposal, he may as well use all of them.

\begin{proposition}\label{nobt} $V_n(\phi) = \underset{\Psi \in \mathfrak{O}^n}{\inf} \ \underset{\tau \in \mc{T}_o^\Psi}{\inf} \ E^\phi \left[ \int_0^\tau e^{-\lambda t} \left( \Phi^\Psi_t - \frac{\lambda}{c} \right) dt \right].$
\end{proposition}

\begin{proof} See Section \ref{nobs_proof}.
\end{proof}

Note that for each $\Psi$, $\Phi^\Psi$ evolves deterministically between observations.  This means that between observations, there is no additional information being accrued.  Therefore, upon making an observation, one may as well determine in that instant when to make the next observation, as opposed to waiting to see what happens $\epsilon$ seconds in the future; no additional information is gained by waiting.  Therefore, the problem is amenable to study by the recursive use of jump operators.  We lay out this strategy now.

%Let 
%\begin{equation}\label{t^*_0}
%t^*_0(\phi) \triangleq \frac{1}{\lambda} \log \left(\frac{c+ \lambda}{c(\phi + 1)} \right) \vee 0,
%\end{equation} 
%and define
%\begin{equation}\label{v_0}
%v_0(\phi) \triangleq \int_0^{t^*_0(\phi)} e^{-\lambda u} \left(\varphi(u,\phi) - \frac{\lambda}{c} \right) du.
%\end{equation}
%\triangleq \underset{t \geq 0}{\inf} \ \left[ (1-\rho)e^{-\lambda t} + ctp + c(1-p) \int_0^t (t-u)\lambda e^{-\lambda u} du \right],

%\noindent so that $v_0$ represents the risk when no further observations can be made.  Note that, if $t^*_0(\phi) > 0$, then $t^*_0(\phi)$ is the time when the posterior process $\Phi_t$ hits the level $\lambda/c$, supposing that it starts from $\phi$ and that observations are not made.  If $t^*_0(\phi) = 0$, corresponding to $\phi \geq \frac{\lambda}{c}$, then $\Psi_t$ is always greater than $\lambda/c$.  Additionally, by construction of $t^*_0(\phi)$,
%\[
%v_0(\phi) = \underset{t \geq 0}{\inf} \int_0^t e^{-\lambda u} \left(\varphi(u,\phi) - \frac{\lambda}{c} \right) du.
%\]
%As $\varphi(u,\phi)$ is concave in $\phi$ for each fixed $u$, it follows that $\int_0^t e^{-\lambda u} \left(\varphi(u,\phi) - \frac{\lambda}{c} \right) du$ is concave in $\phi$ for each $t$, and therefore $v_0(\phi)$ is concave in $\phi$.  It is also evident that $v_0$ is nondecreasing in $\phi$.

For bounded $w: \mb{R}_+ \ra \mb{R}$, define the operators
\begin{equation}\label{Kop}
Kw(t,\phi) \triangleq \int_{-\infty}^\infty w(j(t,\phi,z)) \frac{\exp(-z^2/2)}{\sqrt{2\pi}} dz,
\end{equation}
\begin{equation}\label{Jop}
Jw(t,\phi) \triangleq \int_0^t e^{-\lambda u}\left(\varphi(u,\phi) - \frac{\lambda}{c}\right) du + 1_{\{t>0\}}e^{-\lambda t} Kw(t,\phi),
\end{equation}
and
\begin{equation}\label{J_0op}
J_0 w(\phi) \triangleq \underset{t \geq 0}{\inf} \ Jw(t,\phi).
\end{equation}
Set $v_0(\phi) \triangleq J_0 0$.  Inductively define the value functions, for $j \geq 1$,
\[
v_j(\phi) \triangleq J_0 v_{j-1}(\phi).
\]

%By Lemma \ref{jumpopprop} it follows that $v_j(\phi)$ is concave and nondecreasing for each $j \geq 1$.

Let $0 \leq k \leq n$, and let $\Psi^k =\{\psi^k_1,\ldots,\psi^k_k\} \in \mathfrak{O}^k$.  We set 
\begin{equation}\label{dynamicobs}
\mathfrak{O}^n(\Psi^k) \triangleq \left\{\Psi = \{ \psi_1,\ldots,\psi_n\} \in \mathfrak{O}^n : \psi_i = \psi^k_i, \ 1 \leq i \leq k \right\}.
\end{equation}
\noindent These are the observation strategies whose first $k$ observation times agree with those in $\Psi^k$.  We define the following conditional value functions:
\[
\gamma^n_k(\Psi^k) \triangleq \underset{\Psi \in \mathfrak{O}^n(\Psi^k)}{\ei} \ \underset{\tau \in \mc{T}^\Psi_o, \tau \geq \psi^k_k}{\ei} E \left[ \int_{\psi^k_k}^\tau e^{-\lambda (t-\psi^k_k)} \left( \Phi^\Psi_t - \frac{\lambda}{c} \right) dt \big | \mc{F}^{\Psi^k}_{\psi^k_k} \right].
\]

Proposition \ref{dpp1} allows us to describe the optimization problem in terms of functions defined by the jump operator $J_0$.  It also establishes that the optimization problem is Markov.

\begin{proposition}\label{dpp1} For any $n$, $0 \leq k \leq n$, and $\Psi^k \in \mathfrak{O}^k$, 
\[
\gamma^n_k(\Psi^k) = v_{n-k} \left( \Phi^{\Psi^k}_{\psi_k} \right).
\]
In particular, $V_n(\phi) = v_n(\phi)$.
\end{proposition}

\begin{proof}See Section \ref{nobs_proof}.
\end{proof}

For $0 \leq k < n$ and $\epsilon \geq 0$, define
\[
h^\epsilon_{n-k}(\phi) \triangleq \min \{ s \geq 0 : J v_{n-k}(s,\phi) \leq J_0 v_{n-k}(\phi) + \epsilon \}.
\]
\noindent These functions are used to construct the near optimal strategies needed for the proof of Proposition \ref{dpp1}, but we must show that they are measurable.  Note that in the definition of $h^\epsilon_{n-k}$, we require the first time $s$ such that $J v_{n-k}(s,\phi) \leq J_0 v_{n-k}(\phi) + \epsilon$.  If we simply required any such $\epsilon$-optimal time, we could use a Measurable Selection Theorem, as in \cite{MR0486391}, to imply the measurability of $h^\epsilon_{n-k}$.  Such a theorem, however, would provide only abstract existence.  For computational reasons it is preferable to take the first optimal time.

\begin{lemma}\label{measlem}  For $0 \leq k < n$ and $\epsilon \geq 0$, $\widehat{\psi}^{k+1} \triangleq \psi^k_k + h^\epsilon_{n-k} \left(\Phi^{\Psi^k}_{\psi^k_k} \right) $ is a stopping time, i.e. it is measurable with respect to $\mc{F}^{\Psi^k}_{\psi^k_k}$.
\end{lemma}

\begin{proof}
See Section \ref{nobs_proof}.
\end{proof}

\begin{corollary}  Fix $n \geq 1$ and $\epsilon \geq 0$.  Consider the observation strategy $\widehat{\Psi}^\epsilon \triangleq \widehat{\Psi}^\epsilon(\phi) \triangleq \{\widehat{\psi}^\epsilon_1,\ldots,\widehat{\psi}^\epsilon_n\}$, defined inductively by $\widehat{\psi}^\epsilon_1 \triangleq h^\epsilon_n(\phi)$, and for $2 \leq j \leq n$, $\widehat{\psi}^\epsilon_j \triangleq \widehat{\psi}^\epsilon_{j-1} + h^\epsilon_{n-j} \left(\Phi^{\widehat{\Psi}^\epsilon}_{\widehat{\psi}^\epsilon_{j-1}} \right)$.  Let $\widehat{\tau}^\epsilon(\omega) \triangleq \inf \{ \widehat{\psi}^\epsilon_j(\omega) : \widehat{\psi}^\epsilon_j(\omega) = \widehat{\psi}^\epsilon_{j+1}(\omega) , 0 \leq j \leq n-1 \} \wedge \left( \widehat{\psi}^\epsilon_n(\omega) + t_0^* \left( \Phi^{\widehat{\Psi}^\epsilon}_{\widehat{\psi}^\epsilon_n}(\omega) \right) \right) \in \mc{T}^{\widehat{\Psi}^\epsilon}_o$, where $t_0^*(\phi)$ is defined to satisfy $\varphi(t_0^*(\phi),\phi) = \frac{\lambda}{c}$.
Then
\[
V_n(\phi) \geq E^\phi \left[ \int_0^{\widehat{\tau}^\epsilon} e^{-\lambda t} \left( \Phi^{\widehat{\Psi}^\epsilon} - \frac{\lambda}{c} \right) dt \right] - n\epsilon.
\] 
\end{corollary}

\section{\uppercase{Convergence to the Continuous Observation Problem}}\label{nobs_cont}

In this section, we will show the extended weak convergence of a discretized quickest detection problem to the (classical) continuous observation quickest detection problem, as formulated in \cite{MR2256030}, Chapter $4$.  In all of these problems, the cost functional has the same form, while the dynamics of the underlying odds processes capture the effect of different observation procedures.  The theory of extended weak convergence, as developed by Aldous in \cite{aldous1981weak}, provides a metric under which convergence of optimal stopping problems and their value functions are guaranteed.  
%Let us remark that the question of convergence of optimal stopping problems is a subtle one.  In \cite{MR0235684}, for example, %the authors purport to establish such a convergence result, but their argument uses the fact that hitting times are continuous with %respect to the uniform norm on paths; in reality, they are only lower semi-continuous.

To be more precise, we will show that in a sequence of discrete-time problems, the odds processes $\widetilde{\Phi}^n_t$ extended weak converge to the continuous observation odds process $\phi^{\mathfrak{c}}_t$.  We consider two discrete time problems which are essentially equivalent, one of which fits the model of \cite{MR2777513}.   Studying these problems will give upper bounds for the value function of our (adaptive) $n$-observation problem because they are more restrictive with respect to admissible observation and stopping strategies: in our $n$-observation problem, there is complete freedom over both observation and stopping times, whereas in \cite{MR2777513} there is freedom over the stopping time but observations are confined to a preset grid.  As we will see, the value function $v_{\mathfrak{c}}(\phi)$ in the continuous observation problem always gives a lower bound for our value functions $v_n(\phi)$.  Therefore, we can construct a sequence of functions $\{\widetilde{v}^\mathfrak{D}_n\}_{n \geq 1}$ such that $\widetilde{v}^\mathfrak{D}_n(\phi) \geq v_n(\phi) \geq v_{\mathfrak{c}}(\phi)$ and $\widetilde{v}^\mathfrak{D}_n(\phi) \ra v_{\mathfrak{c}}(\phi)$, which suffices to show that $v_n(\phi) \downarrow v_{\mathfrak{c}}(\phi)$.

\subsection{Review of the Continuous Observation Problem and Comparison to the Lump Sum $n$-Observation Problem}

As before, let $X$ be a standard Brownian motion which gains drift $\alpha$ at the unobservable time $\Theta$, satisfying $P(\Theta = 0) = p$, $P(\Theta \in dt | \Theta >0) = e^{-\lambda t}$.  Let $\mb{F}^c$ be the filtration generated by $X$, and $\mc{S}^{\mathfrak{c}}$ the set of associated stopping times.  In the quickest detection problem with continuous observation, the minimization problem is 
\[
R^{\mathfrak{c}}(p) \triangleq \underset{\tau \in \mc{S}^{\mathfrak{c}}}{\inf} \  P(\tau < \Theta) + cE[(\tau - \Theta)^+].
\]
For details on this problem, see \cite{MR2256030}, Chapter $4$.  Here, $P$ and $E$ refer to a probability measure under which $P(\Theta = 0) = p$.  As in Proposition $2.1$ of \cite{MR2260062}, we may write
\[
R^{\mathfrak{c}}(p) = 1-p + (1-p)c v_{\mathfrak{c}} \left( \frac{p}{1-p} \right),
\]
where $v_{\mathfrak{c}} \left(\frac{p}{1-p} \right) = v_{\mathfrak{c}}(\phi) = \underset{\tau \in \mc{S}^{\mathfrak{c}}}{\inf} E \left[ \int_0^\infty \left(\Phi_t^c - \frac{\lambda}{c} \right) dt \right]$, and $\phi^{\mathfrak{c}}_t$ is the odds process under continuous observation.  The dynamics of $\phi^{\mathfrak{c}}_t$ are given by the following stochastic differential equation, whose derivation is in \cite{MR2256030}:
\begin{equation}\label{contobsphi}
d\phi^{\mathfrak{c}}_t \triangleq \lambda(1 + \phi^{\mathfrak{c}}_t)dt + \alpha \phi^{\mathfrak{c}}_t dW_t,
\end{equation}
with $W$ a standard Brownian Motion.  The quickest detection problem is therefore reformulated as an optimal stopping problem on the diffusion $\phi^{\mathfrak{c}}$.  The following proposition is intuitively clear, since continuous observation is certainly preferable to being limited to a finite set of observation times.  Recall the value function $v_n(\phi)$ of Section \ref{nobs}.
\begin{proposition}\label{prop4.1} For each $n$, $v_n(\phi) \geq v_{\mathfrak{c}}(\phi)$.
\end{proposition}

\begin{proof} Let $n \geq 0$ be fixed.  Let $\Psi = \{\psi_1,\ldots,\psi_n\}$ be an admissible observation strategy, as described in Section \ref{nobs}.  $\Psi$ induces the filtration $\mb{F}^\Psi$, along with its set of stopping times $\mc{S}^\Psi$.  As in \cite{MR2777513}, the optimal stopping problem associated with the observation strategy $\Psi$ is
\[
R^\Psi(p) \triangleq \underset{\tau \in \mc{S}^\Psi}{\inf} R^\Psi_\tau(p),
\]
where $R^\Psi_\tau(p) = P(\tau < \Theta) + cE[(\tau - \Theta)^+]$.  By definition, $\left(\mc{F}^\Psi_t\right)_{t \geq 0} = \mb{F}^\Psi \subset \mb{F}^c = \left(\mc{F}^c_t\right)_{t \geq 0}$, in the sense that $\mc{F}^\Psi_t \subset \mc{F}^c_t$ for each time $t$.  It follows then that $\mc{S}^\Psi \subset \mc{S}^{\mathfrak{c}}$.  Therefore, $R^\Psi(p) \geq R^{\mathfrak{c}}(p)$.  Writing $R^\Psi(p) = 1 - p + (1-p)c v_\Psi \left(\frac{p}{1-p} \right)$, it follows that $v_\Psi \left(\frac{p}{1-p} \right) \geq v_{\mathfrak{c}} \left( \frac{p}{1-p} \right)$.  Let $\mathfrak{O}^n$ denote the set of all admissible $n$-observation strategies.  Then 
\[
v_n(\phi) = \underset{\Psi \in \mathfrak{O}^n}{\inf} \ v_\Psi(\phi) \geq v_{\mathfrak{c}}(\phi).
\]
\end{proof}

\subsection{Defining the Discretized Problem, and the Convergence Result}

We will define two closely related processes, $\widetilde{\Phi}^n$, and $\widetilde{\Phi}^{\mathfrak{D},n}$.  Let $\Delta t = \frac{1}{n}$. The process $\widetilde{\Phi}^n$ will be defined on the grid points $\{0, \Delta t, 2\Delta t, \ldots\}$, and then it will be extended to $\mathbb{R}_+$ as a piecewise constant function. Let $\{Z_1, Z_2, \ldots \}$ be a sequence of i.i.d $N(0,1)$ random variables. We define $\widetilde{\Phi}^n$ and $\widetilde{\Phi}^{\mathfrak{D},n}$ recursively, so that they only differ in between grid points.

\begin{definition} Define the process $\widetilde{\Phi}^n$:
\[
\begin{cases}
\widetilde{\Phi}^n_0  = \phi, \\
\widetilde{\Phi}^n_{k \Delta t } = j \left(\Delta t, \widetilde{\Phi}^n_{(k-1) \Delta t},  Z_k\right) & \mbox{ for } k \in \mathbb{N}, \\
\widetilde{\Phi}^n_t= \widetilde{\Phi}^n_{(k-1)\Delta t} & \mbox{ for } (k-1)\Delta t \leq t < k\Delta t.
\end{cases}
\]
\end{definition}

\begin{definition} Define the process $\widetilde{\Phi}^{\mathfrak{D},n}$:
\[
\begin{cases}
\widetilde{\Phi}^{\mathfrak{D},n}_0 = \phi, \\
\widetilde{\Phi}^{\mathfrak{D},n}_{k \Delta t } = j \left(\Delta t, \widetilde{\Phi}^{\mathfrak{D},n}_{(k-1) \Delta t}, Z_k \right) & \mbox{ for } k \in \mathbb{N}, \\
\widetilde{\Phi}^{\mathfrak{D},n}_t = \varphi \left(\lambda(t - (k-1)\Delta t), \widetilde{\Phi}^{\mathfrak{D},n}_{(k-1)\Delta t} \right) & \mbox{ for } (k-1)\Delta t \leq t < k\Delta t, \\
\end{cases}
\]
\end{definition}
\noindent where $\varphi \left(\lambda(t - (k-1)\Delta t), \widetilde{\Phi}^{\mathfrak{D},n}_{(k-1)\Delta t} \right) = e^{\lambda(t - (k-1)\Delta t)}\left(\widetilde{\Phi}^{\mathfrak{D},n}_{(k-1)\Delta t} + 1 \right) - 1$.  We remark that the dynamics of $\widetilde{\Phi}^{\mathfrak{D},n}$ are precisely those of our $n$-observation problem when observations are taken every $\frac{1}{n}$ units of time.  Since the gaps between observations are deterministic, they are also the typical example of the model in \cite{MR2777513}.  The dynamics of $\widetilde{\Phi}^n$ are modified to make computations more tractable.  Notice also that $\widetilde{\Phi}^n$ and $\widetilde{\Phi}^{\mathfrak{D},n}$ induce the same filtration.  We take $\widetilde{\mathbb{F}}^n$ to be the (continuous time) natural filtration generated by $\widetilde{\Phi}^n_t$, and $\widetilde{\mc{T}}^n$ the set of $\widetilde{\mb{F}}^n$-stoping times.  We set
\[
\widetilde{v}^\mathfrak{D}_n(\phi) \triangleq \underset{\tau \in \widetilde{\mc{T}}^n}{\inf} E \left[ \int_0^\tau e^{-\lambda s} \left( \widetilde{\Phi}^{\mathfrak{D},n}_s - \frac{\lambda}{c} \right) ds \right].
\]

To properly state our result, we need the concept of extended weak convergence, from \cite{aldous1981weak}.  We state the definition for the sake of completeness, but we will essentially only need the fact that extended weak convergence implies convergence of optimal stopping problems.

\begin{definition}  Let $(X, \mb{F})$ be a random process, considered as a random element in $D(\mb{R}_+)$, the set of c\`{a}dl\`{a}g paths on $\mb{R}_+$.  For each $t$, there exists a conditional distribution $Z_t$ for $X$, conditionally on $\mc{F}_t$, and $Z_t$ may be viewed as a random element of $\mc{P}(D(\mb{R}_+))$, the set of probabilities on $D(\mb{R}_+)$.  It is a fact (see Theorem $13.1$ of \cite{aldous1981weak}) that these $Z_t$ can be combined to form a c\`{a}dl\`{a}g process taking values in $\mc{P}(D(\mb{R}_+))$.  This process $Z$ is referred to as the prediction process.  For processes $(X^n,\mb{F}^n)$ and $(X,\mb{F})$, we say that $X^n$ extended converges to $X$, writing $X^n \Rrightarrow X$, if the associated prediction processes $Z^n$ converge weakly to $Z$, i.e. weak convergence of their induced measures on $\mc{P}(D(\mb{R}_+))$.
\end{definition}

Our principal interest in extended weak convergence is derived from the following (in a slightly weakened form) theorem in \cite{aldous1981weak}.  Let $\gamma:[0,\infty) \times \mathbb{R} \rightarrow \mathbb{R}$ be bounded and continuous.  Given a process $(X,\mathbb{F})$, let $\mc{T}_L$ denote its stopping times bounded in size by $L$, and define
\[
\Gamma(L) \triangleq \underset{T \in \mathcal{T}_L}{\sup} E \left[\gamma(T,X_T)\right].
\]

\begin{proposition}\label{prop4}[Theorem $17.2$, Aldous (1981)] Suppose $(X^n,\mathbb{F}^n) \Rrightarrow (X^\infty,\mathbb{F}^\infty)$.  Suppose $(X^\infty, \mathbb{F}^\infty)$ is quasi left continuous (or continuous), and suppose that $\mathbb{F}^\infty$ is the usual filtration for $X^\infty$.  Then $\Gamma_n(L) \rightarrow \Gamma_\infty(L)$.
\end{proposition}

Our goal, therefore, is to show that $\widetilde{\Phi}^n \Rrightarrow \phi^{\mathfrak{c}}$.  The following two results from \cite{aldous1981weak} yield a feasible strategy for establishing extended weak convergence to a diffusion.  The effectiveness of this method lies in the fact that the (complicated) limiting process never needs to be directly studied; this is the basic property of establishing weak convergence.  The message of the next two Propositions is the following: the standard way that one shows weak convergence is Proposition \ref{ald1}, but in fact weak convergence is strictly weaker than the hypotheses in this Proposition.  It turns out that these conditions are exactly equivalent to extended weak convergence.  Thus, by following the ``standard" method for establishing weak convergence, one obtains the more powerful extended weak convergence for free.

\begin{proposition}[Theorem $8.22$, Aldous (1981)]\label{ald1} Let $a(x) > 0$ and $b(x)$ be bounded continuous functions and let $x_0 \in \mathbb{R}$. Let $X$ be the diffusion with drift $b(x)$ and variance $a(x)$, and $X_0 = x_0$. Let $(X^n, \mathbb{F}^n)$ be a sequence of processes. Suppose that for all $L > 0$
\begin{itemize}
\item[(a)] $X^n_0 \Rightarrow X_0$
\item[(b)] $E \left[ \underset{t \leq L}{\sup} \ \left(X^n_t - X^n_{t-} \right)^2 \right] \rightarrow 0$ as $n \rightarrow \infty$.
\end{itemize}

Suppose also that for each $n$, there exist $N^n_t$ and $\mathcal{N}^n_t$ adapted to $\mathbb{F}^n$ such that for all $L > 0$
\begin{itemize}
\item[(c)] $(M^n, \mathbb{F}^n)$ is a martingale, where $M^n_t = X^n_t - \int_0^t b(X^n_s) ds - N^n_t$
\item[(d)] $(S^n, \mathbb{F}^n)$ is a martingale, where $S^n_t = \left(M^n_t \right)^2 - \int_0^t a(X^n_s) ds - \mathcal{N}^n_t$
\item[(e)] $\underset{T \in \mathcal{T}^n_L}{\sup} \ E \left[ \left( N^n_T \right)^2 \right] \rightarrow 0$ as $n \rightarrow \infty$
\item[(f)] $\underset{T \in \mathcal{T}^n_L}{\sup} \ E \left[ | \mathcal{N}^n_T | \right] \rightarrow 0$ as $n \rightarrow \infty$,
\end{itemize}

\noindent where $\mathcal{T}^n_L$ is the set of $\mathbb{F}^n$-stopping times bounded by $L$. Then $X^n \Rightarrow X$ (i.e., weak convergence).
\end{proposition}

\begin{proposition}[Proposition $21.17$,  Aldous (1981)]\label{ald2} Let $(Y^n, \mathbb{F}^n)$ be a sequence of processes, and $X$ the diffusion with drift $b(x)$ and variance $a(x)$. In order that $Y^n \Rrightarrow X$ (i.e. extended weak convergence), it is necessary and sufficient that there exist $X^n$ adapted to $\mathbb{F}^n$ such that
\begin{itemize}
\item[(i)] $\underset{t \leq L}{\sup} \ |X^n_t - Y^n_t| \rightarrow 0$ in probability
\item[(ii)] $(X^n, \mathbb{F}^n)$ satisfies the hypotheses of Proposition \ref{ald1}.
\end{itemize}
\end{proposition}

\begin{proposition}\label{mainprop} As $n \rightarrow \infty$, $\widetilde{\Phi}^n_t \Rrightarrow \phi^{\mathfrak{c}}_t$
\end{proposition}

\begin{proof} The proof consists of checking the six conditions in Proposition \ref{ald1}, which necessitates establishing some moment inequalities on $\widetilde{\Phi}^n$.  We refer the reader to Section \ref{nobs_cont_proof}.
\end{proof}

\begin{corollary} As $n \ra \infty$, $\widetilde{\Phi}^{\mathfrak{D},n} \Rrightarrow \phi^{\mathfrak{c}}$.
\end{corollary}

\begin{proof}  Let $g_n(\phi) = e^{\frac{\lambda}{n}}\left(\phi + 1 \right) - 1 - \phi = (\phi + 1) \cdot O \left(\frac{1}{n} \right)$.

Note that 
\[
\underset{0 \leq t \leq L}{\sup} \ |\widetilde{\Phi}^n_t - \widetilde{\Phi}^{\mathfrak{D},n}_t| = g_n \left(\underset{0 \leq k \leq k_{max}}{\max} \ \widetilde{\Phi}^n_\frac{k}{n} \right).
\]
As in the proof of Proposition \ref{mainprop}, $\widetilde{\Phi}^n$ is a submartingale, and so by Doob's $L^2$ Inequality,
\[
E \left[ \underset{0 \leq t \leq L}{\sup} \ |\widetilde{\Phi}^n_t - \widetilde{\Phi}^{\mathfrak{D},n}_t|^2 \right] = O \left(\frac{1}{n^2} \right) E \left[ \left(1 + \widetilde{\Phi}^n_T \right)^2 \right].
\]
Using the moment bounds on $\widetilde{\Phi}^n$, which we will establish in Section \ref{nobs_cont_proof}, we see that this last quantity above is $O \left(\frac{1}{n^2} \right)$.  Now, we can see that Condition (i) in Proposition \ref{ald2} is satisfied.  Applying it with Proposition \ref{mainprop}, we deduce the Corollary.
\end{proof}

\begin{corollary} As $n \ra \infty$, $\widetilde{v}^\mathfrak{D}_n(\phi) \ra v_{\mathfrak{c}}(\phi)$ and $v_n(\phi) \ra v_{\mathfrak{c}}(\phi)$.
\end{corollary}

\begin{proof}
  First, note that for any $\epsilon > 0$, there exists a $L = L(\epsilon)$ such that for all $n$,
\[
\widetilde{v}^\mathfrak{D}_n(\phi) > \underset{\tau \in \widetilde{\mc{T}}^n, \tau \leq L}{\inf} E \left[ \int_0^\tau e^{-\lambda s} \left( \widetilde{\Phi}^{\mathfrak{D},n}_s - \frac{\lambda}{c} \right) ds \right] - \epsilon,
\]
and the same type of inequality holds true for $v_{\mathfrak{c}}(\phi)$.  This is because the running reward function at time $s$  is greater than $-\frac{\lambda}{c} e^{-\lambda s}$, as $\widetilde{\Phi}^{\mathfrak{D},n}$ and $\phi^{\mathfrak{c}}$ are nonnegative.  Therefore, the value functions $\widetilde{v}^\mathfrak{D}_n$ and $v_{\mathfrak{c}}$ are uniformly approximated by problems where the allowed stopping times are uniformly bounded.  Therefore, to show that $\widetilde{v}^\mathfrak{D}_n$ converges to $v_{\mathfrak{c}}$, we may assume that all stopping times are bounded by some constant $L$.

Now, we cannot apply Proposition \ref{prop4} directly, since the value functions $\widetilde{v}^\mathfrak{D}_n$ and $v_{\mathfrak{c}}$ are optimal stopping problems, not on $\widetilde{\Phi}^{\mathfrak{D},n}$ and $\phi^{\mathfrak{c}}$, but on their time integrals.  Fortunately, there is a simple way to work around this technical difficulty, using one last result from \cite{aldous1981weak}.

\begin{lemma}\label{lem5} Let $H:D(\mathbb{R}) \rightarrow D(\mathbb{R})$ be a continuous mapping such that if $f(u) = g(u)$ for $u \leq t$ then $(Hf)(t) = (Hg)(t)$.  Then if $(X^n,\mathbb{F}^n) \Rrightarrow (X^\infty,\mathbb{F}^\infty)$ and $Y^n = H(X^n)$, $(Y^n,\mathbb{F}^n) \Rrightarrow (Y^\infty, \mathbb{F}^\infty)$.
\end{lemma}

For $H$ defined by $(Hf)(t) = \int_0^t f(s) ds$, it is clear that the conditions of Lemma \ref{lem5} are satisfied, at the very least when $H$ is resticted to continuous paths.  Therefore, $\int_0^\cdot \widetilde{\Phi}^{\mathfrak{D},n}_s ds \Rrightarrow \int_0^\cdot \phi^{\mathfrak{c}}_s ds$.  Therefore, by Proposition \ref{prop4}, we have $\widetilde{v}^\mathfrak{D}_n \ra v_{\mathfrak{c}}$.  In computing $v^\mathfrak{D}_n$, we take $\lfloor Ln \rfloor$ observations, so $v_n^\mathfrak{D} \geq v_{\lfloor Ln \rfloor} \geq v_{\mathfrak{c}}$.  By the monotonicity of $v_n$ with respect to $n$, it follows that $v_n \ra v_{\mathfrak{c}}$.
\end{proof}

\section{\uppercase{The Stochastic Arrival Rate $n$-Observation Problem: Setup, Existence of Optimal Strategies}}\label{stobs}

We will consider two subcases of this problem.  First, we assume that a total of $n$ observation rights arrive via a Poisson process.  Second, we assume that the rates arrive indefinitely from a Poisson process.  The second case will be addressed as a limiting case of the former.  Suppose that, in addition to supporting a Wiener process $X$ and the random variable $\Theta$, the space $(\Omega,P)$ supports an independent, completely observable Poisson process $\{N_t\}_{t \geq 0}$ with arrival rate $\mu > 0$.  Let $\eta_1 \leq \eta_2 \leq \cdots$ denote the increasing sequence of jumps times of $N$.  For convenience, take $\eta_0=0$.  We will define the set of allowed observation strategies for both the ``$n$ total observation rights" problem and for the infinite observation rights problem.  When we consider the ``$n$ total observation rights" problem, we will stop $N$ after $n$ arrivals, and assume that $\eta_{n+1} = \infty$.  As before, we will first define the set of admissible observation strategies.

\begin{definition}\label{finhorobs}  For a sequence of random variables $\psi_1 \leq \psi_2 \leq \cdots \leq \psi_n$, we say $\Psi = \{\psi_1,\ldots,\psi_n\}$ is an admissible observation strategy in the stochastic arrival rate $n$-observation problem, written $\Psi \in \bm{\mathfrak{O}}^n$, if 
\[
\psi_j \in m \ \sigma(X_{\psi_1},\ldots,X_{\psi_{j-1}},\psi_1,\ldots,\psi_{j-1},\eta_1,\ldots,\eta_j) \text{ and }\psi_j \geq \eta_j
\]
for each $1 \leq j \leq n$.  For convenience, we will always set $\psi_0 = 0$ for any $\Psi$.
\end{definition}

\begin{definition}\label{infhorobs}  For $\Psi = \{\psi_1,\psi_2,\ldots\}$, we say that $\Psi$ is an admissible observation strategy in the stochastic arrival rate infinite observation problem, written $\Psi \in \bm{\mathfrak{O}}^{\infty}$, if for each $n$, $\{\psi_1,\ldots,\psi_n\} \in \bm{\mathfrak{O}}^n$.
\end{definition}

Using the same construction as in the previous section, each $\Psi \in \bm{\mathfrak{O}}^n$, $1 \leq n \leq \infty$, induces a continuous time filtration $\widetilde{\mb{F}}^\Psi = (\widetilde{\mc{F}}^\Psi_t)_{t \geq 0}$ which is built up from the discrete observations made at times $\psi_i$.  We take $\mb{F}^\Psi = \widetilde{\mb{F}}^\Psi \vee \mb{F}^N$, $\mb{F}^N$ being the filtration generated by the Poisson process $N$.  Each $\Psi$ induces the set $\mc{T}^\Psi$ of $\mb{F}^\Psi$-stopping times and the observed posterior process $\Phi^\Psi$ defined by \eqref{stateprocess}.

Take $1 \leq n \leq \infty$.  As before, according to Lemma $3.1$ of \cite{MR2260062}, the minimum Bayes risk equals $\bo{R}_n(p) = 1 - p + (1-p)c \bo{V}_n(p/(1-p))$, where
\begin{equation}\label{lasteq7}
\begin{split}
\bo{V}_n(\phi) 
& \triangleq \underset{\Psi \in \bm{\mathfrak{O}}^n}{\inf} \ \underset{\tau \in \mc{T}^\Psi}{\inf} \ E^\phi \left[ \int_0^\tau e^{-\lambda t} \left( \Phi^\Psi_t - \frac{\lambda}{c} \right) dt \right],
\end{split}
\end{equation}
and the expectation $E^\phi[\cdot]$ is with respect to a probability measure $P$ under which $X$ is a standard Weiner process and $\Phi^\Psi_0 = \phi$.  Hence, we will focus on solving \eqref{lasteq7}.

We will first specialize to the case where $n<\infty$.  As in the previous section, in considering the problem $\bo{V}_n$, we can optimize over a smaller set of stopping times than $\mc{T}^\Psi$.

\begin{definition}  Let $\Psi \in \bm{\mathfrak{O}}^n$, and let $\tau \in \mc{T}^\Psi$.  We say that $\tau \in \mc{T}^\Psi_s$ if $\{\psi_i < \tau < \psi_{i+1} \} \cap \{ \psi_i \geq \eta_{i+1} \} = \emptyset$, for each $0 \leq i \leq n - 1.$
\end{definition}

Note that $\{\psi_i \geq \eta_{i+1}\}$ represents the scenarios when, after making observation $i$, the agent has additional observation rights stockpiled.  Therefore, a stopping time $\tau \in \mc{T}_s^\Psi$ is one that does not stop while there are unused observation rights.  As in Section \ref{nobs}, we have

\begin{proposition}\label{prop2.1} $\bo{V}_n(\phi) 
= \underset{\Psi \in \bm{\mathfrak{O}}^n}{\inf} \ \underset{\tau \in \mc{T}_s^\Psi}{\inf} \ E^\phi \left[ \int_0^\tau e^{-\lambda t} \left( \Phi^\Psi_t - \frac{\lambda}{c} \right) dt \right]$.
\end{proposition}

\begin{proof}  Let $\Psi \in \bm{\mathfrak{O}}^n$, and $\tau \in \mc{T}^\Psi$.  First, note that every stopping time $\tau \in \mc{T}^\Psi$ satisfies $\{\psi_0 < \tau < \psi_1\} \cap \{\psi_0 \geq \eta_1 \} = \emptyset$, simply because $\{\psi_0 \geq \eta_1 \} = \emptyset$, $\eta_1$ being a strictly positive random variable.  The proof now is essentially identical to that of Proposition \ref{prop1}.  As before, if we were to stop the game while having unused observation rights, we could construct a new observation strategy which adds in an additional observation at that stopping time, without changing the value received.
\end{proof}

Now, we will define some operators, which have analogs in the lump sum $n$-observation problem.  Let $\Lambda^F \subset \mb{R}^2_+$ denote the set of feasible values of the state process $(t,\Phi_t)$.  Precisely:
\[
\Lambda^F = \left \{ (y,\phi) : y \geq 0, \phi \geq e^{\lambda y} - 1 \right \}.
\]
Here $y$ represents the time since the last observation.  In the absence of observations, the trajectory of the state process follows the path $(t,e^{\lambda t}(\phi + 1) - 1)$, starting from $\phi$ at time zero.  Since $\phi \geq 0$ at time zero, all trajectories must lie in $\Lambda^F$.  

Recall the operator $K$ from \eqref{Kop}.  We will extend it as follows: for $w:\Lambda^F \ra \mb{R}$ bounded, define
\begin{equation}
\bm{K} w(t,\phi) \triangleq \int_{-\infty}^{\infty} w(0,j(t,\phi,z)) \frac{\exp(-z^2/2)}{\sqrt{2\pi}} dz.
\end{equation}
In the next two operators, the ``0" superscript stands for ``no observations stockpiled".  We define, for $w:\Lambda^F \ra \mb{R}$ bounded,
\begin{flalign}\label{J^0op}
& J^0 w(t,y,\phi) & \nonumber \\
& \triangleq \int_0^\infty \mu e^{-\mu u} \left( \int_0^{u \wedge t} e^{-\lambda r} \left( \varphi(r,\phi) - \frac{\lambda}{c} \right) dr + e^{-\lambda u} 1_{\{t > u\}} w(y + u, \varphi(u,\phi)) \right) du, &
\end{flalign}
\begin{equation}\label{J^0_0op}
J^0_0 w(y,\phi) \triangleq \underset{t \geq 0}{\inf} \ J^0 w(t,y,\phi).
\end{equation}
Let us explain the operator $J^0$.  It describes the situation in which the agent has no observations stockpiled, the posterior is $\phi$, and $y$ units of time have passed since the last observation was made.  Faced with this scenario, he stops at time $t$, which may be prior to the arrival time $u$ of the next observation right, or after it.  An agent will be left with no observations stockpiled only if he has just used an observation, so for these operators $y$ will effectively be zero.  For subsequent operators, we will consider scenarios where $y$ is positive, and so for this reason we keep the notation consistent.
%, and when we take up the study of optimal behavior between times of activity (observations spent or received), $y$ will not be zero.

Next we define jump operators $J^+$ and $J^+_0$, corresponding to the scenario when the agent has stockpiled observation rights after he has either just made an observation or received an observation right.   We define, for $w^1,w^2: \Lambda^F \ra \mb{R}$ bounded,
\begin{eqnarray}\label{J^+op}
\lefteqn{\ \ \ \ \ J^+(w^1,w^2)(t,y,\phi)} \\
&& \ \triangleq \int_0^\infty \mu e^{-\mu u} \Bigg( \int_0^{u \wedge t} e^{-\lambda r} \left(\varphi(r,\phi) - \frac{\lambda}{c} \right) dr \nonumber \\ 
&& \ \ \ \ \ \ + e^{-\lambda t} 1_{\{t < u\}} \bm{K} w^1\left(y+ t,\varphi(-y,\phi) \right) + e^{-\lambda u} 1_{\{u \leq t\}} w^2(y+u, \varphi(u,\phi)) \Bigg) du, \nonumber
\end{eqnarray}
\begin{equation}\label{J^+_0op}
J^+_0(w^1,w^2)(y, \phi) \triangleq \underset{t \geq 0}{\inf} \  J^+(w^1,w^2)(t,y,\phi).
\end{equation}
From Proposition \ref{prop2.1}, we have seen that it is never optimal for an agent to stop while he has unused observation rights.  Therefore, if he has observation rights stockpiled, the agent either observes immediately ($t=0$), which is equivalent to stopping, or chooses his next observation time $t>0$.  If $u$ is the next arrival time of an additional observation right, then his next observation time $t$ may be either prior to or after the arrival of the next observation right.  Here of the two continuation functions $w^1$ and $w^2$, $w^1$ corresponds to this former scenario, and $w^2$ to the latter.  The variable $y$ denotes the amount of time that has passed since the agent has last made an observation, which may be nonzero if an observation right has arrived more recently than the last time an observation was made.  

We will need one more pair of operators, corresponding to the times when all $n$ observation rights have been received.  Note that this scenario explains why the ``lump sum $n$ observation rights" problem is essentially embedded in this one.  Therefore, note the similarity between $J^e, J^e_0$, defined below, and $J,J_0$, defined in \eqref{J_0op}.  The main difference consists in allowing $y$ to be nonzero, allowing for the possibility that time has elapsed since the last observation.  We define, for $w:\Lambda^F \ra \mb{R}$ bounded,
\begin{equation}\label{J^eop}
J^e w(t,y,\phi) \triangleq \int_0^t e^{-\lambda r} \left( \varphi(r,\phi) - \frac{\lambda}{c} \right) dr + e^{-\lambda t} \bm{K} w(y+t,\varphi(-y,\phi)),
\end{equation}
\begin{equation}\label{J^e_0op}
J^e_0 w(y,\phi) \triangleq \underset{t \geq 0}{\inf} J^e w(t,y,\phi).
\end{equation}

Fix $1 \leq n < \infty$.  Set $\bm{v}^n_{n,n+1}(y,\phi) \triangleq 0$.  For $0 \leq k \leq n$, set $\bm{v}^n_{n,k}(y,\phi) \triangleq J^e_0 \bm{v}^n_{n,k+1}(y,\phi)$.  The superscript ``n" corresponds to $n$ total observation rights, while the subscript ``n,k" corresponds to $n$ observation rights received and $k$ observations used.  Note that when there are $n$ observation rights arriving stochastically, it is the case that once all of these $n$ observations have arrived, we essentially revert to the lump sum problem. We now define $\bm{v}^n_{n-1,n-1}(y,\phi) \triangleq J^0_0 \bm{v}^n_{n,n-1}(y,\phi)$ and, for $0 \leq k < n-1$, $\bm{v}^n_{n-1,k}(y, \phi) \triangleq J^+_0 (\bm{v}^n_{n-1,k+1},\bm{v}^n_{n,k})(y,\phi)$.  Proceeding inductively in this way, we define
\[
\bm{v}^n_{j,j}(y,\phi) \triangleq J^0_0(\bm{v}^n_{j+1,j})(y,\phi), 0 \leq j \leq n,
\]
\[
\bm{v}^n_{j,k}(y,\phi) \triangleq J^+_0(\bm{v}^n_{j,k+1},\bm{v}^n_{j+1,k})(y,\phi), 0 \leq k < j \leq n.
\]

The function $\bm{v}^n_{j,k}(y,\phi)$ is a value function, representing the value when there are $n$ total observation rights, of which $j$ have been received and $k \leq j$ spent, the current posterior level is $\phi$, and $y$ units of time have elapsed since the last observation was made.  Note that by definition of $J^0_0$, $\bm{v}^n_{j,j} \leq 0$, $0 \leq j \leq n$.  From this and the definition of $J^+_0$, it also follows that $\bm{v}^n_{j,k} \leq 0$ for all $j$ and $k$.  We illustrate the relationship between the value functions and the jump operators through figures \ref{st_scheme} and \ref{st_comp}, found in Appendix \ref{figs}.

Fix $0 \leq k \leq n$, and let $\Psi^k = \{\psi^k_1,\ldots,\psi^k_k\} \in \bm{\mathfrak{O}}^k$, from Definition \ref{finhorobs}.  We set, for $k \leq j \leq n$,
\[	
\bm{\mathfrak{O}}^n_{j,k}(\Psi^k) \triangleq \left \{ \Psi = \{\psi_1,\ldots,\psi_n\} \in \bm{\mathfrak{O}}^n : \psi_i = \psi^k_i, 1 \leq i \leq k \text{ and } \psi_{k+1} \geq \eta_j \right \}.
\]
Intuitively, $\bm{\mathfrak{O}}^n_{j,k}(\Psi^k)$ consists of the observation strategies one can pursue after observing at $\psi^k_1,\ldots,\psi^k_k$, and refraining from observing next until $\eta_j$.  Note that the last requirement $\psi_{k+1} \geq \eta_j$ is vacuous when $j=k,k+1$.  We let, for $0\leq k \leq n$ and $k \leq j \leq n$,
\[
\bm{\gamma}^n_{j,k}(\Psi^k) \triangleq \underset{\Psi \in \bm{\mathfrak{O}}^n_{j,k}(\Psi^k)}{\ei} \ \underset{\tau \in \mc{T}^\Psi_s, \tau \geq \psi^k_k \vee \eta_j}{\ei} E \left[ \int_{\psi^k_k \vee \eta_j}^\tau e^{-\lambda (t-\psi^k_k \vee \eta_j)} \left( \Phi^\Psi_t - \frac{\lambda}{c} \right) dt \big | \mc{F}^{\Psi^k}_{\psi^k_k \vee \eta_j} \right].
\]
Note that the ``reference" time above is $\psi^k_k \vee \eta_j$.  We are in a scenario where $j$ observation rights have been received and $k$ spent; if $\psi^k_k > \eta_j$, we arrived at this state from ``$j$ observation rights received, $k-1$ observations spent", and if $\eta_j > \psi^k_k$, we arrived at this state from ``$j-1$ observations received, $k$ observations spent".

\begin{proposition}\label{prop2.2} For any $n$, $0 \leq k \leq j \leq n$ and $\Psi^k = \{\psi^k_1,\ldots,\psi^k_k\} \in \bm{\mathfrak{O}}^k$,
\begin{equation}\label{prop2.2eq1}
\bm{\gamma}^n_{j,k}(\Psi^k) \geq \bm{v}^n_{j,k} \left(\psi^k_k \vee \eta_j - \psi^k_k, \Phi^{\Psi^k}_{\psi^k_k \vee \eta_j} \right) \end{equation}
on the set $\{\psi^k_k < \eta_{j+1}\}$.
\end{proposition}

\begin{proof}
See Section \ref{stobs_proof}.
\end{proof}

For the proof of the other inequality, we will need to construct some optimal stopping times, describing when one should either observe the process or stop and accept the change hypothesis.  We will do this inductively, with the help of some auxiliary functions.  Set $s^n_{n,n}(y,\phi) = s^n_{n,n}(\phi) \triangleq t_0^*(\phi)$, defined by 

\begin{equation}\label{t^*_0}
t^*_0(\phi) = \frac{1}{\lambda} \log \left(\frac{c+ \lambda}{c(\phi + 1)} \right) \vee 0.
\end{equation} 

For $0 \leq k < n$, define $o^n_{n,k}(y,\phi) \triangleq \inf \left \{ s \geq 0 : J^e \bm{v}^n_{n,k+1}(s,y,\phi) = J^e_0 \bm{v}^n_{n,k+1}(y,\phi) \right\}$.  We define, for $0 \leq j < n$, 
\[s^n_{j,j}(y,\phi) \triangleq \inf \left \{s \geq 0 : J^0 \bm{v}^n_{j+1,j}(s,y,\phi) = J^0_0 \bm{v}^n_{j+1,j}(y,\phi) \right \}\] 
and, for $0 \leq j < n$, $0 \leq k < j$,
\[
o^n_{j,k}(y,\phi) \triangleq \inf \left \{s \geq 0 : J^+ \left(\bm{v}^n_{j,k+1},\bm{v}^n_{j+1,k} \right)(s,y,\phi) = J^+_0 \left(\bm{v}^n_{j,k+1},\bm{v}^n_{j+1,k} \right)(y,\phi) \right \}.
\]

The notation ``s'' and ``o'' stands for, respectively, stop, and observe.  This is in line with the reasoning that one should stop only when there are no available observation rights, i.e. $j=k$.

For $\Psi^k = \{\psi^k_1,\ldots,\psi^k_k \} \in \bm{\mathfrak{O}}^k$, we define the ``action times" (either stopping or making an observation) $\widehat{\tau}^n_{j,k}$, $k \leq j \leq n$.  Set 
\[
\begin{split}
\widehat{\tau}^n_{n,k} & = \widehat{\tau}^n_{n,k}(\Psi^k) \\
& \triangleq \psi^k_k \vee \eta_n + o^n_{n,k}\left(\psi^k_k \vee \eta_n - \psi^k_k,\Phi^{\Psi^k}_{\psi^k_k \vee \eta_n} \right),
\end{split}
\]
and we will inductively define, on the set $\{\psi^k_k < \eta_{j+1}\}$, $k < j < n$,
{\small{
\begin{eqnarray*}
\lefteqn{\widehat{\tau}^n_{j,k} = \widehat{\tau}^n_{j,k}(\Psi^k)} \\
&& \triangleq
\begin{cases}
\psi^k_k \vee \eta_j + o^n_{j,k} \left(\psi^k_k \vee \eta_j - \psi^k_k,\Phi^{\Psi^k}_{\psi^k_k \vee \eta_j} \right) & \text{ if } \psi^k_k \vee \eta_j + o^n_{j,k} \left(\psi^k_k \vee \eta_j - \psi^k_k,\Phi^{\Psi^k}_{\psi^k_k \vee \eta_j} \right) < \eta_{j+1}
\\ \widehat{\tau}^n_{j+1,k}(\Psi^k) & \text{ if } \psi^k_k \vee \eta_j + o^n_{j,k} \left(\psi^k_k \vee \eta_j - \psi^k_k,\Phi^{\Psi^k}_{\psi^k_k \vee \eta_j} \right) \geq \eta_{j+1}
\end{cases}
\end{eqnarray*}
}}
and
\begin{eqnarray*}
\lefteqn{\widehat{\tau}^n_{k,k} = \widehat{\tau}^n_{k,k}(\Psi^k)} \\
&& \triangleq
\begin{cases}
\psi^k_k + s^n_{k,k} \left(\Phi^{\Psi^k}_{\psi^k_k} \right) & \text{ if } \psi^k_k + s^n_{j,k}\left(\Phi^{\Psi^k}_{\psi^k_k} \right) < \eta_{k+1}
\\ \widehat{\tau}^n_{k+1,k}(\Psi^k) & \text{ if } \psi^k_k + s^n_{k,k} \left(\Phi^{\Psi^k}_{\psi^k_k} \right) \geq \eta_{k+1}.
\end{cases}
\end{eqnarray*}

\begin{proposition}\label{prop2.3} For any $n$, $0 \leq k \leq j \leq n$ and $\Psi^k = \{\psi^k_1,\ldots,\psi^k_k\} \in \bm{\mathfrak{O}}^k$,
on the set $\{\psi^k_k < \eta_{j+1}\}$, 
\begin{equation}\label{prop2.3.1}
\bm{\gamma}^n_{j,k}(\Psi^k) \leq \bm{v}^n_{j,k} \left(\psi^k_k \vee \eta_j - \psi^k_k, \Phi^{\Psi^k}_{\psi^k_k \vee \eta_j} \right). 
\end{equation}
Hence, $\bm{\gamma}^n_{j,k}(\Psi^k) = \bm{v}^n_{j,k} \left(\psi^k_k \vee \eta_j - \psi^k_k, \Phi^{\Psi^k}_{\psi^k_k \vee \eta_j} \right)$.  Furthermore, on the set $\{\psi^j_j < \eta_{j+1}\}$,
\begin{multline}\label{prop2.3.2}
\bm{v}^n_{j,j}\left(0,\Phi^{\Psi^j}_{\psi^j_j}\right) = E \Bigg[ \int_{\psi^j_j}^{\widehat{\tau}^n_{j,j} \wedge \eta_{j+1}} e^{-\lambda (s - \psi^j_j)} \left(\varphi \left(s - \psi^j_j,\Phi^{\Psi^j}_{\psi^j_j} \right) - \frac{\lambda}{c} \right) ds 
\\ + e^{-\lambda(\eta_{j+1} - \psi^j_j)}1_{\{\widehat{\tau}^n_{j,j} > \eta_{j+1} \}} \bm{v}^n_{j+1,j} \left(\eta_{j+1} - \psi^j_j, \varphi \left(\eta_{j+1}-\psi^j_j,\Phi^{\Psi^j}_{\psi^j_j} \right) \right) | \mc{F}^{\Psi^j}_{\psi^j_j} \Bigg],
\end{multline}
and for $k < j$, on the set $\{\psi^k_k < \eta_{j+1}\}$,
\small{
\begin{eqnarray}\label{prop2.3.3}
\lefteqn{\ \ \ \bm{v}^n_{j,k} \left(\psi^k_k \vee \eta_j - \psi^k_k, \Phi^{\Psi^k}_{\psi^k_k \vee \eta_j} \right)} \\
&& = E \Bigg[ \int_{\psi^k_k \vee \eta_j}^{\widehat{\tau}^n_{j,k} \wedge \eta_{j+1}} e^{-\lambda (s-\psi^k_k \vee \eta_j)} \left(\varphi \left(s - \psi^k_k \vee \eta_j,\Phi^{\Psi^k}_{\psi^k_k \vee \eta_j} \right) - \frac{\lambda}{c} \right) ds \nonumber \\
&& \ \ \ \ \  + e^{-\lambda (\widehat{\tau}^n_{j,k} - \psi^k_k \vee \eta_j)} 1_{\{\widehat{\tau}^n_{j,k}<\eta_{j+1} \}} \bm{K} \bm{v}^n_{j,k+1}\left(\widehat{\tau}^n_{j,k} - \psi^k_k, \varphi \left(-(\psi^k_k \vee \eta_j - \psi^k_k), \Phi^{\Psi^k}_{\psi^k_k \vee \eta_j} \right) \right) \nonumber \\
&& \ \ \ \ \  + e^{-\lambda (\eta_{j+1} - \psi^k_k \vee \eta_j)} 1_{\{ \eta_{j+1} \leq \widehat{\tau}^n_{j,k}\}} \bm{v}^n_{j+1,k}\left(\eta_{j+1} - \psi^k_k, \varphi \left(\eta_{j+1} - \psi^k_k \vee \eta_j, \Phi^{\Psi^k}_{\psi^k_k \vee \eta_j} \right) \right) \big| \mc{F}^{\Psi^k}_{\psi^k_k \vee \eta_j} \Bigg].\nonumber
\end{eqnarray}
}
In particular, $\bm{V}_n(\phi) = \bm{v}^n_{0,0}(\phi)$.

\end{proposition}

\begin{proof} See Section \ref{stobs_proof}.

\end{proof}

As a consequence of Proposition \ref{prop2.3}, we may inductively describe the optimal observation strategies and stopping times, which are as follows.  Consider a given instant of time, when an observation has just been spent or an observation right has just been received.  Let $j$ be the number of observation rights received, $k \leq j$ be the number of observation rights used, let $\phi$ be the current value of the posterior process, and let $y$ be the amount of time elapsed since the last time an observation was made.

\begin{corollary}  The following observation/stopping strategy is optimal.  Suppose that an observation has just been made or an observation right has just been received.

\textbf{Observation Strategy}
\begin{itemize}
\item[(1)] If $k = j$, there are no available observation rights.  Wait until time $\eta_{j+1}$, increment $j$ by $1$, and proceed to (2) with the appropriate changes to $\phi, y$.
\item[(2)] If $k < j$, calculate $\widehat{\tau}^n_{j,k}$, which is a function of $j,k,\phi,y$, and the current time.  If $\widehat{\tau}^n_{j,k} < \eta_{j+1}$, spend an observation right at time $\widehat{\tau}^n_{j,k}$, and increment $k$ by $1$.  If $k+1 <j$, proceed to (2) and if $k+1 = j$, proceed to (1), making the appropriate changes to $\phi$ and $y$.  Otherwise, if $\widehat{\tau}^n_{j,k} \geq \eta_{j+1}$, increment $j$ by $1$, and proceed to (2) with the appropriate changes to $\phi$ and $y$.
\end{itemize}

\textbf{Stopping Strategy} If $k = j$ and $\widehat{\tau}^n_{j,j} < \eta_{j+1}$, stop the game at time $\widehat{\tau}^n_{j,j}$.  If $\widehat{\tau}^n_{j,j} \geq \eta_{j+1}$, increment $j$ by $1$ and proceed to Observation Strategy (2).  The game is never stopped when $k < j$.
\end{corollary}

\begin{remark} In the corollary above, we say that the agent optimally stops the game only when he has no spare observation rights.  This is essentially a formalism.  One can envision a scenario in which the agent makes an observation, and notices that the posterior is at a very high level, indicating that it is very likely that the disorder has occurred.  The agent will want to stop the game immediately.  In our setup, the agent, if he has spare observation rights, will exercise them all instantaneously to get to the point where he has no observation rights remaining, after which he will stop the game.
\end{remark}

\subsection{The Infinite Horizon Problem}

We consider now the subcase of the stochastic arrival problem in which observation rights continue to arrive indefinitely.  The following proposition says that the value function in the infinite arrival problem is approximated uniformly by the value function in the $n$ arrival problem, as $n$ goes to infinity.  Since the strategy space for the $n$ arrival problem is contained within the strategy space for the infinite arrival problem, it therefore follows that we may use optimal strategies in the $n$ arrival problem to find near optimal strategies in the infinite arrival problem.

\begin{proposition}\label{prop2.4} The value functions $\bm{V}_n(\phi)$ converge to $\bm{V}_\infty(\phi)$ as $n \ra \infty$, uniformly over $\phi$.  More precisely, $0 \leq \bm{V}_\infty - \bm{V}_n \leq \frac{1}{c} \left( \frac{\mu}{\mu + \lambda} \right)^{n+1}$.
\end{proposition}  

\begin{proof} The inequality $\bm{V}_\infty \leq \bm{V}_n$ is an immediate consequence of the fact that $\bm{\mathfrak{O}}^n$ is naturally included in $\bm{\mathfrak{O}}^\infty$, i.e. for any element $\Psi$ of $\bm{\mathfrak{O}}^n$, there is an element $\widetilde{\Psi}$ of $\bm{\mathfrak{O}}^\infty$ such that the first $n$ observation times of $\widetilde{\Psi}$ coincide with those of $\Psi$.

For the second inequality, let $\Psi = \{\psi_1,\psi_2,\ldots\}$ be an arbitrary element of $\bm{\mathfrak{O}}^\infty$.  By definition, it must be the case that $\psi_{n+1} \geq \eta_{n+1}$, and that $\{\psi_1,\ldots,\psi_n\}$ is an element of $\bm{\mathfrak{O}}^n$.  Noting that for all $\Psi$, the posterior process $\Phi^\Psi$ is positive, it follows that
\begin{equation}\label{prop2.4.1}
\bm{V}_\infty - \bm{V}_n \geq E \left[\int_{\eta_{n+1}}^\infty e^{-\lambda s}\frac{-\lambda}{c} ds \right] = \frac{1}{c} E \left[ e^{-\lambda \eta_{n+1}} \right].
\end{equation}
Now, $\eta_{n+1}$, being the sum of $n+1$ independent exponential random variables with parameter $\mu$, has the Erlang distribution $\eta_{n+1} \sim \text{Erlang}(n+1,\mu)$, which has Laplace transform $f^*(s) = \left( \frac{\mu}{\mu + s} \right)^{n+1}$.  It therefore follows that the right hand side of \eqref{prop2.4.1} above is equal to $\frac{1}{c} \left( \frac{\mu}{\mu + \lambda} \right)^{n+1}$, which tends to zero as $n \ra \infty$. 
\end{proof}

\begin{remark} A similar argument may be used to show the convergence of $\bm{v}^n_{0,0}(0,\phi)$ ( $n$ total observations arriving stochastically, of which none have yet arrived) to $v_n(\phi)$, ($n$ total observations, all of which are available), as the \emph{arrival rate} $\mu \ra \infty$.  Suppose that for some $\phi$, $t_n^*(\phi)$, the optimal time to make the first observation in the lump sum $n$-observation problem, is strictly positive.  Using the cumulative distribution of an Erlang random variable, it is easily calculated that the probability that all $n$ observation rights arrive before time $t_n^*(\phi)$ is
\[
1 - \sum_{k=0}^{n-1} \frac{1}{k!} e^{-\mu t_n^*(\phi)} (\mu t_n^*(\phi))^k.
\]
When $n$ and $\phi$ are fixed, so that $t_n^*(\phi)$ is fixed, this expression converges to $1$ almost exponentially fast as $\mu \ra \infty$.  If all observation rights arrive before time $t^*_n(\phi)$, then the stochastic arrival of the observation rights imposes no restriction on observation strategies vis-a-vis the scenario in which all $n$-observation rights are available all along, because the agent has received all observation rights by the time he wishes to make even a single observation.  Therefore, with probability at least $1 - \sum_{k=0}^{n-1} \frac{1}{k!} e^{-\mu t_n^*(\phi)} (\mu t_n^*(\phi))^k$ the strategy that one would pursue in the Lump Sum $n$-observation problem is also feasible in the stochastic arrival rate problem.  This implies that for fixed $n$ and $\phi$, $\bm{v}^n_{0,0}(0,\phi)$ should converge at least almost exponentially fast to $v_n(\phi)$ as $\mu \ra \infty$.  Note that a uniform rate of convergence over all $\phi$ is not guaranteed.  When $t^*_n(\phi)$ is very close to zero, it becomes increasingly important to have observation rights immediately available.  Additionally, this argument does not hold uniformly over all $n$ as $n \ra \infty$.  In fact, the convergence rate of $\bm{v}^\infty_{0,0}(0,\phi)$ to $v_{\mathfrak{c}}(\phi)$ as a function of $\mu$ will be comparable to the convergence of $v_n(\phi)$ to $v_{\mathfrak{c}}(\phi)$ as a function of $n$, which is rather slow (see Table \ref{tab1}).  This is because, in any finite time interval, the expected number of received observation rights will be proportional to the arrival rate $\mu$.
\end{remark}

\section{\uppercase{An Algorithm for the Lump Sum $n$-Observation Problem}}\label{nalg}

In this section, we explicitly describe an algorithm for computing the value functions $v_0,v_1,\ldots,v_N$, as well as the boundaries which determine when observations should be made.  We give a rigorous construction which shows how solutions may be constructed up to any specified error tolerance. We have the following main result in this section, giving worst case error bounds:

\begin{proposition}\label{prop2} Fix a positive integer $N$.  Then in $O(\frac{N^6}{\epsilon^3})$ function evaluations, we may uniformly approximate $v_0(\phi),v_1(\phi),\ldots, v_N(\phi)$ to within $\epsilon$.
\end{proposition}

We note that in the process of calculating the value functions, we also determine the boundaries which determine the optimal observation behavior.  We outline the steps of the algorithm in Subsection \ref{sec_nobs_code}.  In Subsection \ref{sec_nobs_Jw}, we justify the error bounds of Step $2$ of the algorithm.  In Subsection \ref{sec_nobs_phi}, we explain the error bounds of Step $3$, as well as explaining how an upperbound $\overline{\phi}$ may be constructed.  Finally, in Subsection \ref{sec_nobs_alln}, we give error bounds for iterating Steps $2$ and $3$ multiple times.

\subsection{Pseudo-Code for the Algorithm}\label{sec_nobs_code}

Here we outline the steps of the algorithm.  Subsequent parts of this section will explain why such an algorithm works to uniformly approximate the value functions.

\begin{itemize}

\item[(1)] Fix $N$, the total number of observations.  Discretize the $\phi$ variable into $\phi_0 = 0,\phi_1,\ldots,\phi_J = \overline{\phi}$, and set all value functions equal to zero for $\phi \geq \overline{\phi}$.

\item[(2)]  The function $v_0(\phi)$ can be analytically computed.  Fix $\phi_j$, and approximately minimize $t \mapsto Jv_0(t,\phi_j)$ by computing $Jw(t_i,\phi_j)$, for $t_0 = 0,t_1,\ldots,t_K = T$ a discretization of $t$, and $T$ an upper bound on the size of optimal $t$, established in Lemma \ref{lemma2}.  Let the minimizer be $\widehat{t}^*_n(\phi_j)$.

\item[(3)]  Having computed above an approximation to $v_1(\phi_j)$, interpolate these values in a piecewise constant fahion to obtain a function $\widehat{v}_1(\phi)$ which approximates $v_1(\phi)$.

\item[(4)]  Let the collection of points $(\widehat{t}^*_n(\phi_j), \varphi(\widehat{t}^*_n(\phi_j), \phi_j))$ define the observation barrier.

\item[(5)]  Repeat Steps $2$, $3$, except now minimizing $t \mapsto J \widehat{v}_1(t,\phi)$, to obtain an approximation $\widehat{v}_2(\phi)$ to $v_2(\phi)$.

\item[(6)]  Continue this procedure until $\widehat{v}_N(\phi)$ is computed.

\end{itemize}

\subsection{Minimizing $t \mapsto Jw(t,\phi)$ for $\phi$ Fixed}\label{sec_nobs_Jw}

\begin{lemma}\label{lemma1} For each $n \geq 0$, $\frac{-1}{c} \leq V_n(\phi) \leq 0$ for all $\phi$.
\end{lemma}

\begin{proof}  According to Proposition \ref{dpp1}, $V_n = v_n$, where $v_{-1} \equiv 0$, and for $n \geq 0$, $v_n = J_0 v_{n-1}$.  Here $J_0$ is the jump operator defined in Section \ref{nobs}.  Note that $0$ clearly satisfies the conclusion of the lemma.  Therefore, it suffices to establish the inductive step: if $\frac{-1}{c} \leq v_n \leq 0$, then $-\frac{1}{c} \leq J_0 v_n \leq 0$.  The upper bound follows from $J_0 v_n (\phi) \leq J v_n (0,\phi) = 0$.  For the lower bound, calculate that for any $t$,
\begin{eqnarray*}
\lefteqn{J v_n (t,\phi)} \\
&& = \int_0^t e^{-\lambda u} \left( \varphi(u,\phi) - \frac{\lambda}{c} \right) du + e^{-\lambda t} K v_n(t,\phi) \\
&& \geq \int_0^t e^{-\lambda u} \left( \varphi(u,\phi) - \frac{\lambda}{c} \right) du + e^{-\lambda t} \left( \frac{-1}{c} \right) \\
&& \geq \int_0^t e^{-\lambda u} \left(- \frac{\lambda}{c} \right) du + e^{-\lambda t} \left( \frac{-1}{c} \right) \\
&&=\frac{-1}{c}.
\end{eqnarray*}
Taking the infimum over all $t$ yields $J_0 v_n(\phi) \geq \frac{-1}{c}$.
\end{proof}

\begin{lemma}\label{lemma1.1} For each $n \geq 0$, $v_n(\phi)$ is concave and increasing in $\phi$.
\end{lemma}
\begin{proof} Follows by definition, and the fact that an infimum of concave functions is again concave.
\end{proof}

\begin{lemma}\label{lemma2} For $T \triangleq \frac{1}{\lambda}\left( 1 + \frac{\lambda}{c} \right) + \frac{1}{c}$ and each $n$, 
\[
\begin{split}
v_n(\phi) & = J_0 v_{n-1}(\phi) \\
& = \underset{0 \leq t \leq T}{\inf} \int_0^t e^{-\lambda u}\left( \varphi(u,\phi) - \frac{\lambda}{c} \right) du + e^{-\lambda t} K v_{n-1}(t,\phi).
\end{split}
\]
In other words, the optimal time $t^*$ can be assumed to be less than or equal to $T$.
\end{lemma}

\begin{proof} Note that 
\[
\begin{split}
\int_0^t e^{-\lambda u} \left(\varphi(u,\phi) - \frac{\lambda }{c} \right) du 
&= \int_0^t \left[ \phi + 1 - e^{-\lambda u}\left(1 + \frac{\lambda}{c} \right) \right] du
\\&= (\phi + 1)t + \frac{1}{\lambda} \left(1 + \frac{\lambda}{c}\right)(e^{-\lambda t} - 1)
\\&\geq t - \frac{1}{\lambda}\left( 1 + \frac{\lambda}{c} \right).
\end{split}
\]
It follows therefore, that for $t \geq \frac{1}{\lambda}\left( 1 + \frac{\lambda}{c} \right) + \frac{1}{c} = T$, $Jv_n(t,\phi) \geq 0$ for any $n$.  Here we have used the uniform lower bound for $v_n$ established in Lemma \ref{lemma1}.  By the upper bound in that Lemma, $v_n \leq 0$, so it is sufficient to minimize $Jv_n(t,\phi)$ over $ t \in [0,T]$.
\end{proof}

\begin{lemma}\label{lemma3}  Let $|| \cdot ||_{Lip}$ denote the Lipschitz norm.  For each $n \geq 0$, $||v_n||_{Lip} \leq T + ||v_{n-1}||_{Lip}$.
\end{lemma}

\begin{proof}  Take $\phi_1 < \phi_2$.  We have, using Lemma \ref{lemma2} for the first inequality,
\[
\begin{split}
|v_n(\phi_1) - v_n(\phi_2)|
& \leq \underset{0 \leq t \leq T}{\sup} \left| \int_0^t \left[e^{-\lambda u} \left(\varphi(u,\phi_1) - \frac{\lambda}{c} \right) - e^{-\lambda u} \left(\varphi(u,\phi_2) - \frac{\lambda}{c} \right) \right] du \right|
\\& \ \ \ \ \ \ + \underset{0 \leq t \leq T}{\sup} \left| e^{-\lambda t} \left(K v_{n-1}(t,\phi_1) - K v_{n-1}(t,\phi_2) \right) \right|.
\end{split}
\]
We treat these two terms on the right hand side separately.  We calculate that the first term is actually equal to 
\[
\underset{0 \leq t \leq T}{\sup} \left| \int_0^t (\phi_1 - \phi_2) du \right| = T|\phi_1 - \phi_2|.
\]
To calculate the second term, fix $t \in [0,T]$.  Then
\begin{eqnarray*}
\lefteqn{\left| e^{-\lambda t} \int_{-\infty}^\infty \Big( v_{n-1}(j(t,\phi_1,z)) - v_{n-1}(j(t,\phi_2,z)) \Big) \frac{e^{-z^2/2}}{\sqrt{2 \pi}} dz \right|} \\
&& \leq ||v_{n-1}||_{Lip} \ e^{-\lambda t} \int_{-\infty}^\infty |j(t,\phi_1,z) - j(t,\phi_2,z)| \frac{e^{-z^2/2}}{\sqrt{2 \pi}} dz \\
&& = ||v_{n-1}||_{Lip} e^{-\lambda t} |\phi_1 - \phi_2| \int_{-\infty}^\infty e^{\alpha z \sqrt{t} + (\lambda - \alpha^2/2)t} \frac{e^{-z^2/2}}{\sqrt{2 \pi}} dz \\
&& = ||v_{n-1}||_{Lip} e^{-\lambda t} |\phi_1 - \phi_2| e^{\lambda t} \\
&& = ||v_{n-1}||_{Lip} |\phi_1 - \phi_2|,
\end{eqnarray*}
where the first equality uses the definition of $j(t,\phi,z)$, in Section \ref{nobs}.  It now follows that $|v_n(\phi_1) - v_n(\phi_2)| \leq \left(T + ||v_{n-1}||_{Lip} \right)|\phi_1 - \phi_2|$.
\end{proof}
\begin{lemma}\label{lemma4} The mapping $t \mapsto J v_n(t,\phi)$ is $\frac{1}{2}$-H\"{o}lder continuous.  In particular, $|\frac{d}{dt} J v_n(t,\phi)| \leq \phi + a + b||v_n||_{Lip}t^{-1/2}$, for constants 
\begin{align*}
a & =  \left(1 + \frac{\lambda}{c} + \frac{1}{c}\right) e^{-\lambda t}, \\
b & =  \left( \phi \left(\frac{1}{2} \alpha C_1 + \lambda \right) + \lambda + \frac{1}{2 \lambda} C_1 \right).
\end{align*}
\end{lemma}

\begin{proof} We calculate that 
\[
\begin{split}
\frac{d}{dt} J v_n (t,\phi) 
&= \phi + 1 - e^{-\lambda t}\left(1 + \frac{\lambda}{c} \right) + e^{-\lambda t} \frac{d}{dt} K v_n (t,\phi) - \lambda e^{-\lambda t} K v_n (t,\phi).
\end{split}
\]
Using $|v_n| \leq \frac{1}{c}$, this implies that
\begin{equation}\label{lemma4.eq1}
\left| \frac{d}{dt} J v_n (t,\phi) \right| \leq \phi + e^{-\lambda t} \left(1 + \frac{\lambda}{c} + \frac{1}{c} \right) + \left|e^{-\lambda t} \frac{d}{dt} K v_n(t,\phi) \right|.
\end{equation}
Taking $a = \left(1 + \frac{\lambda}{c}\right)e^{-\lambda t} + \frac{1}{c}$, it therefore suffices to bound the last term on the right hand side above.  So,
\begin{equation}\label{lemma4.eq2}
\begin{split}
\left| e^{-\lambda t} \frac{d}{dt} K v_n (t,\phi) \right|
& = \left| e^{-\lambda t} \frac{d}{dt} \int_{-\infty}^\infty v_n(j(t,\phi,z)) \frac{e^{-z^2/2}}{\sqrt{2 \pi}} dz \right|
\\& = \left| e^{-\lambda t} \int_{-\infty}^\infty v_n'(j(t,\phi,z)) \frac{dj}{dt} \frac{e^{-z^2/2}}{\sqrt{2 \pi}} dz \right|
\\& \leq e^{-\lambda t} ||v_n||_{Lip} \int_{-\infty}^\infty \left| \frac{dj}{dt} \frac{e^{-z^2/2}}{\sqrt{2 \pi}} \right| dz,
\end{split}
\end{equation}
with the exchange of derivatives and integrals in the second equality justified by the fact that $v'_n$ is bounded, established in Lemma \ref{lemma3}.  We now examine more closely the integrand in the last line above.
We calculate that

\begin{flalign}\label{lasteq8}
\frac{dj}{dt} &=\left(\frac{1}{2}\alpha z t^{-1/2} + (\lambda - \alpha^2/2)\right)\exp \left\{\alpha z \sqrt{t} + (\lambda - \alpha^2/2)t \right\}\phi \\
&\ \ \ \ \ + \lambda \exp \left\{\alpha z \sqrt{t} + (\lambda - \alpha^2/2)t \right\} \\
& \ \ \ \ \ + \int_0^t \lambda \left(\frac{-1}{2} \alpha z u t^{-3/2} + \frac{1}{2} \alpha^2 u^2 t^{-2} \right) \exp \left\{ \left(\lambda + \frac{\alpha z}{\sqrt{t}}\right)u - \frac{\alpha^2 u^2}{2t} \right\} du,
\end{flalign}

\noindent and here the first term above came from differentiating the first term of $j$, and the second and third terms came from differentiating the second term of $j$.  We label these terms in \eqref{lasteq8} $(A_1),(A_2),(A_3)$.  Then
\[
\begin{split}
\int_{-\infty}^\infty \left|(A_1)\right| e^{-z^2/2}/\sqrt{2 \pi} 
&= e^{\lambda t} \phi \int_{-\infty}^\infty \left|\frac{1}{2}\alpha z t^{-1/2} + (\lambda - \alpha^2/2) \right| e^{-(z - \alpha \sqrt{t})^2/2}/\sqrt{2\pi} dz
\\& \leq e^{\lambda t} \phi \int_{-\infty}^\infty \left(\frac{1}{2}\alpha t^{-1/2} \left|z  - \alpha \sqrt{t} \right| + \lambda\right) e^{-(z - \alpha \sqrt{t})^2/2}/\sqrt{2\pi} dz
\\&= e^{\lambda t} \phi \left(\frac{1}{2} \alpha t^{-1/2} C_1 + \lambda \right),
\end{split}
\]
with $C_1$ a universal constant equal arising from the expectation of the absolute value of a standard normal r.v.  The second term can be treated similarly, yielding
\[
\int_{-\infty}^\infty |(A_2)| e^{-z^2/2}/\sqrt{2\pi} dz \leq \lambda e^{\lambda t}.
\]
For the third term, we have, using Fubini's Theorem for the first inequality,
\[
\begin{split}
 \int_{-\infty}^\infty |(A_3)| \frac{e^{-z^2}}{\sqrt{2\pi}} \ dz
&=\lambda \int_0^t \int_{-\infty}^\infty \frac{1}{2}\alpha u t^{-3/2} e^{\lambda u}\left| z - \frac{\alpha u}{\sqrt{t}} \right| \frac{e^{-(z- \alpha u/\sqrt{t})^2/2}}{\sqrt{2 \pi}} \ dz du
\\&= \lambda \frac{1}{2} \alpha t^{-3/2} \int_0^t ue^{\lambda u} \int_{-\infty}^\infty \left| z - \frac{\alpha u}{\sqrt{t}} \right| \frac{e^{-(z- \alpha u/\sqrt{t})^2/2}}{\sqrt{2 \pi}} \ dz du
\\&= \lambda \frac{1}{2} \alpha t^{-3/2} C_1 \int_0^t u e^{\lambda u} du
\\&\leq \lambda \frac{1}{2} \alpha t^{-3/2} C_1 \frac{e^{\lambda t}(\lambda t -1) +1}{\lambda^2}.
\end{split}
\]
This has absolute value less than or equal to $\lambda \frac{1}{2} \alpha t^{-3/2} C_1 \frac{e^{\lambda t} t}{\lambda} = \lambda \frac{1}{2} \alpha t^{-1/2} C_1 \frac{e^{\lambda t}}{\lambda}$.  

Plugging these three estimates into \eqref{lemma4.eq2}, we obtain:

\begin{eqnarray}\label{lemma4.eq3}
\lefteqn{\left|e^{-\lambda t} \frac{d}{dt} K v_n(t,\phi) \right|} \nonumber \\
&& \leq e^{-\lambda t} ||v_n||_{Lip} \int_{-\infty}^\infty \left| \frac{dj}{dt} \frac{e^{-z^2/2}}{\sqrt{2 \pi}} \right| dz \nonumber \\
&& \leq e^{-\lambda t} ||v_n||_{Lip} \Bigg( e^{\lambda t} \phi \left(\frac{1}{2} \alpha t^{-1/2} C_1 + \lambda \right) + \lambda e^{\lambda t} + \lambda \frac{1}{2} \alpha t^{-1/2} C_1 \frac{e^{\lambda t}}{\lambda} \Bigg) \nonumber \\
&& = b ||v_n||_{Lip} t^{-1/2},
\end{eqnarray}
with $b = \left( \phi \left(\frac{1}{2} \alpha C_1 + \lambda \right) + \lambda + \frac{1}{2} \alpha C_1 \right).$
\end{proof}

Using the H\"{o}lder continuity established above, we can do a trivial discretization to find $\epsilon$-optimal times.  
%In the rest of this section, we let $w:\mb{R}_+ \ra \left[-\frac{1}{c},0 \right]$ be a concave increasing function.  Note that each value function $v_n, n \geq 0$, satisfies the constraints imposed on $w$.

\begin{corollary}\label{disc_cor1}  Fix $\overline{\phi} > 0 $.  For $0 \leq \phi \leq \overline{\phi}$, one may find $t^*(\phi,\epsilon)$ such that
$J v_n(t^*(\phi,\epsilon),\phi) < \underset{0 \leq t \leq T}{\min} \ J v_n(t,\phi) + \epsilon$ by making $||v_n||_{Lip} \cdot O(\frac{1}{\epsilon^2})$ evaluations of $J v_n(\cdot,\phi)$.
\end{corollary}

\begin{proof}  Discretize $[0,T]$ into $N$ equally spaced points $t_1,\ldots,t_N$, where $N = \left \lceil \frac{M}{\epsilon^2} \right \rceil$, and $M$ is derived from the H\"{o}lder constant established in Lemma \ref{lemma4}; for example, we may take $M \triangleq \overline{\phi} + a + b||v_n||_{Lip}$, with $a = 1 + \frac{\lambda}{c} + \frac{1}{c}$ and $b = \left(\overline{\phi} \left(\frac{1}{2} \alpha C_1 \right) + \lambda + \frac{1}{2 \lambda} C_1 \right)$.  Then, choose $t^*(\phi, \epsilon) \in \underset{1 \leq i \leq N}{\arg \min} \ Jv_n(t_i,\phi)$.  By Lemma \ref{lemma4}, \\
$\left|\underset{1 \leq i \leq n}{ \min} \ J v_n(t_i,\phi) - \underset{0 \leq t \leq T}{\min} \ J v_n(t,\phi) \right| \leq \epsilon$, so $t^*(\epsilon,\phi)$ must be $\epsilon$-optimal.  
\end{proof}

We will uniformly aproximate $v_n$ by a function $\widehat{v}_n$, but we do not know a priori what Lipschitz properties the approximation $\widehat{v}_{n-1}$ has, only that it is close to $v_{n-1}$.  Therefore we need Corollary \ref{disc_cor2} and Lemma \ref{lemma7} to estimate $J_0 \widehat{v}_n(\phi)$.

\begin{lemma}\label{lemma7}  Suppose that $||w_1 - w_2||_{L^\infty} < \epsilon$.  Then $||J_0 w_1 - J_0 w_2||_{L^\infty} < \epsilon$ and $\left| J w_1(t,\phi) - J w_2(t,\phi) \right| < \epsilon$ for all $t, \phi \geq 0$.
\end{lemma}

\begin{proof} The proof follows by noticing that
\[
Jw_1(t,\phi) - Jw_2(t,\phi) = e^{-\lambda t}\int_{-\infty}^\infty \left[w_1(j(t,z,\phi)) - w_2(j(t,z,\phi)) \right] e^{-z^2/2}/\sqrt{2 \pi} dz,
\]
which is bounded in size by $\epsilon$ for all $t$ and $\phi$.
\end{proof}

\begin{corollary}\label{disc_cor2}  Suppose that $w:\mb{R}_+ \ra \mb{R}$ is a function such that $||w - v_n||_{L^\infty} < \epsilon_1$.   Fix $\overline{\phi} > 0 $.  For $0 \leq \phi \leq \overline{\phi}$, one may find $t^{**}(\phi,\epsilon)$ such that
\[
\left| J w(t^{**}(\phi,\epsilon),\phi) - \underset{0 \leq t \leq T}{\min} \ J v_n(t,\phi) \right| < \epsilon + 3\epsilon_1
\]
by making $||v_n||_{Lip} \cdot O(\frac{1}{\epsilon^2})$ evaluations of $J w(\cdot,\phi)$.
\end{corollary}

\begin{proof}  Perform the same discretization as in Corollary \ref{disc_cor2}, and let $t^{**}(\phi, \epsilon) \in \underset{1 \leq i \leq N}{\arg \min} \ J w(t_i,\phi)$.  Since $||w - v_n||_{L^\infty}< \epsilon_1$, Lemma \ref{lemma7} implies that \\ $\left| Jw(t,\phi) - J v_n(t,\phi) \right| < \epsilon_1$ for any $t,\phi \geq 0$.  then Lemma \ref{lemma4} implies that 
\[
\left| J v_n(t^{**}(\phi, \epsilon)) - \underset{0 \leq t \leq T}{\inf} \ J v_n(t,\phi) \right| < \epsilon + 2\epsilon_1.
\]
Using $||w - v_n||_{L^\infty} < \epsilon_1$ again, 
\[\left| J v_n(t^{**}(\phi, \epsilon)) - J w(t^{**}(\phi, \epsilon)) \right| < \epsilon_1.
\]  
Therefore, 
\[
\left| J w(t^{**}(\phi, \epsilon)) - \underset{0 \leq t \leq T}{\inf} \ J v_n(t,\phi) \right| < \epsilon + 3\epsilon_1.
\]

\end{proof}

\subsection{Approximating $J_0 v_n(\phi)$ Over All $\phi$, for Fixed $n$}\label{sec_nobs_phi}

\begin{lemma}\label{lemma6}  Fix $n \geq 0$, and suppose that $w:\mb{R}_+ \ra \mb{R}$ satisfies $||w - v_n||_{L^\infty} < \epsilon_1$.  Then using $||v_n||_{Lip}||v_{n+1}||_{Lip} O \left(\frac{1}{\epsilon^3} \right)$ evaluations of $J w(t,\phi)$, we can construct a function $\widehat{J}_0 w$ such that $||\widehat{J}_0 w(\phi) - J_0 v_n(\phi)||_{L^\infty} < \epsilon + 3\epsilon_1$.

%Recall  $nT \geq ||v_n||_{Lip}$.  Given $v_n$, then using $O(\frac{n^2}{\epsilon^3})$ evaluations of $J v_n(t,\phi)$, we %can construct a function $\widehat{J}_0 v_n(\phi)$ such that $||\widehat{J}_0 v_n(\phi) - J_0 v_n(\phi)||_{L^\infty} \leq %\epsilon$.
\end{lemma}

\begin{proof}  Following Section $4.4$ of \cite{MR2374974}, the function $v_{\mathfrak{c}}(\phi)$ can be explicitly computed, and in fact we can construct $\overline{\phi}$ such that $v_{\mathfrak{c}}(\overline{\phi}) = 0$.   By Proposition \ref{prop4.1}, $v_n(\phi) \geq v_{\mathfrak{c}}(\phi)$, and $v_n$ is increasing and nonnegative, so it follows that $v_n(\phi) = 0$ for $\phi \geq \overline{\phi}$ for all $n \geq 0$.  Therefore, we set $\widehat{J}_0 w(\phi) = 0$ for $\phi \geq \overline{\phi}$.  So from now on, we assume that $0 \leq \phi \leq \overline{\phi}$.  Discretize $[0,\overline{\phi}]$ into $R = \left \lceil \frac{||v_{n+1}||_{Lip}}{\epsilon} \right \rceil$ points $\phi_1,\ldots,\phi_R$, with $\phi_1 =0$ and $\phi_R = \overline{\phi}$.  Using Corollary \ref{disc_cor2}, for each $i$, we may, given $w$, in $||v_n||_{Lip} O \left(\frac{1}{\epsilon^2} \right)$ function evaluations calculate $\widehat{J}_0 w(\phi_i)$ such that $|\widehat{J}_0 w(\phi_i) - J_0 v_n(\phi_i)| < \epsilon + 3\epsilon_1$ for $1 \leq i \leq R$.  For $\phi_i \leq \phi < \phi_{i+1}$, $1 \leq i \leq R-1$, set $\widehat{J}_0 w(\phi) = \widehat{J}_0 w(\phi_i)$.  We have, for $\phi_i \leq \phi < \phi_{i+1}$,
\[
\begin{split}
|\widehat{J}_0 w(\phi) - J_0 v_n(\phi)|
& = |\widehat{J}_0 w(\phi_i) - J_0 v_n(\phi)|
\\& \leq |\widehat{J}_0 w(\phi_i) - J_0 v_n(\phi_i)| + |J_0 v_n(\phi_i) - J_0 v_n(\phi)|
\\& \leq \left(\epsilon + 3\epsilon_1 \right) + \epsilon,
\end{split}
\]
where the second $\epsilon$ term above is derived from the Lipschitzness of $J_0 v_n = v_{n+1}$, established in Lemma \ref{lemma3}.  Since each point $i$ requires $||v_n||_{Lip} O \left( \frac{1}{\epsilon^2} \right)$ function evaluations, computing the approximations for all $R \approx \frac{||v_{n+1}||_{Lip}}{\epsilon}$ points requires $||v_n||_{Lip}||v_{n+1}||_{Lip} O \left( \frac{1}{\epsilon^3} \right)$ function evaluations.
\end{proof}

\subsection{Approximating $v_n(\phi)$, for all $0 \leq n \leq N$}\label{sec_nobs_alln}

\begin{proof}[Proof of Proposition \ref{prop2}]  The function $v_0(\phi)$ may be computed analytically.  According to Lemma \ref{lemma6}, we may compute a function $\widehat{v_1}(\phi)$ such that
\[
||\widehat{v}_1(\phi) - v_1(\phi)||_{L^\infty} < \frac{\epsilon}{N}
\]
in $||v_0||_{Lip}||v_1||_{Lip} O \left(\frac{N^3}{\epsilon^3} \right)$ function evaluations.  Applying Lemma \ref{lemma6} again, we construct $\widehat{v}_2(\phi)$ satisfying
\[
\left|\left| \widehat{v}_2(\phi) - \underbrace{J_0 v_1(\phi)}_{v_2(\phi)} \right|\right|_{L^\infty} < \frac{\epsilon}{N} + \frac{\epsilon}{N} 
\]
in $||v_1||_{Lip}||v_2||_{Lip} O \left( \frac{N^3}{\epsilon^3} \right)$ function evaluations.  Arguing inductively in this way, we see that we may compute $\widehat{v}_N(\phi)$ satisfying
\[
||\widehat{v}_N(\phi) - v_N(\phi)||_{L^\infty} < \frac{N \epsilon}{N} = \epsilon
\]

in 
\[
O \left( \frac{N^3}{\epsilon^3} \right) \sum_{i=0}^{N-1} ||v_i||_{Lip}||v_{i+1}||_{Lip} 
\]
function evaluations.  By Lemma \ref{lemma3}, $||v_i||_{Lip} \leq iT$.  Therefore

\[
\begin{split}
O \left( \frac{N^3}{\epsilon^3} \right) \sum_{i=0}^{N-1} ||v_i||_{Lip}||v_{i+1}||_{Lip}
& = O \left( \frac{N^3}{\epsilon^3} \right) \sum_{i=0}^{N-1} i^2T
\\& \leq O \left( \frac{N^3}{\epsilon^3} \right) N^3 T
\\& = O \left( \frac{N^6}{\epsilon^3} \right).
\end{split}
\]

\end{proof}

\section{\uppercase{Numerical Results for the Lump Sum $n$-Observation Problem}}\label{nalg_res}

\subsection{Comparison to the Continuous Value Function}

\begin{center}
\begin{figure}[h!]
\centering
\includegraphics[scale=.45]{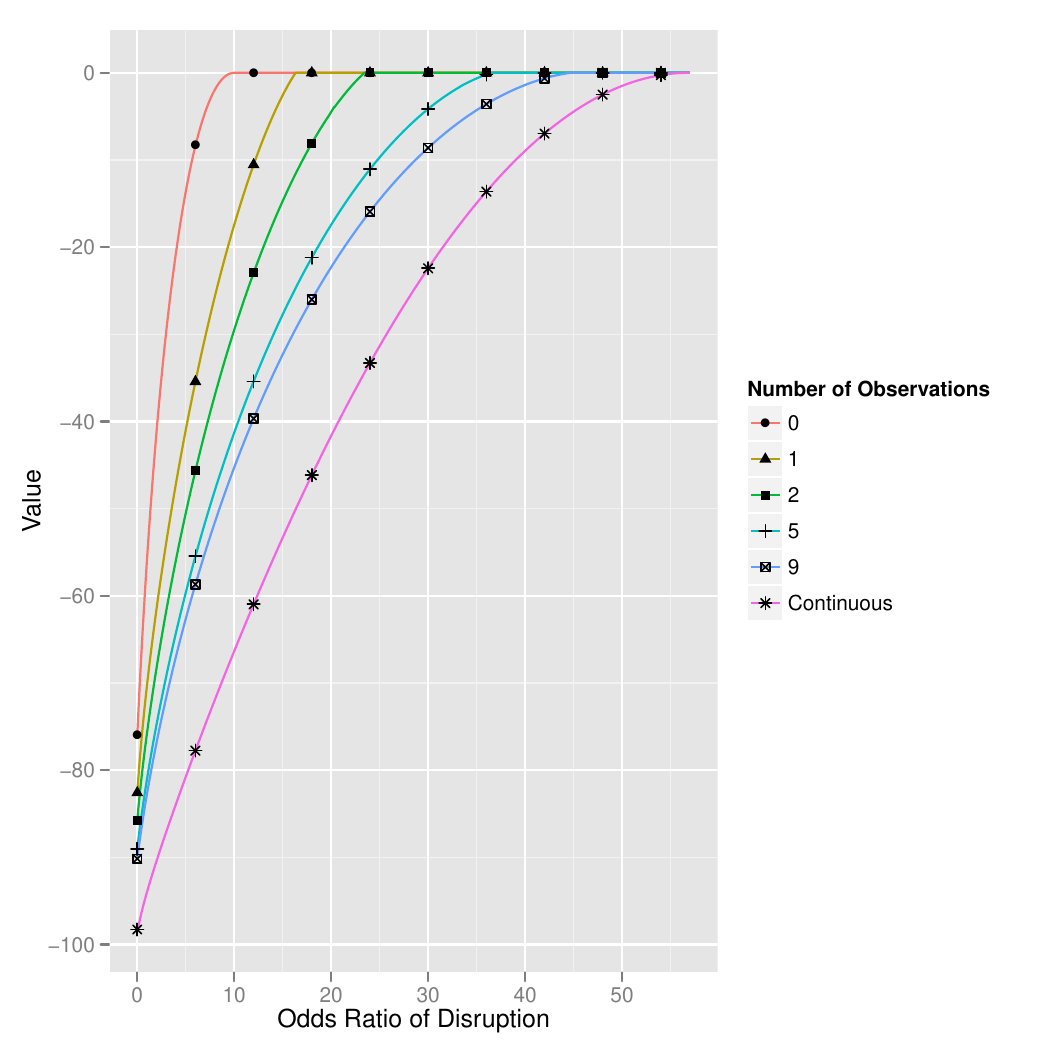}
\caption{Value function $v_n(\phi)$ for $0 \leq n \leq 9$, and continuous observation value function $v_{\mathfrak{c}}(\phi)$, $\lambda=.1$, $c=.01$, $\alpha=1$.}
\label{obs_valfun}
\end{figure}
\end{center}

\begin{center}
\begin{figure}[h!]
\centering
\includegraphics[scale=.45]{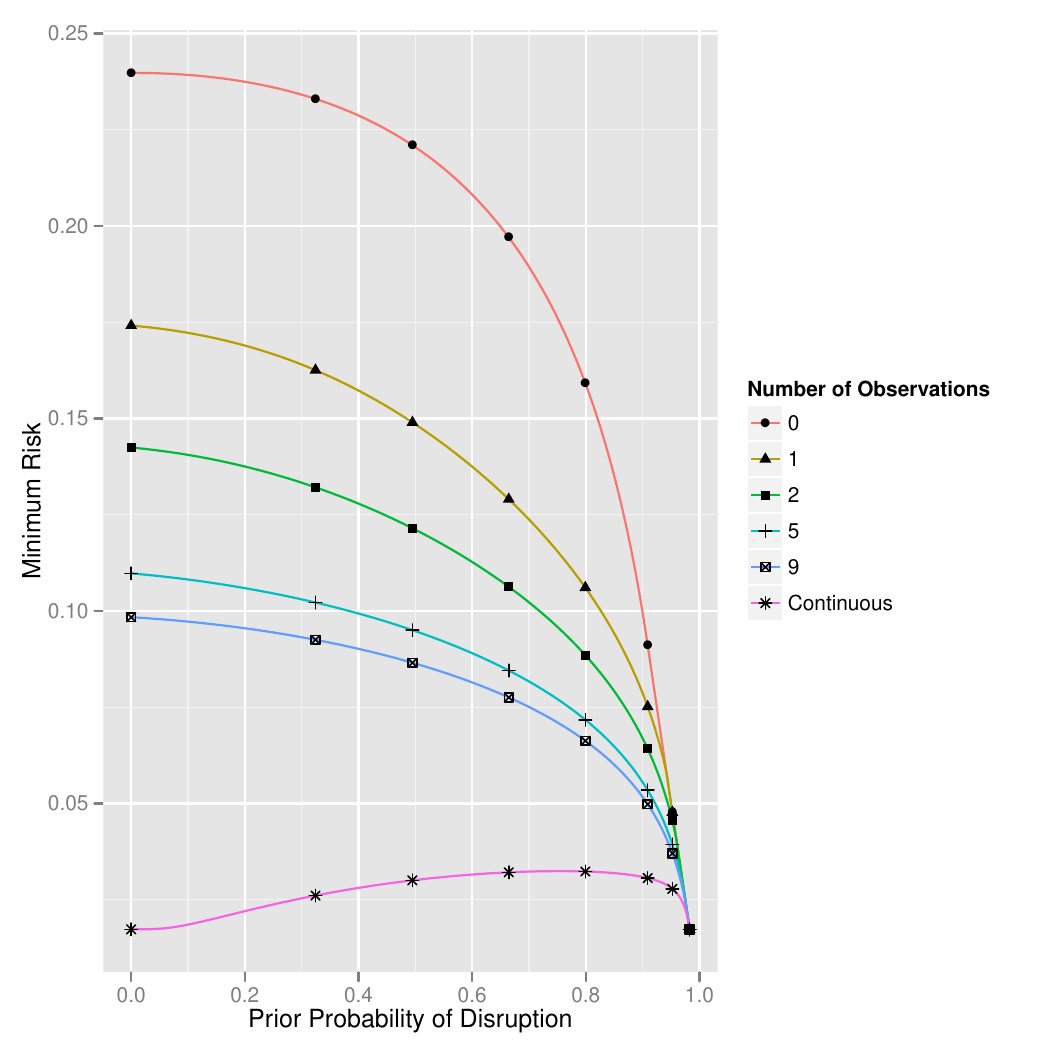}
\caption{Bayesian Risk Associated with Total Number of Observations, $\lambda = .1$, $c = .01$, $\alpha=1$.}
\label{obs_bayrisk}
\end{figure}
\end{center}

\begin{table}

\begin{tabular}{|c || c | c| c | c|}
  \hline
  \textbf{Observations (n) } & $v_n(0)$ & $v_{\mathfrak{c}}(0)$ & $v_n(0) - v_{\mathfrak{c}}(0)$ & $\frac{\log (v_n(0) - v_{\mathfrak{c}}(0))}{\log (n+1)} - \log(v_0(0) - v_{\mathfrak{c}}(0)), n \geq 1$ \\ \hline
  \textbf{0} & -76.021 & -98.237 & 22.216 & \\ \hline
  \textbf{1} & -82.586 & -98.237 & 15.651 & -.505 \\ \hline
  \textbf{2} & -85.755 & -98.237 & 12.482 & -.525 \\ \hline
  \textbf{3} & -87.410 & -98.237 & 10.827 & -.518 \\ \hline
  \textbf{4} & -88.392 & -98.237 & 9.845 & -.506 \\ \hline
  \textbf{5} & -89.024 & -98.237 & 9.213 & -.491 \\ \hline
  \textbf{6} & -89.455 & -98.237 & 8.782 & -.477 \\ \hline
  \textbf{7} & -89.762 & -98.237 & 8.475 & -.463 \\ \hline
  \textbf{8} & -89.990 & -98.237 & 8.247 & -.451 \\ \hline
  \textbf{9} & -90.163 & -98.237 & 8.074 & -.440 \\ \hline
  \textbf{10} & -90.299 & -98.237 & 7.938 & -.429 \\ \hline
\end{tabular}
\caption{Effect of Observation Size on Value Functions at $\phi = 0$}
\label{tab1}
\end{table}
As expected, the value functions $v_n(\phi)$ are all concave and increasing, and between $\frac{-1}{c} = 100$ and $0$.  Furthermore, as $n$ increases, the value functions decrease.  From Figure \ref{obs_valfun}, it is not immediately obvious whether $\underset{n \ra \infty}{\lim} v_n = v_{\mathfrak{c}}$, although the results of Section \ref{nobs_cont} prove that this is the case.  In any case, the convergence rate is quite slow, as demonstrated by Table \ref{tab1}.
 Using the results in \cite{higham2011mean}, we would expect that, as $n \ra \infty$,
\[
|v_n^{\mathfrak{D}} - v_{\mathfrak{c}}| = O \left(n^{-\frac{1}{2}}\right).
\]
The results of \cite{higham2011mean} apply to diffusions which are uniformly elliptical.  The diffusion $\phi^{\mathfrak{c}}$, defined in \eqref{contobsphi} has vanishing volatility at zero, but its behavior at the stopping threshold is elliptic, and this fact, combined with its upward drift, should allow one to extend the results of \cite{higham2011mean} to $\phi^{\mathfrak{c}}$.  Also, the numerical scheme introducted in \cite{BF11}, as a combination of Monte Carlo and a finite difference scheme for solving obstacle problems, is quite relevant.  The convergence rate proof here could be adapted to give the rate of convergence of the value functions in our setup.  The problem here would be to obtain similar results with the non-degeneracy assumption.  
%
%We note that these techniques allow for establishing the same convergence rate of value functions in  \cite{MR2777513} and our own adaptive observation model.  It is likely the case that these two methods have the same convergence rate, but with different constants.  A more sophisticated analysis will be required to uncover these differences.  In this direction, the theory of adaptive discretization of SDE's may be useful (see \cite{hofmann2000optimal}, \cite{hofmann2001optimal}), since we may interpret an observation strategy as a kind of discretization of the diffusion $\phi^{\mathfrak{c}}$.  

%Additionally, it is relevant to study the differences between deterministic observation and adaptive observations without letting $n \ra \infty$.  The efficiency difference between these two observation strategies is most crucial precisely when there is a relatively small amount of observation rights.  For numerical results in this direction, see Figure \ref{vfun_comp2}.
%

\subsection{Comparison to  \cite{MR2777513}'s Discrete Observation Model}

In this subsection, we compare the value functions of the lump sum $n$-observation problem with those found in Dayanik's model of discrete observation, \cite{MR2777513}.  More precisely, we consider models of one or five total observation rights, and specify fixed time intervals at which observations will be made.  In Figures \ref{vfun_comp1} and \ref{vfun_comp2}, it is not surprising that the value function from our $n$ observation problem is smallest, but the efficiency gap can be quite large, especially for higher values of $\phi$ when it can be crucial to make an observation quickly.  Furthermore, we can see that the value functions associated to fixed observation schedules have widely varying performance on different levels of $\phi$, and one which performs well for one value of $\phi$ may do quite poorly at another.  Therefore, it is hard to achieve good performance using fixed observation strategies.  This should not be surprising: our value function is the concave hull of the value functions corresponding to deterministic observation schedules.  The difference is magnified with more observations, as flexibility becomes more important.

\begin{center}
\begin{figure}[h!]
\centering
\includegraphics[scale=.50]{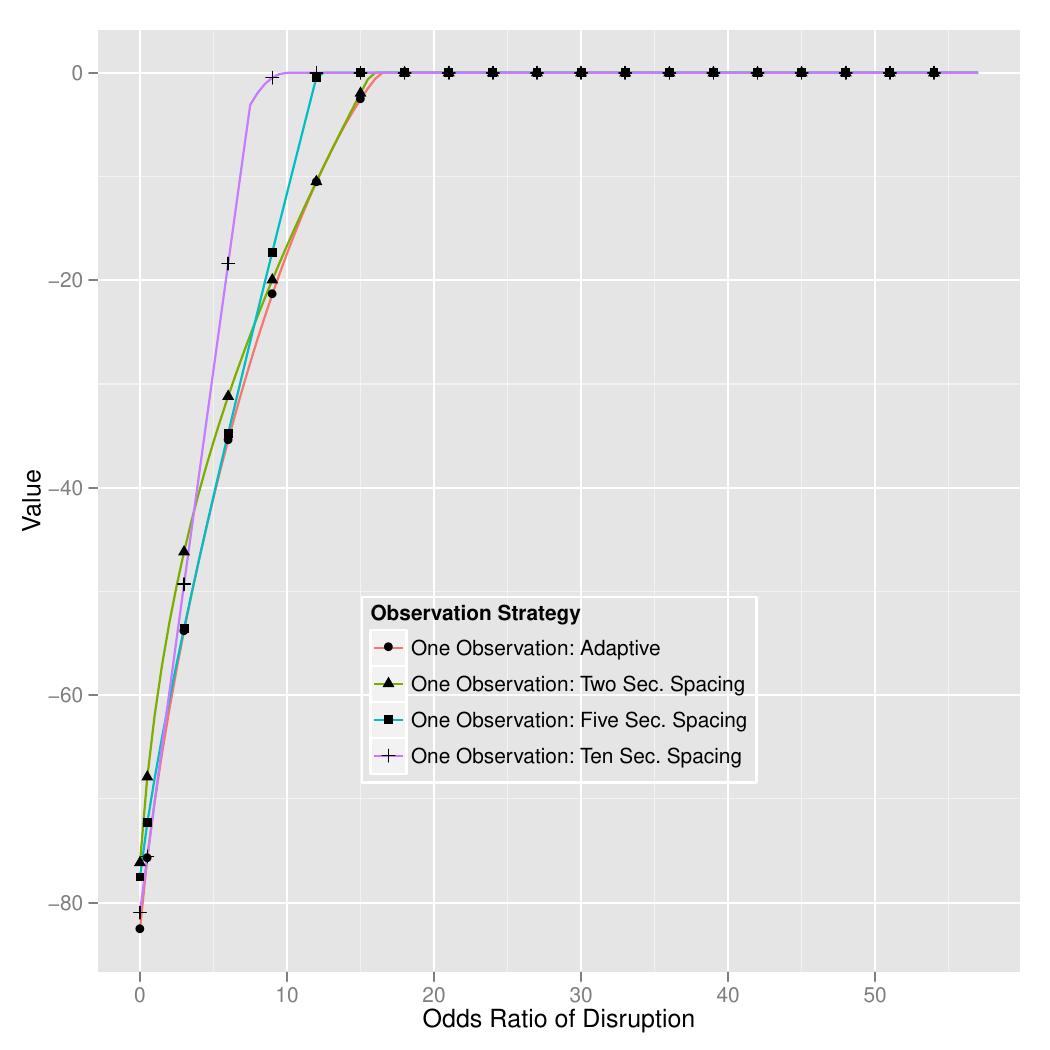}
\caption{Value Functions for One Obsevation: At Two, Five, and Ten Seconds, and Chosen Adaptively, $\lambda = .1$, $c = .01$, $\alpha=1$.}
\label{vfun_comp1}
\end{figure}
\end{center}

\begin{center}
\begin{figure}[h!]
\centering
\includegraphics[scale=.50]{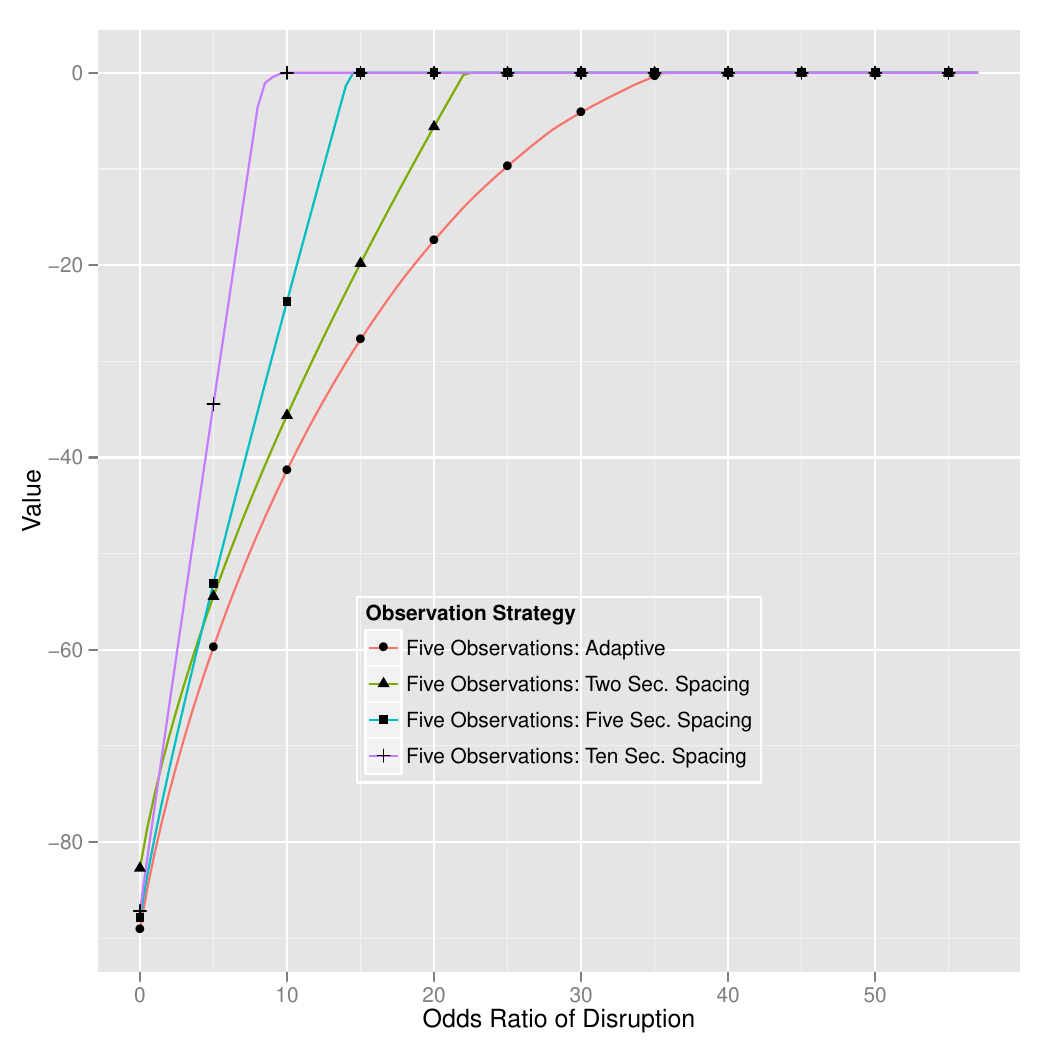}
\caption{Value Functions for Five Obsevations: At Two, Five, and Ten Seconds Intervals, and Chosen Adaptively, $\lambda = .1$, $c = .01$, $\alpha=1$.}
\label{vfun_comp2}
\end{figure}
\end{center}

\subsection{Depiction of Observation Boundaries}

\begin{center}
\begin{figure}[h!]
\centering
\includegraphics[scale=.45]{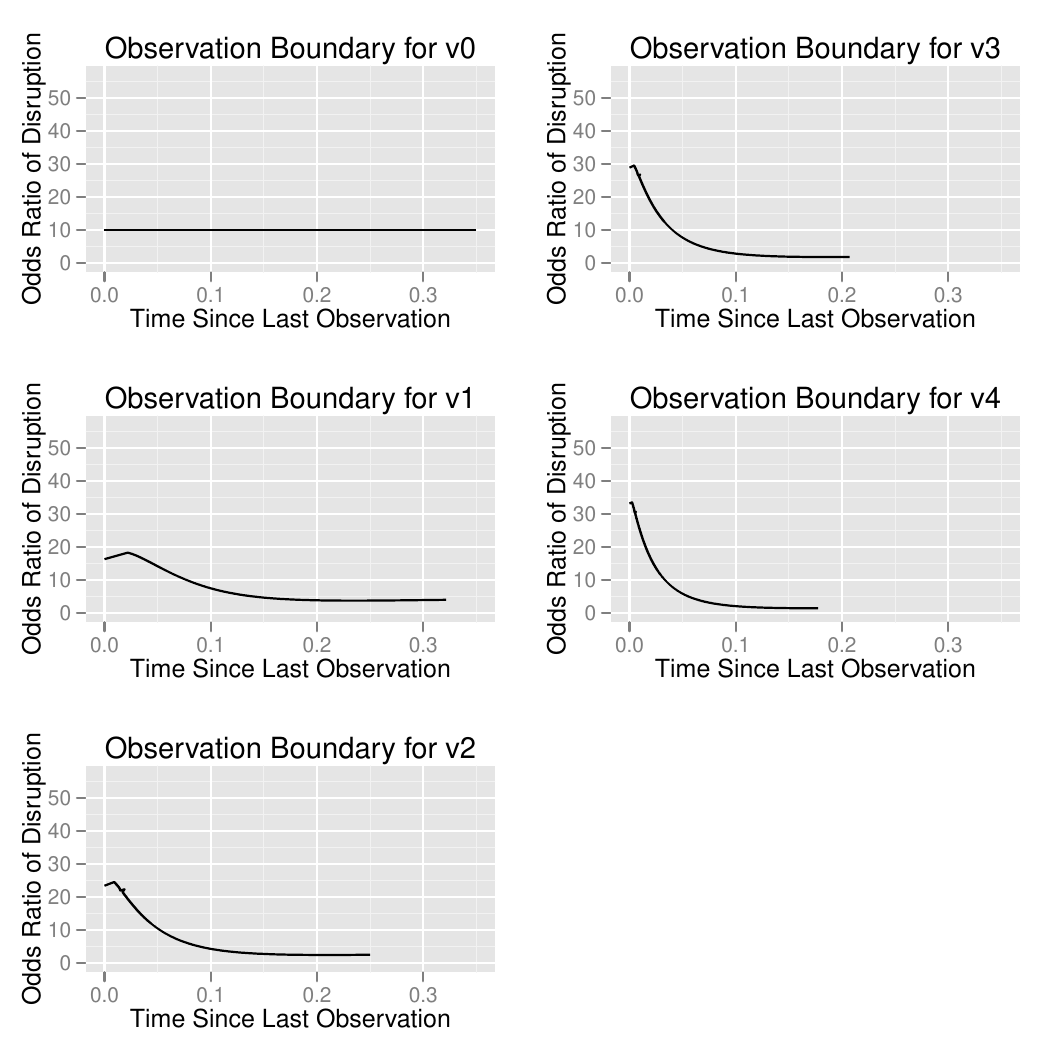}
\caption{Observation Boundaries, $\lambda = .1$, $c = .01$, $\alpha=1$}
\label{obs_bdry1}
\end{figure}
\end{center}

\begin{center}
\begin{figure}[h!]
\centering
\includegraphics[scale=.45]{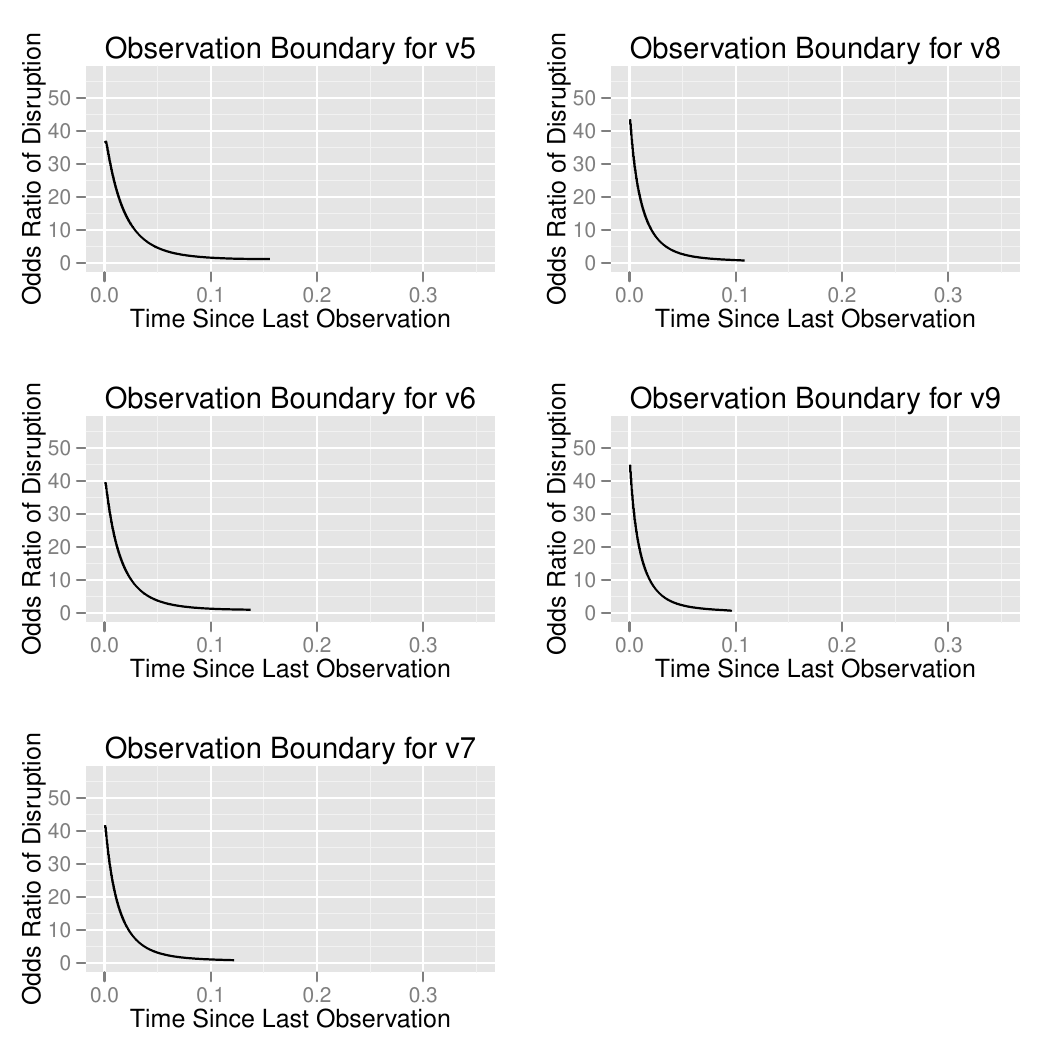}
\caption{Observation Boundaries, $\lambda = .1$, $c = .01$, $\alpha=1$.}
\label{obs_bdry2}
\end{figure}
\end{center}

The observation boundary for $v_0$ is of course identically equal to a horizontal line at $10 = \frac{\lambda}{c}$.  Without any observations, the posterior process is perpetually increasing, and so the observation boundary (which is really a stopping boundary here) should always stop when $\phi$ is equal to $\frac{\lambda}{c}$.  Note furthermore that the boundary does not depend on the time since the last observation: since there will never be any more observations made, there is time homogeneity.

In the rest of the observation boundaries, we notice two general trends: first, the curves are decreasing in $n$ for large values of time, and second, they are increasing in $n$ for small values of time.  The former phenomenon reflects the idea that if an agent has more observations, than he should be more willing to use them, which corresponds to the barrier being easier to get to, and hence lower.  At small time values, however, the barriers are increasing.  This reflects the fact that as one has more observations, one is less willing to ``give up" and stop.  For example, when one has only a single observation and the odds process is above $20$, an optimally acting agent knows that he only hurts himself by waiting, and so will observe immediately at time zero (which is equivalent to stopping the game).  With more observations, however, the agent is willing to wait a little bit and see how things will go, and for this reason, the curves increase at small times.

A natural question is whether, as the observations increase, do the curves tend to infinity for very small values of time?  In fact, the observation boundaries are uniformly bounded for all $n$ and $t$.  One may deduce this fact by comparing the discrete observation value functions with the continuous observation value function.  In Figure \ref{obs_valfun}, one sees that the continuous value function is zero for $\phi \geq \overline{\phi}$ (here $\overline{\phi} \approx 55$, and $\overline{\phi}$ can be explicitly computed: it is the optimal threshhold level for stopping in the continuous observation problem, ).  This implies that every discrete observation value function is also zero for $\phi \geq \overline{\phi}$.  Therefore, one always wants to observe immediately at such $\phi$ values.  It follows then that all observation boundaries from Figures \ref{obs_bdry1} and \ref{obs_bdry2} will be bounded from above by $\overline{\phi}$.  

\section{\uppercase{Numerics for the Stochastic Arrival Rate Problem}}\label{stalg}

\subsection{A Heuristic Algorithm for the Stochastic Arrival Rate $n$-Observation Problem}

Here we outline a computational algorithm for solving the stochastic arrival rate problem.  The infinite horizon problem is a limiting case of this one.  As discussed before, the lump sum $n$-observation problem is essentially embedded into this problem, so it should be of no surprise that this problem must be solved first.  

\begin{itemize}

\item[(1)]  Fix the total number of observations $N$.  Discretize $\phi$ into $\phi_0 = 0,\phi_1,\ldots,\phi_J = \overline{\phi}$ as before, and set all value functions equal to zero for $\phi \geq \overline{\phi}$.  Compute the (approximations to) value functions $\widehat{\bm{v}}^N_{N,j}(0,\phi), 0 \leq j \leq N$, as in the previous section, as well as the optimal times $\widehat{t}^{*,N}_{N,j}(\phi)$.

\item[(2)]  We have computed $\widehat{\bm{v}}^N_{N,j}(0, \phi)$, the approximation to $\bm{v}^N_{N,j}(0,\phi)$.  For $y,t \geq 0$ such that $(y+t,\varphi(t,\phi)) \in \Lambda^F$, define $\widehat{\bm{v}}^N_{N,j}(y+t,\varphi(t,\phi))$ by

{\small{
\begin{multline}\label{stalg_step2}
e^{-\lambda(y+t)} \widehat{\bm{v}}^N_{N,j}(y+t,\varphi(t,\phi)) = \\
\left\{
		\begin{array}{ll}
		\widehat{\bm{v}}^N_{N,j}(0,\varphi(-y,\phi)) - \int_0^{y+t} e^{-\lambda r} \left(\varphi(r-y,\phi) - \frac{\lambda}{c} \right) dr & \mbox{if } y+t < \widehat{t}^{*,N}_{N,j}(\varphi(-y,\phi)) \\
		e^{-\lambda(y+t)} \bm{K} \widehat{\bm{v}}^N_{N,j+1}(y+t,\varphi(-y,\phi)) & \mbox{if } y+t \geq \widehat{t}^{*,N}_{N,j}(\varphi(-y,\phi))
		\end{array}
\right.
\end{multline}
}}

\item[(3)]  Fix $\phi_j$.  Discretize time into $t_0=0,t_1,\ldots,t_K = T$, for an appropriately chosen upper bound $T$, as in Lemma \ref{lemma2}.  Compute $\widehat{\bm{v}}^N_{N-1,N-1}(0,\phi_j)$ by minimizing $J^0 \widehat{\bm{v}}^N_{N,N-1}(t_i,0,\phi_j)$ over the $t_i$.  Let $\widehat{t}^{*,N}_{N-1,N-1}(\phi_j)$ be the minimizing $t_i$.  

\item[(4)]  Interpolate to find a function $\widehat{\bm{v}}^N_{N-1,N-1}(0,\phi)$ which approximates $\bm{v}^N_{N-1,N-1}(0,\phi)$, and a stopping boundary $\widehat{t}^{*,N}_{N-1,N-1}(\phi)$.

\item[(5)]  Fix $\phi_j$.  As in Step $3$, compute $\widehat{\bm{v}}^N_{N-1,N-2}(0,\phi_j)$ by minimizing \\ $J^+(\widehat{\bm{v}}^N_{N-1,N-1},\widehat{\bm{v}}^N_{N,N-2})(t_i,0,\phi_j)$ over the $t_i$.  Let $\widehat{t}^{*,N}_{N-1,N-2}(\phi_j)$ be the minimizing $t_i$.

\item[(6)]  Interpolate to find a function $\widehat{\bm{v}}^N_{N-1,N-2}(0,\phi)$ which approximates $\bm{v}^N_{N-1,N-2}(0, \phi)$, and an observation boundary $\widehat{t}^{*,N}_{N-1,N-2}(\phi)$.

\item[(7)] Repeat Steps $5$ and $6$ to compute $\widehat{\bm{v}}^N_{N-1,j}(0,\phi)$ and $\widehat{t}^{*,N}_{N-1,j}(\phi)$ for $0 \leq j \leq N-2$.

\item[(8)] We now need to repeat the analog of Step $2$.  For $y,t \geq 0$ such that $(y+t,\varphi(t,\phi)) \in \Lambda^F$, inductively define $\widehat{\bm{v}}^N_{N-1,j}(y+t,\varphi(t,\phi)$, $0 \leq j \leq N-2$, by
\[
e^{-\lambda(y+t)} e^{-\mu (y+t)}\widehat{\bm{v}}^N_{N-1,j}(y+t,\varphi(t,\phi)) =
\]
{\scriptsize
\[
\left\{
		\begin{array}{ll}
		\widehat{\bm{v}}^N_{N-1,j}(0,\varphi(-y,\phi)) & \mbox{if } y+t < \widehat{t}^{*,N}_{N-1,j}(\varphi(-y,\phi))\\
		\ \ \ \ \ \  - \int_0^\infty \mu e^{-\mu u} \Bigg(\int_0^{u \wedge (y+t)} e^{-\lambda r}\left( \varphi(r-y,\phi) - \frac{\lambda}{c} \right) dr \\ \ \ \ \ \ \ \ - 1_{\{u \leq y + t\}}e^{-\lambda u} \widehat{\bm{v}}^N_{N-1,j+1}(y+u,\varphi(u,\phi)) \Bigg) du\\
		e^{-\lambda(y+t)} e^{-\mu (y+t)} \bm{K} \widehat{\bm{v}}^N_{N-1,j+1}(y+t,\varphi(-y,\phi)) & \mbox{if } y+t \geq \widehat{t}^{*,N}_{N-1,j}(\varphi(-y,\phi))
		\end{array}
\right.		
\]
}
\item[(9)] Repeat steps $3$ through $8$ for each $0 \leq n \leq N-2$, computing $\widehat{v}^N_{n,j}$, $0 \leq j \leq n$ and their associated optimal times $\widehat{t}^{*,n}_{n,j}(\phi)$.
\end{itemize}

\subsection{Discussion of the Heuristic}

\begin{itemize}
\item[1.]  The formula in Step $(2)$ comes from a dynamic programming principle.  In a simplified version (with $y=0$), the dynamic programming principle says that for $t \leq \widehat{t}^{*,N}_{N,j}(\phi)$,
\[
\bm{v}^N_{N,j}(0,\phi) = \int_0^t e^{-\lambda r} \left(\varphi(r,\phi) - \frac{\lambda}{c} \right) dr + e^{-\lambda t} \bm{v}^N_{N,j}(t,\varphi(t,\phi)):
\]
in other words, if it is not optimal to make an observation before time $t$, then by waiting until time $t$ no utility is lost, and the only difference between the value functions at the two times is the exponential discounting and the running cost lost between them.  On the other hand, for all $t > \widehat{t}^{*,N}_{N,j}(\phi)$, we assume that it is optimal to immediately observe, hence the term $e^{-\lambda t} \bm{K} \bm{v}^N_{N,j+1}(t,\phi)$ in \eqref{stalg_step2}, describing the expected value after an observation is made.  This step is a heuristic because we have not shown that the optimal observation behavior has this simple strategy.  It is in theory possible, although unlikely in practice, that, when starting at $t=0$ the optimal observation time is $\widehat{t}^{*,N}_{N,j}(\phi)$, but when starting at some $t_1 > \widehat{t}^{*,N}_{N,j}(\phi)$, the optimal observation time is not $t_1$, but some other $t_2 > t_1$.  Numerical evidence suggests that this is not the case, but we do not have a proof of this fact.

\item[2.]  In Step $3$, The $J_0^0$ operator is applied in the case when the agent has no spare observation rights.  If he must wait to receive an observation, then even in that first instant when he receives this right, a positive amount of time has passed since the last observation was made.  Therefore, we need information about the value function, i.e. $\bm{v}^N_{N,N-1}(\cdot,\cdot)$, when its first argument is positive: this explains the necessity of Step $2$.  Similar considerations apply to the calculation in Step $5$.  

\item[3.] The derivation of Step $8$ is similar to that of Step $2$, except that whereas in Step $2$, we took the dynamic programming principle with all observation rights received, here we use the dynamic programming principle when there are still observation rights receiving.  For example, dynamic programming implies that for $t \leq t^{*,N}_{N-1,j}(\phi)$,
\begin{multline*}
\bm{v}^N_{N-1,j}(0,\phi) = \int_0^\infty \mu e^{-\mu} \Bigg(\int_0^{u \wedge t} e^{-\lambda r} \left(\varphi(r,\phi) - \frac{\lambda}{c} \right) dr
\\ + 1_{\{t < u \}} e^{-\lambda t} \bm{v}^N_{N-1, j}(t,\varphi(t,\phi)) + 1_{\{u \leq t\}} e^{-\lambda u} \bm{v}^N_{N,j}(u, \varphi(u,\phi)) \Bigg) du.
\end{multline*}
From this equation, we can solve for $\bm{v}^N_{N-1,j}(t,\varphi(t,\phi))$ to obtain the formula used in Step $8$.

\item[4.] In Step $3$, the minimum of a function is calculated by an exhaustive search on grid points.  Numerical evidence suggests that $J^0 \widehat{\bm{v}}^N_{N,N-1}(t,0,\phi)$ is actually convex as a function of $t$, which would allow for much more efficient ways of finding its minimum.  The same holds true for the minimization in Step $5$.  We currently do not have analytic proofs of these facts.  We note that this reasoning can additionally be applied to the minimization of $Jv_n(\cdot,\phi)$ in the lump sum $n$-observation problem.
\end{itemize}

\subsection{Numerical Results for the Stochastic Arrival Rate $n$-Observation Problem}

\begin{center}
\begin{figure}[h!]
\centering
\includegraphics[scale=.45]{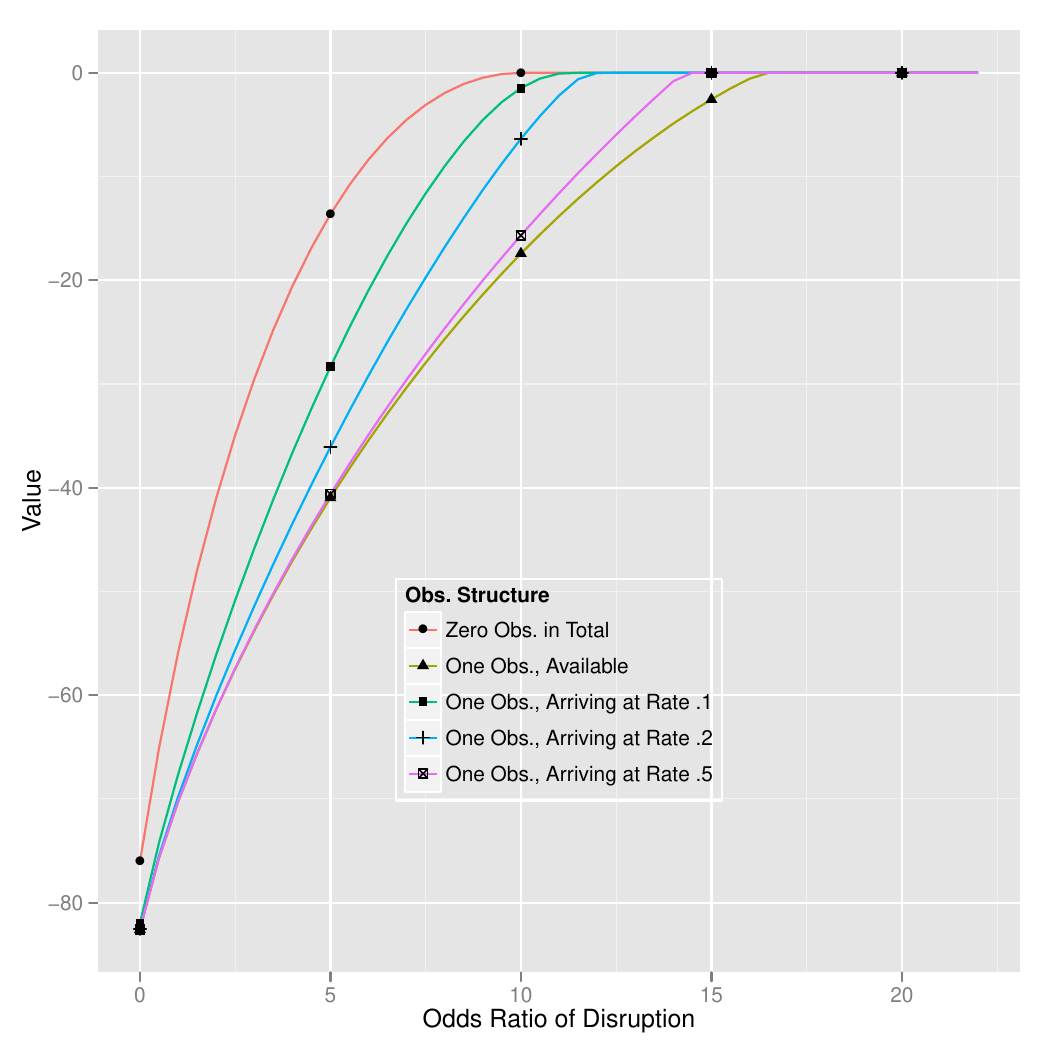}
\caption{$\bm{v}^1_{1,1}(\phi), \bm{v}^1_{1,0}(\phi)$, and $\bm{v}^1_{0,0}(\phi)$ for different arrival rates $\mu$; $\lambda = .1$, $c = .01$, $\alpha=1$.}
\label{st_valfuns}
\end{figure}
\end{center}

Concerning Figure \ref{st_valfuns}, we make a few basic obsevations.  First, it should be clear that $v^1_{1,1}$, corresponding to the case where the single observation right has been used, is the worst-performing value function, and $v^1_{1,0}$, corresponding to the case where the observation right has been received but not used, is the best-performing.  We also expect $v^1_{0,0}$, as the arrival parameter $\mu$ varies, to interpolate between these two extreme curves, so that $v^1_{0,0}$ resembles $v^1_{1,0}$ when $\mu$ is large, and $v^1_{1,1}$ when $\mu$ is small.  Furthermore, the gap between $v^1_{0,0}$ and $v^1_{1,0}$ is smallest when $\phi$ is small.  This reflects the fact that when $\phi$ is small, an optimally acting agent, even if he had an observation right in hand, would wait to exercise it.  As a consequence, having to wait to receive such a right is a less stringent constraint than when $\phi$ is large, in which case it is important to make an observation relatively quickly. 

\section{\uppercase{Proofs from Section 2}}\label{nobs_proof}

\begin{proof}[Proof of Proposition \ref{nobt}]  Let $\Psi = \{\psi_1,\ldots,\psi_n\} \in \mathfrak{O}^n$, and let $\tau \in \mc{T}^\Psi$.  Supposing that $\tau \geq \psi_k$, we will show how to modify $\Psi$ to yield an observation strategy $\widetilde{\Psi}$ under which $\tau$ is a $\mb{F}^{\widetilde{\Psi}}$-stopping time that does not stop between $\psi_k$ and $\psi_{k+1}$.  By inductively following the same procedure, this allows us to construct an observation strategy $\Psi'$ such that $\tau \in \mc{T}^{\Psi'}_o$ and $E^\phi \left[ \int_0^\tau e^{-\lambda t} \left( \Phi^\Psi_t - \frac{\lambda}{c} \right) dt \right] = E^\phi \left[ \int_0^\tau e^{-\lambda t} \left( \Phi^{\Psi'}_t - \frac{\lambda}{c} \right) dt \right]$.  This will establish the lemma.

The basic method is to just add in an observation whenever $\tau$ stops between observations; if a stop is made between observations, there are always ``spare observations".  Let $A = \{\psi_k < \tau < \psi_{k+1}\}$.  According to Proposition \ref{prop1}, $A \in \mc{F}^\Psi_{\psi_k}$, or $A \in \sigma \left(\left\{X_{\psi_i},\psi_i \right\}_{1 \leq i \leq k} \right)$.  Define $\widetilde{\psi}_1 = \psi_1,\ldots,\widetilde{\psi}_k = \psi_k$, $\widetilde{\psi}_{k+1} = 1_A \tau + 1_{A^c} \psi_{k+1}$, and for $k+2 \leq j \leq n$, $\widetilde{\psi}_j = 1_A \psi_{j-1} + 1_{A^c} \psi_{j}$.  With $A \in \sigma \left(\left\{X_{\psi_i},\psi_i \right\}_{1 \leq i \leq k} \right) = \sigma \left(\left\{X_{\widetilde{\psi}_i},\widetilde{\psi}_i \right\}_{1 \leq i \leq k} \right)$, the following claim will imply that for each $k+1 \leq j \leq n$, 
\[
\widetilde{\psi}_j \in m \ \sigma(X_{\widetilde{\psi}_1},\ldots,X_{\widetilde{\psi}_{j-1}},\widetilde{\psi}_1,\ldots,\widetilde{\psi}_{j-1}).
\]
This fact is obvious for $j=1,\ldots,k$, as $\widetilde{\psi}_j = \psi_j$ for $j \leq k$.  Set $\widetilde{\Psi} \triangleq \{\widetilde{\psi}_1,\ldots,\widetilde{\psi}_n\}$.

\begin{claim}\label{claim1} Let Let $k+1 \leq j \leq n$.  Let $X \in \mc{F}^{\Psi}_{\psi_{j-1}}$ and let $Y \in \mc{F}^{\Psi}_{\psi_j}$.  Then $[A \cap X] \cup [A^c \cap Y] \in \mc{F}^{\widetilde{\Psi}}_{\widetilde{\psi}_j}$.
\end{claim}

\begin{proof}[Proof Of Claim]

Write $1_X = x \left(\{X_{\psi_i},\psi_i : 1 \leq i \leq j-1 \right \})$, \\ $1_Y = y \left(\{X_{\psi_i},\psi_i : 1 \leq i \leq j \} \right)$, $1_A = a\left(\left\{X_{\psi_i},\psi_i \right\}_{1 \leq i \leq k} \right)$, where $x$, $y$, and $a$ are all Borel functions with respective domains $\mb{R}^{2(j-1)},\mb{R}^{2j}$, and $\mb{R}^{2k}$.

Then 
\begin{eqnarray*}
\lefteqn{1_A 1_X + 1_{A^c} 1_Y} \\ 
&& = 1_A x \left( \{X_{\psi_i},\psi_i : 1 \leq i \leq j-1 \} \right) \\
&& \ \ \ \ \ + 1_{A^c} y \left( \{X_{\psi_i},\psi_i : 1 \leq i \leq j \} \right) \\
&& = 1_A x \left( X_{\psi_1},\psi_1, \{1_A X_{\psi_{i-1}} + 1_{A^c} X_{\psi_i}, 1_A \psi_{i-1} + 1_{A^c} \psi_i  : 3 \leq i \leq j \} \right) \\
&& \ \ \ \ \ + 1_{A^c} y \Big( X_{\psi_1},\psi_1, 1_A X_\tau + 1_{A^c} X_{\psi_2}, 1_A \tau + 1_{A^c} \psi_2,  \\
&& \ \ \ \ \ \ \ \ \ \{1_A X_{\psi_{i-1}} + 1_{A^c} X_{\psi_i}, 1_A \psi_{i-1} + 1_{A^c} \psi_i  : 3 \leq i \leq j \} \Big) \\
&& =  a\left(\left\{X_{\psi_i},\psi_i \right\}_{1 \leq i \leq k} \right) x \left( X_{\psi_1},\psi_1, \{1_A X_{\psi_{i-1}} + 1_A X_{\psi_i}, 1_A \psi_{i-1} + 1_{A^c} \psi_i  : 3 \leq i \leq j \} \right) \\
&& \ \ \ \ \  + \left(1- a\left(\left\{X_{\psi_i},\psi_i \right\}_{1 \leq i \leq k} \right) \right) y \Big( X_{\psi_1},\psi_1, 1_A X_\tau + 1_{A^c} X_{\psi_2}, 1_A \tau + 1_{A^c} \psi_2, \\
&& \ \ \ \ \  \ \ \ \ \{1_A X_{\psi_{i-1}} + 1_{A^c} X_{\psi_i}, 1_A \psi_{i-1} + 1_{A^c} \psi_i  : 3 \leq i \leq j \} \Big) \in m\mc{F}^{\widetilde{\Psi}}_{\widetilde{\psi}_j}.
\end{eqnarray*}
\end{proof}

Having concluded the proof of the claim, we have shown that the observation strategy $\widetilde{\Psi}$ is admissible, or $\widetilde{\Psi} \in \mathfrak{O}^n$.  By construction, $\tau$ does not stop between $\widetilde{\psi}_k$ and $\widetilde{\psi}_{k+1}$.

We have to check that $\tau$ is a $\mb{F}^{\widetilde{\Psi}}$-stopping time.  Let $t>0$.  Then $\{\tau \leq t \} \in \mc{F}^{\widetilde{\Psi}}_t$ if and only if $\{\tau \leq t \} \cap \{\widetilde{\psi}_j \leq t\} \in \mc{F}^{\widetilde{\Psi}}_{\widetilde{\psi}_j}$ for each $1 \leq j \leq n$.  This is clear for $1 \leq j \leq k$ as $\widetilde{\psi}_j = \psi_j$ for $1 \leq j \leq k$ and $\tau$ is a $\mb{F}^\Psi$-stopping time.  For $j = k + 1$, we work on $A$ and $A^c$ separately.  We must show
\[
\{\tau \leq t \} \cap \{\widetilde{\psi}_{k+1} \leq t \} \in \mc{F}^{\widetilde{\Psi}}_{\widetilde{\psi}_{k+1}}.
\]
Invoking the claim, it is suffcient to show that
\begin{equation}\label{sto_eq1}
\{ \tau \leq t \} \cap \{\widetilde{\psi}_{k+1} \leq t \} \cap A \in \mc{F}^\Psi_{\psi_k}
\end{equation}
and
\begin{equation}\label{sto_eq2}
\{ \tau \leq t \} \cap \{\widetilde{\psi}_{k+1} \leq t \} \cap A^c \in \mc{F}^\Psi_{\psi_{k+1}}.
\end{equation}
To show \eqref{sto_eq1}, note that on the set $A$, $\widetilde{\psi}_{k+1} = \tau$, and so $\{ \tau \leq t \} \cap \{\widetilde{\psi}_{k+1} \leq t\} \cap A = \{\tau \leq t \} \cap A$.  Since $\tau$ is a $\mb{F}^\Psi$-stopping time, greater than or equal to $\psi_k$, it follows that $\{\tau \leq t \} = \{\tau \leq t \} \cap \{\psi_k \leq t\} \in \mc{F}^\Psi_{\psi_k}$.  To show \eqref{sto_eq2}, note that on the set $A^c$, $\widetilde{\psi}_{k+1} = \psi_{k+1}$, so $\{ \tau \leq t \} \cap \{\widetilde{\psi}_{k+1} \leq t \} \cap A^c = \{ \tau \leq t \} \cap \{\psi_{k+1} \leq t \} \cap A^c \in \mc{F}^\Psi_{\psi_{k+1}}$, noting that $\tau$ is a $\mb{F}^\Psi$ stopping time and $A \in \mc{F}^\Psi_{\psi_k}$.

Now, fix $j \geq k+2$.   We may write
\[
\{\tau \leq t \} \cap \{\widetilde{\psi}_j \leq t\} = \Big[ A \cap \{\tau \leq t \} \cap \{\psi_{j-1} \leq t\} \Big] \cup \Big[ A^c \cap \{\tau \leq t \} \cap \{\psi_j \leq t\} \Big].
\]
Since $\tau$ is a $\mb{F}^\Psi$-stopping time, we know that $\{\tau \leq t \} \cap \{\psi_{j-1}  \leq t\} \in \mc{F}^\Psi_{\psi_{j-1}}$ and $\{\tau \leq t \} \cap \{\psi_j \leq t\} \in \mc{F}^\Psi_{\psi_j}$.  Therefore, $\{\tau \leq t \} \cap \{\widetilde{\psi}_j \leq t\} \in \mc{F}^{\widetilde{\Psi}}_{\psi_j}$, again by Claim \ref{claim1}.

Finally, note that $E^\phi \left[ \int_0^\tau e^{-\lambda t} \left( \Phi^\Psi_t - \frac{\lambda}{c} \right) dt \right] = E^\phi \left[ \int_0^\tau e^{-\lambda t} \left( \Phi^{\widetilde{\Psi}}_t - \frac{\lambda}{c} \right) dt \right]$ because $\Phi^\Psi = \Phi^{\widetilde{\Psi}}$ a.s.  on the random time interval $[0,\tau)$.

%\[
%\{\tau \leq \widetilde{\psi}_j\} = \left( A \cap \{\tau \leq \psi_{j-1}\}\right) \cup \left( A^c \cap \{ \tau \leq \psi_j \} \right)
%\]

%To do this, it suffices, by definition of $\mb{F}^{\widetilde{\Psi}}$, to show that $\{\tau \leq \widetilde{\psi_j} \} \in \mc{F}^{\Psi^k}_{\psi_j}$ for each $1 \leq j \leq n$.  This is clear for $j=1$ as $\widetilde{\psi}_1 = \psi_1$.  For $j = 2$, $\tau \leq \widetilde{\psi}_2$ by construction, so $\{\tau \leq \widetilde{\psi}_2\} = \Omega \in \mc{F}^{\Psi^k}_{\widetilde{\psi}_2}$.  For $3 \leq j \leq n$,
\end{proof}

\begin{proof}[Proof of Proposition \ref{dpp1}] 

\textbf{Step 1.} $\gamma^n_k(\Psi^k) \geq v_{n-k} \left( \Phi^{\Psi^k}_{\psi_k} \right)$:
 
We proceed by backwards induction, so consider the base case $k = n$, and take $\Psi^n =\{\psi^n_1,\ldots,\psi^n_n\} \in \mathfrak{O}^n$.  We will prove equality here.  We must show that 
\[
\begin{split}
v_0 \left(\Phi^{\Psi^n}_{\psi^n_n} \right)
&=\underset{\Psi \in \mathfrak{O}^n(\Psi^n)}{\ei} \ \underset{\tau \in \mc{T}^\Psi_o, \tau \geq \psi^n_n}{\ei} E \left[ \int_{\psi^n_n}^\tau e^{-\lambda (t-\psi^n_n)} \left( \Phi^\Psi_t - \frac{\lambda}{c} \right) dt \big | \mc{F}^{\Psi^n}_{\psi^n_n} \right]
\\&= \underset{\tau \in \mc{T}^{\Psi^n}, \tau \geq \psi^n_n}{\ei} E \left[ \int_{\psi^n_n}^\tau e^{-\lambda (t-\psi^n_n)} \left( \Phi^{\Psi^n}_t - \frac{\lambda}{c} \right) dt \big | \mc{F}^{\Psi^n}_{\psi^n_n} \right],
\end{split}
\]
the second inequality following from the fact that we can observe a total of $n$ times, so that $\mathfrak{O}^n(\Psi^n)$ is a singleton, and equal to $\Psi^n$.

Recall
\begin{equation}
t^*_0(\phi) = \frac{1}{\lambda} \log \left(\frac{c+ \lambda}{c(\phi + 1)} \right) \vee 0.
\end{equation} 

A simple calculation confirms that $J_0 0(\phi) = J0(t^*(\phi),\phi)$.

According to Proposition \ref{prop1}, there is a one-to-one correspondence between $\{\tau \in \mc{T}^{\Psi^n} : \tau \geq \psi^n_n \}$ and the set $\{\psi^n_n + R_n : R_n \geq 0, R_n \in m \mc{F}^{\Psi^n}_{\psi^n_n}\}$.  So, take $\tau^* = \psi^n_n + t^*_0(\Phi^{\Psi^n}_{\psi^n_n}) \in \{\tau \in \mc{T}^{\Psi^n} : \tau \geq \psi^n_n \}$.  Thus,
\begin{eqnarray*}
\lefteqn{\underset{\tau \in \mc{T}^{\Psi^n}, \tau \geq \psi^n_n}{\ei} E \left[ \int_{\psi^n_n}^\tau e^{-\lambda (t-\psi^n_n)} \left( \Phi^{\Psi^n}_t - \frac{\lambda}{c} \right) dt \big | \mc{F}^{\Psi^n}_{\psi^n_n} \right]} \\
&& \leq E \left[ \int_{\psi^n_n}^{\tau^*} e^{-\lambda (t-\psi^n_n)} \left( \Phi^{\Psi^n}_t - \frac{\lambda}{c} \right) dt \big | \mc{F}^{\Psi^n}_{\psi^n_n} \right].
\end{eqnarray*}

According to \eqref{stateprocess}, the process $\Phi^{\Psi^n}_t$ is increasing for $t \geq \psi^n_n$, and if $t^*_0(\Phi^{\Psi^n}_t)>0$, then $\psi^n_n + t^*_0(\Phi^{\Psi^n}_{\psi^n_n})$ is the unique time $ t \geq \psi^n_n$ when $\Phi_t^{\Psi^n} - \frac{\lambda}{c}$ changes sign from negative to positive.  If $t^*_0(\Phi^{\Psi^n}_t)=0$, then $\Phi^{\Psi^n}_t$ is always greater than $\lambda/c$.
Therefore,
\begin{eqnarray*}
\lefteqn{\underset{\tau \in \mc{T}^{\Psi^n}, \tau \geq \psi^n_n}{\ei} E \left[ \int_{\psi^n_n}^\tau e^{-\lambda (t-\psi^n_n)} \left( \Phi^{\Psi^n}_t - \frac{\lambda}{c} \right) dt \big | \mc{F}^{\Psi^n}_{\psi^n_n} \right]} \\
&& \geq E \left[ \int_{\psi^n_n}^{\tau^*} e^{-\lambda (t-\psi^n_n)} \left( \Phi^{\Psi^n}_t - \frac{\lambda}{c} \right) dt \big | \mc{F}^{\Psi^n}_{\psi^n_n} \right],
\end{eqnarray*}
the inequality holding at the pointwise level.  It therefore follows that $\gamma^n_n(\Psi^n) = v_0(\Phi^{\Psi^n}_{\psi^n_n})$.

Now, for the inductive step, suppose that $\gamma^n_{k+1}(\Psi^{k+1}) \geq v_{n-k - 1} \left( \Phi^{\Psi^{k+1}}_{\psi^{k+1}_{k+1}} \right)$ for all $\Psi^{k+1} =\{\psi^{k+1}_i\}_{1 \leq i \leq k +1} \in \mathfrak{O}_{k+1}$.  Let $\Psi^k = \{\psi^k_i\}_{1 \leq i \leq k} \in \mathfrak{O}^k$.  We wish to show that 
\[
\gamma^n_k(\Psi^k) \geq v_{n-k} \left(\Psi^k_{\psi^k_k} \right).
\]
Let $\widetilde{\Psi} \in \mathfrak{O}^n(\Psi^k)$, and write $\widetilde{\Psi} = \{\psi^k_1,\ldots,\psi^k_k,\widetilde{\psi}_{k+1},\ldots,\widetilde{\psi}_n\}$.  Without loss of generality, we assume that $\widetilde{\psi}_{k+1} > \psi^k_k$.

Then

{\small{
\begin{flalign*}
& \underset{\tau \in \mc{T}^{\widetilde{\Psi}}_o, \tau \geq \psi^k_k}{\ei} E \left[ \int_{\psi^k_k}^\tau e^{-\lambda (t - \psi^k_k)} \left( \Phi^{\widetilde{\Psi}}_t - \frac{\lambda}{c} \right) dt \big | \mc{F}^{\Psi^k}_{\psi^k_k} \right] & \\
& =\min \Bigg\{0, \underset{\tau \in \mc{T}^{\widetilde{\Psi}}_o, \tau \geq \widetilde{\psi}_{k+1}}{\ei} E \Bigg[ \int_{\psi^k_k}^{\widetilde{\psi}_{k+1}} e^{-\lambda (t - \psi^k_k)} \left( \Phi^{\widetilde{\Psi}}_t - \frac{\lambda}{c} \right) dt + \int_{\widetilde{\psi}_{k+1}}^\tau e^{-\lambda (t - \psi^k_k)} \left( \Phi^{\widetilde{\Psi}}_t - \frac{\lambda}{c} \right) dt \big | \mc{F}^{\Psi^k}_{\psi^k_k} \Bigg] \Bigg\} & \\ 
& = \min \Bigg\{0, E \left[ \int_{\psi^k_k}^{\widetilde{\psi}_{k+1}} e^{-\lambda (t - \psi^k_k)} \left( \Phi^{\widetilde{\Psi}}_t - \frac{\lambda}{c} \right) dt \big | \mc{F}^{\Psi^k}_{\psi^k_k} \right] \\
&\ \ \ \ \ \ \ \ \ \ \ \ + \underset{\tau \in \mc{T}^{\widetilde{\Psi}}_o, \tau \geq \widetilde{\psi}_{k+1}}{\ei} E \left[  \int_{\widetilde{\psi}_{k+1}}^\tau e^{-\lambda (t - \psi^k_k)} \left( \Phi^{\widetilde{\Psi}}_t - \frac{\lambda}{c} \right) dt \big | \mc{F}^{\Psi^k}_{\psi^k_k} \right] \Bigg \} & \\ 
& \geq \min \Bigg\{0, \int_0^{\widetilde{\psi}_{k+1} - \psi^k_k} e^{-\lambda t} \left (\varphi \left(t,\Phi^{\Psi^k}_{\psi^k_k} \right) - \frac{\lambda}{c} \right) dt + e^{-(\widetilde{\psi}_{k+1} - \psi^k_k)} E \left[ v_{n - k - 1}\left( \Phi^{\widetilde{\Psi}}_{\widetilde{\psi}_{k+1}} \right) \big | \mc{F}^{\Psi^k}_{\psi^k_k} \right] \Bigg\} &
\end{flalign*}
}}
by the inductive hypothesis, where we have used the fact that $\widetilde{\psi}_{k+1}$ is $\mc{F}^{\Psi^k}_{\psi^k_k}$-measurable as well as the deterministic  dynamics of $\Phi^{\widetilde{\Psi}}$ in between jumps.  Next,
{\small{
\begin{eqnarray*}
\lefteqn{\min \left\{0, \int_0^{\widetilde{\psi}_{k+1} - \psi^k_k} e^{-\lambda t} \left (\varphi \left(t,\Phi^{\Psi^k}_{\psi^k_k} \right) - \frac{\lambda}{c} \right) dt + e^{-(\widetilde{\psi}_{k+1} - \psi^k_k)} E \left[ v_{n - k - 1}\left( \widetilde{\Psi}_{\widetilde{\psi}_{k+1}} \right) \big | \mc{F}^{\Psi^k}_{\psi^k_k} \right] \right\}} \\
&& = \min \Bigg \{0, \int_0^{\widetilde{\psi}_{k+1} - \psi^k_k} e^{-\lambda t} \left (\varphi \left(t,\Phi^{\Psi^k}_{\psi^k_k} \right) - \frac{\lambda}{c} \right) dt \\
&& \ \ \ \ \  + e^{-(\widetilde{\psi}_{k+1} - \psi^k_k)} E \left[v_{n - k - 1}\left( j \left( \widetilde{\psi}_{k+1} - \psi^k_k, \Phi^{\Psi^k}_{\psi^k_k},\frac{X_{\widetilde{\psi}_{k+1}} - X_{\psi^k_k}}{\sqrt{\widetilde{\psi}_{k+1} - \psi^k_k}} \right) \right) \big | \mc{F}^{\Psi^k}_{\psi^k_k} \right] \Bigg \} \\
&& = \min \left \{0, \int_0^{\widetilde{\psi}_{k+1} - \psi^k_k} e^{-\lambda t} \left (\varphi \left(t,\Phi^{\Psi^k}_{\psi^k_k} \right) - \frac{\lambda}{c} \right) dt + e^{-(\widetilde{\psi}_{k+1} - \psi^k_k)} K v_{n-k-1} \left( \widetilde{\psi}_{k+1} - \psi^k_k, \Phi^{\Psi^k}_{\psi^k_k} \right) \right \} \\
&& = \min \left\{ 0, J v_{n-k-1} \left(\widetilde{\psi}_{k+1} - \psi^k_k,\Phi^{\Psi^k}_{\psi^k_k}\right) \right\},
\end{eqnarray*}
}}
noting that $\widetilde{\psi}_{k+1} - \psi^k_k > 0$ and that $\frac{X_{\widetilde{\psi}_{k+1}} - X_{\psi^k_k}}{\sqrt{\widetilde{\psi}_{k+1} - \psi^k_k}}$ has standard normal distribution, independent of $\mc{F}^{\Psi^k}_{\psi^k_k}$.  Finally, $\min \left\{ 0, Jv_{n - k -1} \left(\widetilde{\psi}_{k+1} - \psi^k_k,\Phi^{\Psi^k}_{\psi^k_k}\right) \right\} \geq J_0 v_{n-k-1} \left(\Phi^{\Psi^k}_{\psi^k_k}\right) = v_{n-k} \left(\Phi^{\Psi^k}_{\psi^k_k} \right)$.

This establishes that 
\[
\underset{\tau \in \mc{T}^{\widetilde{\Psi}}_o, \tau \geq \psi^k_k}{\ei} E \left[ \int_{\psi^k_k}^\tau e^{-\lambda (t - \psi^k_k)} \left( \Phi^{\widetilde{\Psi}}_t - \frac{\lambda}{c} \right) dt \big | \mc{F}^{\Psi^k}_{\psi^k_k} \right] \geq v_{n-k} \left(\Phi^{\Psi^k}_{\psi^k_k} \right).
\]
Taking the infimum over all $\widetilde{\Psi} \in \mathfrak{O}^n(\Psi^k)$, we obtain that $\gamma^n_k(\Psi^k) \geq v_{n-k} \left(\Psi^k_{\psi^k_k} \right)$.  

\textbf{Step 2.} $\gamma^n_k(\Psi^k) \leq v_{n-k} \left( \Phi^{\Psi^k}_{\psi_k} \right)$:

 As in the first step, we proceed by reverse induction, and note that the base case has already been established in Step 1.  Therefore, we assume that $\gamma^n_{k+1}(\Psi^{k+1}) = v_{n - k - 1} \left( \Phi^{\Psi^{k+1}}_{\psi^{k+1}_{k+1}} \right)$ for all $\Psi^{k+1} \in \mathfrak{O}^{k+1}$.  Let $\Psi^k \in \mathfrak{O}^k$.  We wish to show that $\gamma^n_k(\Psi^k) \leq v_{n-k} \left(\Phi^{\Psi^k}_{\psi^k_k} \right)$.  Recall the functions
\[
h^0_{n-k}(\phi) = \min \{ s \geq 0 : J v_{n-k}(s,\phi) \leq J_0 v_{n-k}(\phi) \},
\]
and the stopping times $\widehat{\psi}^{k+1} \triangleq \psi^k_k + h^0_{n-k} \left(\Phi^{\Psi^k}_{\psi^k_k} \right) $.  Note that we can use a minimum above instead of an infimum because $J v_{n-k}(s,\phi)$ is lower semi-continuous in $s$, as $v_{n-k}$ is nonpositive.  Write $\widehat{\Psi}^{k+1} = \left \{\psi^k_1,\ldots,\psi^k_k,\widehat{\psi}_{k+1} \right\}$.

Then
{\small{
\[
\begin{split}
\gamma^n_k(\Psi^k)
&=  \underset{\Psi \in \mathfrak{O}^n(\Psi^k)}{\ei} \ \underset{\tau \in \mc{T}^\Psi_o, \tau \geq \psi^k_k}{\ei} E \left[ \int_{\psi^k_k}^\tau e^{-\lambda (t-\psi^k_k)} \left( \Phi^\Psi_t - \frac{\lambda}{c} \right) dt \big | \mc{F}^{\Psi^k}_{\psi^k_k} \right]
\\&\leq \underset{\Psi \in \mathfrak{O}^n(\widehat{\Psi}^{k+1})}{\ei} \ \underset{\tau \in \mc{T}^\Psi_o, \tau \geq \psi^k_k}{\ei} E \left[ \int_{\psi^k_k}^\tau e^{-\lambda (t-\psi^k_k)} \left( \Phi^\Psi_t - \frac{\lambda}{c} \right) dt \big | \mc{F}^{\Psi^k}_{\psi^k_k} \right]
\\&=\min \Bigg\{0, E \Bigg[ \int_{\psi^k_k}^{\widehat{\psi}_{k+1}} e^{-\lambda (t - \psi^k_k)} \left( \Phi^{\Psi}_t - \frac{\lambda}{c} \right) dt 
\\& \ \ \ \ \ \ \ \ \ + \underset{\Psi \in \mathfrak{O}^n(\widehat{\Psi}^{k+1})}{\ei} \ \underset{\tau \in \mc{T}^{\Psi}_o, \tau \geq \widehat{\psi}_{k+1}}{\ei} \ \int_{\widehat{\psi}_{k+1}}^\tau e^{-\lambda (t - \psi^k_k)} \left( \Phi^{\Psi}_t - \frac{\lambda}{c} \right) dt \big | \mc{F}^{\Psi^k}_{\psi^k_k} \Bigg] \Bigg\}
\\&=\min \left\{0, E \left[ \int_{\psi^k_k}^{\widehat{\psi}_{k+1}} e^{-\lambda (t - \psi^k_k)} \left( \Phi^{\widehat{\Psi}^{k+1}}_t - \frac{\lambda}{c} \right) dt + e^{-(\widehat{\psi}_{k+1}-\psi^k_k)}\gamma^n_{k+1} \left(\Phi^{\widehat{\Psi}^{k+1}}_{\widehat{\psi}_{k+1}} \right) \big | \mc{F}^{\Psi^k}_{\psi^k_k} \right] \right\}
\\&=\min \left\{0, E \left[ \int_0^{\widehat{\psi}_{k+1}-\psi^k_k} e^{-\lambda t} \left( \varphi(t,\Phi^{\Psi^k}_{\psi^k_k}) - \frac{\lambda}{c} \right) dt + e^{-(\widehat{\psi}_{k+1}-\psi^k_k)}v_{n-k-1} \left(\Phi^{\widehat{\Psi}^{k+1}}_{\widehat{\psi}_{k+1}} \right) \big | \mc{F}^{\Psi^k}_{\psi^k_k} \right] \right\},
\end{split}
\]
}}
with the last equality following from the inductive hypothesis and the deterministic evolution of $\Phi^{\Psi^k}$ in between jumps.  It now follows, using an argument similar to that at the end of the proof of Step 1, that
\begin{flalign*}
& \min \left\{0, E \left[ \int_0^{\widehat{\psi}^{k+1}-\psi^k_k} e^{-\lambda t} \left( \varphi(t,\Phi^{\Psi^k}_{\psi^k_k}) - \frac{\lambda}{c} \right) dt + e^{-(\widehat{\psi}^{k+1}-\psi^k_k)}v_{n-k-1} \left(\Phi^{\widehat{\Psi}^{k+1}}_{\widehat{\psi}^{k+1}} \right) \big | \mc{F}^{\Psi^k}_{\psi^k_k} \right] \right\} & \\
& = J_0 v_{n-k-1} \left(\Phi^{\Psi^k}_{\psi^k_k} \right) & \\
& = v_{n-k} \left(\Phi^{\Psi^k}_{\psi^k_k} \right), &
\end{flalign*}
with the first equality above by definition of $\widehat{\psi}_{k+1}$ and $h^0_{n-k}$.  It follows then that
\[
\gamma^n_k(\Psi^k) \leq v_{n-k}(\Phi^{\Psi^k}_{\psi_k}),
\]
and the equality now follows.
\end{proof}

\begin{proof}[Proof of Lemma \ref{measlem}] We claim that $h^\epsilon_{n-k}(\phi)$ is lower semi-continuous.  This would imply that $h^\epsilon_{n-k}(\phi)$ is Borel-measurable, which will in turn imply that $\widehat{\psi}^{k+1}$ is a stopping time.

Let $\phi_i \ra \phi_\infty$ in $\mb{R}_+$, and let $s_i = h^\epsilon_{n-k}(\phi_i)$.  Note that $J_0 v_{n-k}(\cdot)$ is bounded, and that
\[
\underset{s \ra \infty}{\lim} \ \underset{\phi \geq 0}{\inf} Jv_{n-k}(s,\phi) = +\infty.
\]
Therefore, we may assume that the sequence $\{s_i\}_{i \geq 0}$ is bounded.  It follows then that $\underset{i \ra \infty}{\liminf}\  s_i = s_\infty <\infty$.   It is straightforward to see that $J v_{n-k}(\cdot, \cdot)$ and $J_0 v_{n-k}(\cdot)$ are continuous in their arguments.  Therefore,
\[
\begin{split}
J v_{n-k}(s_\infty,\phi_\infty) & \leq \underset{i \ra \infty}{\liminf} \  J v_{n-k}(s_i, \phi_i) \\
& \leq \underset{i \ra \infty}{\lim} \ J_0 v_{n-k}(\phi_i) + \epsilon \\
& = J_0 v_{n-k}(\phi_\infty) + \epsilon.
\end{split}
\]
Thus, 
\[
\begin{split}
h^\epsilon_{n-k}(\phi_\infty) & \leq s_ \infty \\
& = \underset{i \ra \infty}{\liminf} \ s_i \\
& = \underset{i \ra \infty}{\liminf} \ h^\epsilon_{n-k}(\phi_i),
\end{split}
\]
establishing lower semi-continuity.
\end{proof}

\section{\uppercase{Proofs from Section 3}}\label{nobs_cont_proof}

\subsection{Estimates for the Second Moment of $\widetilde{\Phi}^n_t$}

\begin{lemma}\label{lem1} Let $Z$ be a standard normal random variable. Then for $t>0$ and $\phi \geq 0$, $E[j(t,Z,\phi)] = e^{\lambda t}(\phi + 1) - 1$.
\end{lemma}

\begin{proof} We may write
\[
E[j(t,Z,\phi)] = \int_{-\infty}^\infty \left( e^{\alpha z \sqrt{t} + (\lambda - \alpha^2/2)t} \phi + \int_0^t \lambda e^{\lambda u + \frac{\alpha z}{\sqrt{t}} u - \frac{\alpha^2}{2t} u^2} \right) e^{-z^2/2}/\sqrt(2\pi) dz.
\]
The integral of the first term is calculated by completing the square and equals $e^{\lambda t}\phi$. The integral of the second term is calculated by switching the order of integration and completing the square, yielding $e^{\lambda t} - 1$.
\end{proof}

\begin{corollary}\label{cor1} Let $n$ and $k$ be positive integers. Then 
\[
E \left[\widetilde{\Phi}^n_{\frac{k}{n}} \big| \mathcal{F}^n_{\frac{k-1}{n}} \right] = e^{\frac{\lambda}{n}}\left(\widetilde{\Phi}^n_{\frac{k-1}{n}} + 1 \right) - 1.
\]
\end{corollary}

\begin{proof} Apply Lemma \ref{lem1} with $t = \frac{1}{n}$, using the fact that $\widetilde{\Phi}^n_{\frac{k}{n}} = j \left(\frac{1}{n}, Z_k, \widetilde{\Phi}^n_{\frac{k-1}{n}} \right)$, with $Z_k$ independent of $\mathcal{F}^n_{\frac{k-1}{n}}$.
\end{proof}

\begin{lemma}\label{lem2} Let $Z$ be a standard normal random variable. Then for $\phi \geq 0$ and $t > 0$,
\[
E[j(t,Z,\phi)^2] \leq \phi^2 e^{2 \lambda t + \alpha^2 t} + 2 \phi e^{\lambda t} \frac{\lambda}{\lambda + \alpha^2}\left(e^{\lambda t + \alpha^2 t} - 1 \right) + e^{\alpha^2 t} \left(e^{\lambda t} - 1 \right)^2.
\]
\end{lemma}

\begin{proof}
We start by expanding $E\left[j(t, Z, \phi)^2\right]$ into three terms:
\begin{eqnarray*}
\lefteqn{E \left[j(t,Z,\phi)^2 \right]} \\
&& = \left. \int_{-\infty}^\infty \phi^2 e^{2 \alpha z \sqrt{t} + 2(\lambda - \alpha^2/2)t} e^{-z^2/2}/\sqrt(2\pi) dz \right\} (1) \\
&& \ \ \ \ \ + \left. 2\phi \int_{-\infty}^\infty e^{\alpha z \sqrt{t} + (\lambda - \alpha^2/2)t} \left(\int_0^t \lambda e^{\lambda u + \frac{\alpha z}{\sqrt{t}}u - \frac{\alpha^2}{2t}u^2} du \right) e^{-z^2/2}/\sqrt{2\pi} dz \right\} (2) \\
&& \ \ \ \ \ + \left. \int_{-\infty}^\infty \left(\int_0^t \lambda e^{\lambda u + \frac{\alpha z}{\sqrt{t}}u - \frac{\alpha^2}{2t}u^2} du \right) \left(\int_0^t \lambda e^{\lambda w + \frac{\alpha z}{\sqrt{t}}w - \frac{\alpha^2}{2t}w^2} dw \right) e^{-z^2/2}/\sqrt{2\pi} dz \right\} (3)
\end{eqnarray*}

We calculate each of these terms seprately. $(1)$ is the simplest, and by completing the square, we can calculate its value to be $\phi^2 e^{2 \lambda t} e^{\alpha ^2 t}$. We have
\[
\begin{split}
(2)
&= 2\phi \int_0^t \lambda e^{\lambda t} e^{\lambda u} e^{\alpha^2 u} \left( \int_{-\infty}^\infty e^{- \left(z - \frac{\alpha u}{\sqrt{t}} - \alpha \sqrt{t} \right)^2/2}/\sqrt{2\pi} dz \right) du
\\& = 2 \phi e^{\lambda t} \int_0^t \lambda e^{\lambda u} e^{\alpha^2 u} du
\\& = 2 \phi e^{\lambda t} \frac{\lambda}{\lambda + \alpha^2} \left(e^{\lambda t + \alpha^2 t} - 1 \right).
\end{split}
\]

Now, the third term we cannot calculate exactly, but we give an upper bound for it which will be good enough:
\[
\begin{split}
(3)
& = \int_0^t \int_0^t \lambda e^{\lambda u} e^{\lambda w} e^{\frac{\alpha^2 uw}{t}} \left( \int_{-\infty}^\infty e^{-\left( z - \frac{\alpha u}{\sqrt{t}} - \frac{\alpha w}{\sqrt{t}} \right)^2/2}/\sqrt{2\pi} dz \right) du dw
\\& = \int_0^t \int_0^t \lambda e^{\lambda u} e^{\lambda w} e^{\frac{\alpha^2 uw}{t}} du dw
\\& \leq \int_0^t \int_0^t \lambda e^{\lambda u} e^{\lambda w} e^{\alpha^2 t} du dw
\\& = e^{\alpha^2 t} \left(e^{\lambda t} - 1 \right)^2,
\end{split}
\]
where in the inequality above, we have the used the fact that $uw \leq t^2$ for $u, w \in [0,t]$. Combining $(1), (2)$, and $(3)$, we deduce the lemma.
\end{proof}

\begin{corollary}\label{cor2} Let $n$ and $k$ be positive integers. Then

{\small{
\[
E \left[ \left(\widetilde{\Phi}^n_{\frac{k}{n}} \right)^2 \big| \mathcal{F}^n_{\frac{k-1}{n}} \right] \leq \left(\widetilde{\Phi}^n_{\frac{k-1}{n}}\right)^2 e^{\frac{1}{n}(2\lambda + \alpha^2)} + 2\widetilde{\Phi}^n_{\frac{k-1}{n}} e^{\frac{\lambda}{n}} \frac{\lambda}{\lambda + \alpha^2} \left(e^{\frac{1}{n}(\lambda + \alpha^2)}
-1 \right) +  e^{\frac{\alpha^2}{n}} \left(e^{\frac{\lambda}{n}} - 1 \right)^2. 
\]
}}
\end{corollary}

\begin{proof} Apply Lemma \ref{lem2}, as Lemma \ref{lem1} is used in the proof of Corollary \ref{cor1}.
\end{proof}
Since we are interested in the limit as $n \rightarrow \infty$, we we also set down asymptotic versions of Lemma \ref{lem2} and Corollary \ref{cor2}.

\begin{lemma}\label{lem3} Let $\phi \geq 0$, and $Z$ a standard normal random variable. Then as $t \downarrow 0$,
\[
E \left[ j(t, Z, \phi)^2 \right] = \phi^2(1 + 2 \lambda t + \alpha^2 t + O(t^2)) + \phi(2 \lambda t + O(t^2)) + O(t^2).
\]
\end{lemma}

\begin{proof} By examining the proof of Lemma \ref{lem2}, we can see that $E \left[ j(t, Z, \phi)^2 \right] = \phi^2 e^{2 \lambda t}e^{\alpha^2 t} + 2\phi e^{\lambda t} \frac{\lambda}{\lambda + \alpha^2} \left(e^{\lambda t + \alpha^2 t} - 1 \right)$, plus a positive third term, which can be bounded from above by $e^{\alpha^2 t} \left(e^{\lambda t} - 1 \right)^2$. Note that this term
\[
\begin{split}
e^{\alpha^2 t} \left(e^{\lambda t} - 1 \right)^2 & = (1 + \alpha^2 t + O(t^2))(\lambda t + O(t^2))^2 \\
& = O(t^2).
\end{split}
\]
So, we are left with the first two terms. We have
\[
\begin{split}
\phi^2 e^{2 \lambda t}e^{\alpha^2 t}
& = \phi^2(1 + 2\lambda t + O(t^2))(1 + \alpha^2 t + O(t^2))
\\& = \phi^2(1 + 2\lambda t + \alpha^2 t + O(t^2)),
\end{split}
\]
and
\[
\begin{split}
2\phi e^{\lambda t} \frac{\lambda}{\lambda + \alpha^2} \left(e^{\lambda t + \alpha^2 t} - 1 \right)
& = 2 \phi (1 + \lambda t + O(t^2))\frac{\lambda}{\lambda + \alpha^2} (\lambda t + \alpha^2 t + O(t^2))
\\& = 2\phi (t + O(t^2)).
\end{split}
\]
\end{proof}

\begin{corollary}\label{cor3}
Let $k$ be a positive integer. As $n \rightarrow \infty$,

{\small{
\[
E \left[ \left(\widetilde{\Phi}^n_{\frac{k}{n}} \right)^2 \big| \mathcal{F}^n_{\frac{k-1}{n}} \right] = \left(\widetilde{\Phi}^n_{\frac{k-1}{n}} \right)^2 \left(1 + \frac{2 \lambda}{n} + \frac{\alpha^2}{n} + O \left(\frac{1}{n^2} \right) \right) + \widetilde{\Phi}^n_{\frac{k-1}{n}}\left(\frac{2 \lambda}{n} + O \left(\frac{1}{n^2}\right) \right) + O \left(\frac{1}{n^2} \right).
\]
}}
\end{corollary}

\subsection{Proving Proposition \ref{mainprop} by Establishing the Conditions of Proposition \ref{ald1}}

We verify the six conditions of Proposition \ref{ald1} separately. First note that Condition $(a)$ is satisfied by construction. We will defer $(b)$ until last.

\begin{proof}[Conditions $(c), (e)$]

First, we will construct the process $N^n_t$. This is essentially done by Doob Decomposition. We set $N^n_0 = 0$. First, we will define $N^n$ at grid points $\{\frac{1}{n},\frac{2}{n},\ldots \}$, and then extend to all of $\mathbb{R}_+$. We construct $N^n$ on the grid points inductively: for $k \geq 1$, 
\begin{equation}\label{eq1}
\begin{split}
N^n_{\frac{k}{n}} - N^n_{\frac{k-1}{n}}
& \triangleq E \left[ \widetilde{\Phi}^n_{\frac{k}{n}} - \widetilde{\Phi}^n_{\frac{k-1}{n}} - \int_{\frac{k-1}{n}}^{\frac{k}{n}} b(\widetilde{\Phi}^n_s) ds | \mathcal{F}^n_{\frac{k-1}{n}} \right]
\\& = E \left[ \widetilde{\Phi}^n_{\frac{k}{n}} - \widetilde{\Phi}^n_{\frac{k-1}{n}} - \frac{1}{n}\lambda \left(1 + \widetilde{\Phi}^n_{\frac{k-1}{n}} \right) | \mathcal{F}^n_{\frac{k-1}{n}} \right],
\end{split}
\end{equation}
using $b(x) = \lambda(1 + x)$, as well as the fact that $\widetilde{\Phi}^n_s = \widetilde{\Phi}^n_{\frac{k-1}{n}}$ for $\frac{k-1}{n} \leq s < \frac{k}{n}$ by construction.
Now by Corollary \ref{cor2},

\begin{equation}\label{eq1.1}
\begin{split}
E \left[ \widetilde{\Phi}^n_{\frac{k}{n}} \big| \mathcal{F}^n_{\frac{k-1}{n}} \right]
&= e^{\frac{\lambda}{n}}\left(\widetilde{\Phi}^n_{\frac{k-1}{n}} + 1 \right) - 1
\\&= \left( 1 + \frac{\lambda}{n} + \frac{\lambda^2}{n^2} + \cdots \right) \left(\widetilde{\Phi}^n_{\frac{k-1}{n}} + 1 \right) - 1
\\& = \left(1 + \frac{\lambda}{n} + O \left(\frac{1}{n^2} \right) \right) \left(\widetilde{\Phi}^n_{\frac{k-1}{n}} + 1 \right) - 1.
\end{split}
\end{equation}

Therefore, plugging \eqref{eq1.1} in \eqref{eq1},
\begin{equation}\label{eq2}
\begin{split}
N^n_\frac{k}{n} - N^n_\frac{k-1}{n}
& = \left(1 + \frac{\lambda}{n} + O \left(\frac{1}{n^2} \right) \right) \left(\widetilde{\Phi}^n_{\frac{k-1}{n}} + 1 \right) - 1 -\widetilde{\Phi}^n_\frac{k-1}{n} - \frac{1}{n} \lambda \left( 1 + \widetilde{\Phi}^n_\frac{k-1}{n} \right)
\\& = \left( \frac{\lambda^2}{n^2} + \frac{\lambda^3}{n^3} + \cdots \right)\left(\widetilde{\Phi}^n_\frac{k-1}{n} + 1 \right)
\\& = O \left(\frac{1}{n^2} \right) \left(\widetilde{\Phi}^n_\frac{k-1}{n} + 1 \right).
\end{split}
\end{equation}
This defines $N^n$ on all grid points. Next, for $\frac{k-1}{n} \leq t < \frac{k}{n}$,
\[
\begin{split}
N^n_t - N^n_\frac{k-1}{n}
& \triangleq E \left[ - \int_\frac{k-1}{n}^t b(\widetilde{\Phi}^n_s) ds \big| \mathcal{F}^n_\frac{k-1}{n} \right]
\\& = \left(\frac{k-1}{n} - t \right) \lambda \left(1 + \widetilde{\Phi}^n_\frac{k-1}{n} \right),
\end{split}
\]
using as before the forms of $b(x)$ and $\widetilde{\Phi}^n_t$ in between grid points.

We have now defined $N^n_t$ for all $ t \geq 0$, and by construction, $(b)$ is satisfied. The goal now is to show that Condition $(e)$ is satisfied by the $N^n$'s as $n \rightarrow \infty$. We start with some observations about the process $N^n$. Recall the fixed $L > 0$ from Proposition \ref{ald1}. We let $k_{max} = k_{max}(n) \triangleq Ln$.
\begin{itemize}
\item[(1)] For all $n$, the sequence $\overline{N}^n \triangleq \{N^n_0, N^n_\frac{1}{n}, N^n_\frac{2}{n},\ldots\}$, is increasing: Note that from \eqref{eq2}
\[
N^n_\frac{k}{n} - N^n_\frac{k-1}{n} = \left(\frac{\lambda^2}{n^2} + \frac{\lambda^3}{n^3} + \cdots \right)\left(\widetilde{\Phi}^n_\frac{k-1}{n} + 1 \right) \geq 0.
\]
As a consequence of this fact, the maximum of $E \left[\overline{N}^n_T \right]$ over all of its stopping times $T$ is equal to $E \left[\overline{N}^n_\frac{k_{max}}{n} \right]$.
\item[(2)] For any $k$, $N^n_t - N^n_\frac{k-1}{n} = \left(\frac{k-1}{n} - t \right)\lambda \left( 1 + \widetilde{\Phi}^n_\frac{k-1}{n} \right)$ for $\frac{k-1}{n} \leq t < \frac{k}{n}$. This is a negative term whose magnitude is maximized when $t$ approaches $\frac{k}{n}$. Furthermore, $N^n_{\frac{k}{n}-} - N^n_\frac{k-1}{n} = -\frac{\lambda}{n}\left( 1 + \widetilde{\Phi}^n_\frac{k-1}{n} \right)$. Since $N^n_\frac{k}{n}$ is always nonnegative, this implies that 
\begin{equation}\label{infeq}
\underset{0 \leq t \leq L}{\inf} \ N^n_t \geq \underset{1 \leq k \leq k_{max} }{\min} -\frac{\lambda}{n}\left(1 + \widetilde{\Phi}^n_\frac{k-1}{n} \right).
\end{equation}
\end{itemize}

Note that $\left\{ \widetilde{\Phi}^n_\frac{k}{n} : 0 \leq k \leq k_{max} \right\}$ is a submartingale, as seen in \eqref{eq1.1}, and so its negative is a supermartingale. Therefore, by Doob's inequality,
\begin{equation}\label{doobineq}
E \left[ \left(\underset{1 \leq k \leq k_{max} }{\min} -\frac{\lambda}{n}\left(1 + \widetilde{\Phi}^n_\frac{k-1}{n} \right) \right)^2 \right] \leq 2 \frac{\lambda^2}{n^2} E \left[ \left(1 + \widetilde{\Phi}^n_\frac{k_{max}}{n} \right)^2 \right].
\end{equation}

We iteratively apply Corollary \ref{cor2} to estimate the size of $\widetilde{\Phi}^n_\frac{k}{n}$, deducing that
\begin{equation}\label{eq3}
E \left[ \left(\widetilde{\Phi}^n_\frac{k}{n} \right)^2 \right] \ll (\phi^2 + \phi)e^{\frac{k}{n}(2 \lambda + \alpha^2)}.
\end{equation}
In particular, for all $n$ and for $k$ less than or equal to $k_{max}(n) = Ln$ and fixed $\phi$, the above quantity is uniformly bounded over all $k$. It follows from \eqref{infeq}, \eqref{doobineq}, and \eqref{eq3}, therefore, that $|| \underset{0 \leq t \leq L}{\inf} N^n_t||_{L^2} = O \left(\frac{1}{n} \right)$. We find, therefore, that
\[
\underset{T \in \mathcal{T}^n_L}{\sup} E \left[ \left( N^n_T \right)^2 \right] = E \left[\left(N^n_\frac{k_{max}}{n} \right)^2 \right] + O \left(\frac{1}{n^2} \right).
\]

It remains to control the size of this last term. Recall from \eqref{eq2} that $N^n_\frac{k}{n} - N^n_\frac{k-1}{n} = O \left(\frac{1}{n^2} \right) \left(\widetilde{\Phi}^n_\frac{k-1}{n} + 1 \right)$. Iterating this formula over $1 \leq k \leq k_{max}$, we have
\[
\begin{split}
N^n_\frac{k_{max}}{n} & = O \left(\frac{1}{n^2} \right) \sum_{k = 1}^{k_{max}} \left(\widetilde{\Phi}^n_\frac{k-1}{n} + 1 \right) \\
& = O \left( \frac{1}{n^2} \right) \sum_{k = 1}^{k_{max}} \left(\widetilde{\Phi}^n_\frac{k-1}{n} \right) + O \left(\frac{1}{n} \right),
\end{split}
\]
using $k_{max} = Ln$. Squaring both sides of this equation and taking expectations, we deduce by \eqref{eq3} and H\"{o}lder's Inequality that
\begin{eqnarray*}
\lefteqn{E \left[ \left(N^n_\frac{k_{max}}{n} \right)^2 \right]} \\
&& = O \left(\frac{1}{n^4} \right) (k_{max}(n))^2 e^{2L(2 \lambda + \alpha^2)} + O \left(\frac{1}{n^3} \right) k_{max}(n) e^{L(2 \lambda + \alpha^2)} + O \left(\frac{1}{n^2} \right) \\
&& = O \left(\frac{1}{n^4} \right) (Ln)^2 e^{2L(2 \lambda + \alpha^2)} + O \left(\frac{1}{n^3} \right) (Ln) e^{L(2 \lambda + \alpha^2)} + O \left(\frac{1}{n^2} \right) \\
&& = O \left(\frac{1}{n^2} \right).
\end{eqnarray*}
This establishes Condition $(e)$.
\end{proof}

We next address Conditions $(d)$ and $(f)$.

\begin{proof}[Conditions $(d)$ and $(f)$]

First, we will construct the process $\mathcal{N}^n_t$. The procedure mimicks the one in the construction of $N^n_t$. We write
\begin{equation}\label{eq3a}
\begin{split}
S^n_\frac{k}{n} - S^n_\frac{k-1}{n}
& = \left(M^n_\frac{k}{n} \right)^2 - \left(M^n_\frac{k}{n} \right)^2 - \int_\frac{k-1}{n}^\frac{k}{n} a(\widetilde{\Phi}^n_s) ds
\\& = \left(M^n_\frac{k}{n} \right)^2 - \left(M^n_\frac{k}{n} \right)^2 - \frac{\alpha^2}{n} \left(\widetilde{\Phi}^n_\frac{k-1}{n} \right)^2,
\end{split}
\end{equation}
using the fact that $a(x) = \alpha^2 x^2$, and that $\widetilde{\Phi}^n_s = \widetilde{\Phi}^n_\frac{k-1}{n}$ for $\frac{k-1}{n} \leq s < \frac{k}{n}$. We set $\mathcal{N}^n_0 = 0$. We first define $\mathcal{N}^n$ on the grid points $\{ \frac{1}{n}, \frac{2}{n}, \ldots \}$. Define
\begin{equation}\label{voleq1}
\begin{split}
\mathcal{N}^n_\frac{k}{n} - \mathcal{N}^n_\frac{k-1}{n}
& \triangleq E \left[ S^n_\frac{k}{n} - S^n_\frac{k-1}{n} \big| \mathcal{F}^n_\frac{k-1}{n} \right]
\\& = E \left[ \left(\widetilde{\Phi}^n_\frac{k}{n} - \int_0^\frac{k}{n} b(\widetilde{\Phi}^n_s) ds - N^n_\frac{k}{n} \right)^2 \big| \mathcal{F}^n_\frac{k-1}{n} \right]
\\& \ \ \ \ \ - E \left[ \left( \widetilde{\Phi}^n_\frac{k-1}{n} - \int_0^\frac{k-1}{n} b(\widetilde{\Phi}^n_s) ds - N^n_\frac{k-1}{n} \right)^2 \big| \mathcal{F}^n_\frac{k-1}{n} \right]
\\& \ \ \ \ \ - \frac{\alpha^2}{n} \left( \widetilde{\Phi}^n_\frac{k-1}{n} \right)^2.
\end{split}
\end{equation}

We can expand the first term in \eqref{voleq1} as
{\small{
\[
E \Bigg[ \Bigg( \widetilde{\Phi}^n_\frac{k-1}{n} + \left(\widetilde{\Phi}^n_\frac{k}{n} - \widetilde{\Phi}^n_\frac{k-1}{n} \right) - \int_0^\frac{k-1}{n} b(\widetilde{\Phi}^n_s) ds - \int_\frac{k-1}{n}^\frac{k}{n} b(\widetilde{\Phi}^n_s) ds  - N^n_\frac{k-1}{n} - \left(N^n_\frac{k}{n} - N^n_\frac{k-1}{n} \right) \Bigg)^2 \big| \mathcal{F}^n_\frac{k-1}{n} \Bigg],
\]
}}
or
{\small{
\[
E \Bigg[ \Bigg( \left(\widetilde{\Phi}^n_\frac{k-1}{n} - \int_0^\frac{k-1}{n} b(\widetilde{\Phi}^n_s) ds - N^n_\frac{k-1}{n} \right) + \left(\widetilde{\Phi}^n_\frac{k}{n} - \widetilde{\Phi}^n_\frac{k-1}{n} \right) - \int_\frac{k-1}{n}^\frac{k}{n} b(\widetilde{\Phi}^n_s) ds - \left(N^n_\frac{k}{n} - N^n_\frac{k-1}{n} \right) \Bigg)^2 \big| \mathcal{F}^n_\frac{k-1}{n} \Bigg].
\]
}}

This is designed to cancel with the second term in \eqref{voleq1}. After some algebraic manipulation, we arrive at

{\small{
\begin{flalign*}
& \mathcal{N}^n_\frac{k}{n} - \mathcal{N}^n_\frac{k-1}{n} & \\
& = 2 \left(\widetilde{\Phi}^n_\frac{k-1}{n} - \int_0^\frac{k-1}{n} b(\widetilde{\Phi}^n_s) ds - N^n_\frac{k-1}{n} \right) E \left[ \left(\widetilde{\Phi}^n_\frac{k}{n} - \widetilde{\Phi}^n_\frac{k-1}{n} \right) - \int_\frac{k-1}{n}^\frac{k}{n} b(\widetilde{\Phi}^n_s) ds - \left(N^n_\frac{k}{n} - N^n_\frac{k-1}{n} \right) \big| \mathcal{F}^n_\frac{k-1}{n} \right] & \nonumber \\
& \ \ \ \ \ + E \left[ \left( \left(\widetilde{\Phi}^n_\frac{k}{n} - \widetilde{\Phi}^n_\frac{k-1}{n} \right) - \int_\frac{k-1}{n}^\frac{k}{n} b(\widetilde{\Phi}^n_s) ds - \left(N^n_\frac{k}{n} - N^n_\frac{k-1}{n} \right) \right)^2 \big| \mathcal{F}^n_\frac{k-1}{n} \right] & \nonumber \\
& \ \ \ \ \ - \frac{\alpha^2}{n} \left( \widetilde{\Phi}^n_\frac{k-1}{n} \right)^2. & \nonumber
\end{flalign*}
}}
In the first term of the right hand side above, the conditional expectation is zero, by definition of $N^n$. So we have\

{\small{
\[
\begin{split}
\mathcal{N}^n_\frac{k}{n} - \mathcal{N}^n_\frac{k-1}{n}
&= E \left[ \left( \left(\widetilde{\Phi}^n_\frac{k}{n} - \widetilde{\Phi}^n_\frac{k-1}{n} \right) - \int_\frac{k-1}{n}^\frac{k}{n} b(\widetilde{\Phi}^n_s) ds - \left(N^n_\frac{k}{n} - N^n_\frac{k-1}{n} \right) \right)^2 \big| \mathcal{F}^n_\frac{k-1}{n} \right] - \frac{\alpha^2}{n} \left( \widetilde{\Phi}^n_\frac{k-1}{n} \right)^2
\\& = E \left[ \left( \widetilde{\Phi}^n_\frac{k}{n} - \left(\widetilde{\Phi}^n_\frac{k-1}{n} + \frac{\lambda}{n}\left(1 + \widetilde{\Phi}^n_\frac{k-1}{n} \right) + \left(N^n_\frac{k}{n} - N^n_\frac{k-1}{n} \right) \right) \right)^2 \big| \mathcal{F}^n_\frac{k-1}{n} \right] - \frac{\alpha^2}{n} \left( \widetilde{\Phi}^n_\frac{k-1}{n} \right)^2.
\end{split}
\]
}}
We will now expand this term above. Recall the useful facts that $E \left[ \widetilde{\Phi}^n_\frac{k}{n} | \mathcal{F}^n_\frac{k-1}{n} \right] = e^{\frac{\lambda}{n}}\left(\widetilde{\Phi}^n_\frac{k-1}{n} + 1 \right) - 1$, and that $N^n_\frac{k}{n}$ is $\mathcal{F}^n_\frac{k-1}{n}$-measurable. We have
{\small{
\begin{equation}\label{eq4}
\begin{split}
\mathcal{N}^n_\frac{k}{n} - \mathcal{N}^n_\frac{k-1}{n}
& = E \left[ \left( \widetilde{\Phi}^n_\frac{k}{n} \right)^2 \big| \mathcal{F}^n_\frac{k-1}{n} \right] - 2 \left(\widetilde{\Phi}^n_\frac{k-1}{n} + \frac{\lambda}{n}\left(1 + \widetilde{\Phi}^n_\frac{k-1}{n} \right) + \left(N^n_\frac{k}{n} - N^n_\frac{k-1}{n} \right) \right) E \left[ \widetilde{\Phi}^n_\frac{k}{n} | \mathcal{F}^n_\frac{k-1}{n} \right]
\\& \ \ \ \ \ + \left(\widetilde{\Phi}^n_\frac{k-1}{n} + \frac{\lambda}{n}\left(1 + \widetilde{\Phi}^n_\frac{k-1}{n} \right) + \left(N^n_\frac{k}{n} - N^n_\frac{k-1}{n} \right) \right)^2 - \frac{\alpha^2}{n} \left( \widetilde{\Phi}^n_\frac{k-1}{n} \right)^2.
\end{split}
\end{equation}
}}
The second term on the right hand side of \eqref{eq4} is equal to
\begin{equation}\label{voleq3}
- 2 \left(\widetilde{\Phi}^n_\frac{k-1}{n} + \frac{\lambda}{n}\left(1 + \widetilde{\Phi}^n_\frac{k-1}{n} \right) + \left(N^n_\frac{k}{n} - N^n_\frac{k-1}{n} \right) \right) \left( e^{\frac{\lambda}{n}} \left( \widetilde{\Phi}^n_\frac{k-1}{n} + 1 \right) - 1 \right).
\end{equation}
According to Equations \eqref{eq2} and \eqref{eq3}, $\left(N^n_\frac{k}{n} - N^n_\frac{k-1}{n} \right) = O \left(\frac{1}{n^2} \right) \left(\widetilde{\Phi}^n_\frac{k-1}{n} + 1 \right)$, uniformly over all $k \leq k_{max}(n)$. Then \eqref{voleq3} becomes
\[
- 2 \left[\widetilde{\Phi}^n_\frac{k-1}{n} + \frac{\lambda}{n}\left(1 + \widetilde{\Phi}^n_\frac{k-1}{n} \right) + O\left(\frac{1}{n^2}\right) \left(\widetilde{\Phi}^n_\frac{k-1}{n} + 1 \right) \right] \left[ e^{\frac{\lambda}{n}} \left( \widetilde{\Phi}^n_\frac{k-1}{n} + 1 \right) - 1 \right]
\]
{\scriptsize{
\[
= - 2 \left[\widetilde{\Phi}^n_\frac{k-1}{n} + \frac{\lambda}{n}\left(1 + \widetilde{\Phi}^n_\frac{k-1}{n} \right) + O\left(\frac{1}{n^2}\right) \left(\widetilde{\Phi}^n_\frac{k-1}{n} + 1 \right) \right] \left[ \widetilde{\Phi}^n_\frac{k-1}{n} + \frac{\lambda}{n} \left(\widetilde{\Phi}^n_\frac{k-1}{n} + 1 \right) + O \left(\frac{1}{n^2} \right) \left(\widetilde{\Phi}^n_\frac{k-1}{n} + 1 \right)\right].
\]
}}

As we can see, this term partially cancels with the third term in Equation \eqref{eq4}. In doing so, we obtain

{\small{
\begin{equation}\label{eq5}
\begin{split}
\mathcal{N}^n_\frac{k}{n} - \mathcal{N}^n_\frac{k-1}{n} 
&= E \left[ \left( \widetilde{\Phi}^n_\frac{k}{n} \right)^2 \big| \mathcal{F}^n_\frac{k-1}{n} \right] - \left(\widetilde{\Phi}^n_\frac{k-1}{n} + \frac{\lambda}{n} \left( 1 + \widetilde{\Phi}^n_\frac{k-1}{n} \right) \right)^2 - \frac{\alpha^2}{n} \left( \widetilde{\Phi}^n_\frac{k-1}{n} \right)^2 + O \left( \frac{1}{n^2} \right) \left(\widetilde{\Phi}^n_\frac{k-1}{n} \right)^2
\\& = E \left[ \left( \widetilde{\Phi}^n_\frac{k}{n} \right)^2 \big| \mathcal{F}^n_\frac{k-1}{n} \right] - \left(\widetilde{\Phi}^n_\frac{k-1}{n} \right)^2 \left(1 + \frac{2\lambda}{n} \right) - \frac{2\lambda}{n} \widetilde{\Phi}^n_\frac{k-1}{n} + O \left(\frac{1}{n^2} \right) \left( \left(\widetilde{\Phi}^n_\frac{k-1}{n} \right)^2 + \widetilde{\Phi}^n_\frac{k-1}{n} + 1 \right).
\end{split}
\end{equation}
}}
From Corollary \ref{cor3}, we have
{\small{
\[
E \left[ \left(\widetilde{\Phi}^n_{\frac{k}{n}} \right)^2 \big| \mathcal{F}^n_{\frac{k-1}{n}} \right] = \left(\widetilde{\Phi}^n_{\frac{k-1}{n}} \right)^2 \left(1 + \frac{2 \lambda}{n} + \frac{\alpha^2}{n} + O \left(\frac{1}{n^2} \right) \right) + \widetilde{\Phi}^n_{\frac{k-1}{n}}\left(\frac{2 \lambda}{n} + O \left(\frac{1}{n^2}\right) \right) + O \left(\frac{1}{n^2} \right),
\]
}}
and this perfectly cancels out with the second and third terms in \eqref{eq5}. We are ultimately left with
\begin{equation}\label{eq5b}
\mathcal{N}^n_\frac{k}{n} - \mathcal{N}^n_\frac{k-1}{n} = O \left( \frac{1}{n^2} \right) \left( \left( \widetilde{\Phi}^n_\frac{k-1}{n} \right)^2 + \widetilde{\Phi}^n_\frac{k-1}{n} + 1 \right).
\end{equation}

Before pursuing this line of reasoning further, let us define $\mathcal{N}^n_t$ between grid points.  For $\frac{k-1}{n} \leq t < \frac{k}{n}$,
\[
\begin{split}
\mathcal{N}^n_t - \mathcal{N}^n_\frac{k-1}{n} 
&\triangleq E \left[ -\int_\frac{k-1}{n}^t a(\widetilde{\Phi}^n_s) ds | \mathcal{F}^n_\frac{k-1}{n} \right]
\\& = \left( \frac{k-1}{n} - t \right) \alpha^2 \left(\widetilde{\Phi}^n_\frac{k-1}{n} \right)^2,
\end{split}
\]
using $a(x) = \alpha^2 x^2$ as well as the the fact that $\widetilde{\Phi}^n_s = \widetilde{\Phi}^n_\frac{k-1}{n}$ for $\frac{k-1}{n} \leq s < \frac{k}{n}$.  As in the proof of Parts $(c)$ and $(e)$, we may show, using the submartingality of $\left\{\widetilde{\Phi}^n_\frac{k}{n} : k = 0,1,\ldots \right\}$, that the $L^1$ norm $\left|\left|\underset{0 \leq k \leq k_{max}(n)}{\inf} -\frac{\alpha^2}{n} \left( \widetilde{\Phi}^n_\frac{k-1}{n} \right)^2 \right|\right|_{L^1} \ll \frac{1}{n} ||\widetilde{\Phi}^n_\frac{k_{max}}{n}||_{L^2}$, which is $O \left( \frac{1}{n} \right)$.  Therefore, in attempting to establish $(e)$, we may ignore any possible times in between grid points $\{0, \frac{1}{n},\ldots\}$.  Consequently, we consider the discrete-time process
\[
\overline{\mathcal{N}}^N \triangleq \{\mathcal{N}^n_0, \mathcal{N}^n_\frac{1}{n}, \ldots, \mathcal{N}^n_\frac{k_{max}}{n} \},
\]
with associated bounded stopping times $\overline{\mathcal{T}}^N_L$, and we must show
\[
\underset{\tau \in \overline{\mathcal{T}}^N_L}{\sup} \ E[|\mathcal{N}^N_T|] \rightarrow 0 \mbox{ as } n \rightarrow \infty.
\]

According to \eqref{eq5b}, for any $k$,
\[
\left|\mathcal{N}^n_\frac{k}{n}\right| = O \left(\frac{1}{n^2} \right) \sum_{j=1}^k \left( \left(\widetilde{\Phi}^n_\frac{k-1}{n} \right)^2 + \widetilde{\Phi}^n_\frac{k-1}{n} + 1 \right).
\]

Therefore, noting that $\widetilde{\Phi}^n_\frac{k}{n}$ is always nonnegative
\[
\underset{0 \leq k \leq k_{max}}{\sup} \left|\mathcal{N}^n_\frac{k}{n} \right| = O \left( \frac{1}{n^2} \right) \sum_{j=1}^{k_{max}} \left( \left(\widetilde{\Phi}^n_\frac{k-1}{n} \right)^2 + \widetilde{\Phi}^n_\frac{k-1}{n} + 1 \right).
\]

Now, since $\left\{\widetilde{\Phi}^n_\frac{k}{n}\right\}_{k \geq 0}$ is also a submartingale and $k_{max} = Ln$, it follows that 
\[
\begin{split}
E \left[ \underset{0 \leq k \leq k_{max}}{\sup} \left|\mathcal{N}^n_\frac{k}{n} \right| \right] 
& = O \left( \frac{1}{n^2} \right) E \left[ \sum_{j=1}^{k_{max}} \left( \left(\widetilde{\Phi}^n_\frac{k-1}{n} \right)^2 + \widetilde{\Phi}^n_\frac{k-1}{n} + 1 \right) \right]
\\& \leq O \left(\frac{1}{n^2} \right) (Ln) E \left[ \left(\widetilde{\Phi}^n_\frac{k_{max}}{n} \right)^2 + \widetilde{\Phi}^n_\frac{k_{max}}{n} + 1 \right]
\\& = O \left(\frac{1}{n} \right).
\end{split}
\]

\end{proof}

\begin{proof}[Condition $(b)$] First, we claim that $E \left[ \left( j(t,Z,\phi) - \phi \right)^2 \right] = O(t) \phi^2$ as $t \rightarrow 0$, for $Z$ a standard normal random variable.  Supposing that this claim is established, then applied conditionally, it entails that
\[
E \left[ \left(\widetilde{\Phi}^n_\frac{k}{n} - \widetilde{\Phi}^n_\frac{k-1}{n} \right)^2 | \mc{F}^n_\frac{k-1}{n} \right] = O \left(\frac{1}{n} \right)  \left(\widetilde{\Phi}^n_\frac{k-1}{n} \right)^2,
\]
and this will establish $(b)$, given our control over the $L^2$ norm of $\underset{0 \leq k \leq Ln}{\sup} \ \left| \widetilde{\Phi}^n_\frac{k}{n} \right|$.  So, let us address the claim.

We have
\[
\begin{split}
E \left[ \left( j(t, Z, \phi) - \phi \right)^2 \right]
& = E[j(t,Z,\phi)^2] - 2\phi E[j(t,Z,\phi)] + \phi^2
\\& = \phi^2 \left(1 + O(t) \right) - 2 \phi (\phi + O(t)) + \phi^2
\\& = O(t)\phi^2;
\end{split}
\]
here, we have used Lemma \ref{lem3} for the first term, and Lemma \ref{lem1} for the second term.

\end{proof}

\section{\uppercase{Proofs from Section 4}}\label{stobs_proof}

\begin{proof}[Proof of Proposition \ref{prop2.2}]  The proof will be by backwards induction on $j$.  For reference, the reader should see Figures \ref{st_scheme} and \ref{st_comp}: we will use backwards induction on the columns in Figure \ref{st_scheme}, and in each column, we will use backwards induction on the elements of the column.  First, the base case, where we will prove equality.  Suppose that $j=n$ and $0 \leq k \leq n$.  Let $\Psi^k =\{\psi^k_1,\ldots,\psi^k_k\} \in \bm{\mathfrak{O}}^k$.  As mentioned before, the Poisson process $N$ is assumed to be stopped at $\eta_n$, so that $\eta_{n+1} = \infty$.  Since all observation rights have arrived by the time $\eta_j$ when $j=n$, we have remaining a total of $n-k$ observation rights, with no restrictions on when they may be used.  Therefore, following the proof of Proposition \ref{dpp1} with slight modifications, we may establish that for each $0 \leq k \leq n$, $\bm{\gamma}^n_{n,k}(\Psi^k) = \bm{v}^n_{n,k}\left(\psi^k_k \vee \eta_n - \psi^k_k, \Phi^{\Psi^k}_{\psi^k_k \vee \eta_n} \right)$.  

We next tackle the inductive step.  Suppose then that 
\[
\bm{\gamma}^n_{j+1,k}(\Psi'^k) \geq \bm{v}^n_{j+1,k} \left(\psi'^k_k \vee \eta_{j+1} - \psi'^k_k,\Phi^{\Psi'^k}_{\psi'^k_k \vee \eta_{j+1}} \right)
\]
\noindent holds for all $0 \leq k \leq j + 1$ and any $\Psi'^k \in \bm{\mathfrak{O}}^k$, on the set $\{\psi'^k_k < \eta_{(j+1)+1}\}$.  We wish to show that $\bm{\gamma}^n_{j,k}(\Psi^k) \geq \bm{v}^n_{j,k} \left(\psi^k_k \vee \eta_j - \psi^k_k, \Phi^{\Psi^k}_{\psi^k_k \vee \eta_j} \right)$ holds for all $0 \leq k \leq j$ and any $\Psi^k \in \bm{\mathfrak{O}}^k$, on the set $\{\psi^k_k < \eta_{j+1} \}$.  To establish this, we will proceed with another round of backwards induction, this time on $k$, starting from $k = j$.  So, we fix $\Psi^j = \{\psi^j_1,\ldots,\psi^j_j\} \in \bm{\mathfrak{O}}^j$.  Let $\widetilde{\Psi} = \{\psi^j_1,\ldots,\psi^j_j,\widetilde{\psi}_{j+1},\ldots,\widetilde{\psi}_n\} \in \bm{\mathfrak{O}}^n_{j,j}(\Psi^j)$.  Note that since $\widetilde{\Psi} \in \bm{\mathfrak{O}}^n$, it is the case that $\psi^j_j \geq \eta_j$, so that $\psi^j_j \vee \eta_j = \psi^j_j$.  We have, with all arguments taking place on the set $\{\psi^j_j < \eta_{j+1}\}$,
\begin{eqnarray}\label{prop2.2.1}
\lefteqn{\underset{\tau \in \mc{T}^{\widetilde{\Psi}}_s, \tau \geq \psi^j_j}{\ei} E \left[ \int_{\psi^j_j}^\tau e^{-\lambda(t - \psi^j_j)} \left( \Phi^{\widetilde{\Psi}}_t - \frac{\lambda}{c} \right) dt \big| \mc{F}^{\Psi^j}_{\psi^j_j} \right]} \\
&& = \underset{\tau \in \mc{T}^{\widetilde{\Psi}}_s, \tau \geq \psi^j_j}{\ei} E \Bigg[ \int_{\psi^j_j}^{\eta_{j+1} \wedge \tau} e^{-\lambda(t - \psi^j_j)} \left( \Phi^{\widetilde{\Psi}}_t - \frac{\lambda}{c} \right) dt \nonumber \\
&& \ \ \ \ \  + 1_{\{\tau > \eta_{j+1} \}} \int_{\eta_{j+1}}^{\tau \vee \eta_{j+1}} e^{-\lambda(t - \psi^j_j)} \left( \Phi^{\widetilde{\Psi}}_t - \frac{\lambda}{c} \right) dt \big| \mc{F}^{\Psi^j}_{\psi^j_j} \Bigg]. \nonumber
\end{eqnarray}
For each such $\tau$ above, $\tau \vee \eta_{j+1}$ is an element of $\mc{T}^{\widetilde{\Psi}}$ which is greater than $\eta_{j+1} = \psi^j_j \vee \eta_{j+1}$.  Therefore, the right hand side of \eqref{prop2.2.1} above is greater than or equal to
\[
\underset{\tau \in \mc{T}^{\widetilde{\Psi}}_s, \tau \geq \psi^j_j}{\ei} E \left[ \int_{\psi^j_j}^{\eta_{j+1} \wedge \tau} e^{-\lambda(t - \psi^j_j)} \left( \Phi^{\widetilde{\Psi}}_t - \frac{\lambda}{c} \right) dt + 1_{\{\tau > \eta_{j+1} \}} e^{-\lambda(\eta_{j+1} - \psi^j_j)}\bm{\gamma}^n_{j+1,j}(\Psi^j) \big| \mc{F}^{\Psi^j}_{\psi^j_j} \right],
\]
which by the (initial) induction hypothesis is greater than or equal to
\begin{eqnarray}\label{prop2.2.2}
\lefteqn{\ \ \ \underset{\tau \in \mc{T}^{\widetilde{\Psi}}_s, \tau \geq \psi^j_j}{\ei} E \Bigg[ \int_{\psi^j_j}^{\eta_{j+1} \wedge \tau} e^{-\lambda(t - \psi^j_j)} \left( \Phi^{\widetilde{\Psi}}_t - \frac{\lambda}{c} \right) dt} \\
&& \ \ \ \ \ + 1_{\{\tau > \eta_{j+1} \}} e^{-\lambda(\eta_{j+1} - \psi^j_j)}\bm{v}^n_{j+1,j}\left(\eta_{j+1} - \psi^j_j, \varphi \left(\eta_{j+1} - \psi^j_j,\Phi^{\Psi^j}_{\psi^j_j} \right) \right) \big| \mc{F}^{\Psi^j}_{\psi^j_j} \Bigg]. \nonumber
\end{eqnarray}
Now, recall that $\eta_{j+1} - \psi^j_j$ is independent of $\mc{F}^{\Psi^k}_{\psi^j_j}$ and exponentially distributed with parameter $\mu$, on the set $\{\psi^j_j < \eta_{j+1}\}$.  Additionally, from Theorem $3.2$ of \cite{MR2777513}, it can be seen that for some nonnegative random random variable $R_j \in m \mc{F}^{\Psi^j}_{\psi^j_j}$, $1_{\{\tau > \eta_{j+1}\}} = 1_{\{R_j > \eta_{j+1} - \psi^j_j\}}$, and $\tau 1_{\{\tau \leq \eta_{j+1} \}} = (\psi^j_j + R_j)1_{\{R_j \leq \eta_{j+1} - \psi^j_j \}}$.  Using the deterministic dynamics of $\Phi^{\Psi^j}$ in between observations, \eqref{prop2.2.2} becomes
{\small{
\[
\underset{R_j \in m \mc{F}^{\Psi^j}_{\psi^j_j}, R_j \geq 0}{\ei} \int_0^\infty \mu e^{-\mu u} \left[ \int_0^{R_j \wedge u} e^{-\lambda t} \left( \varphi \left(t,\Phi^{\Psi^j}_{\psi^j_j} \right) - \frac{\lambda}{c} \right) dt + 1_{\{R_j > u \}} e^{-\lambda u}\bm{v}^n_{j+1,j}\left(u, \varphi \left(u,\Phi^{\Psi^j}_{\psi^j_j} \right) \right) \right] du,
\]
}}
which is equal to 
\[
\begin{split}
\underset{R_j \in m \mc{F}^{\Psi^j}_{\psi^j_j}, R_j \geq 0}{\ei}  J^0 \bm{v}^n_{j+1,j}\left(R_j, 0,\Phi^{\Psi^j}_{\psi^j_j} \right) & \geq J^0_0 \bm{v}^n_{j+1,j} \left(0, \Phi^{\Psi^j}_{\psi^j_j} \right) \\
& = \bm{v}^n_{j,j} \left(0, \Phi^{\Psi^j}_{\psi^j_j} \right).
\end{split}
\]
This implies that
\[
\underset{\tau \in \mc{T}^{\widetilde{\Psi}}_s, \tau \geq \psi^j_j}{\ei} E \left[ \int_{\psi^j_j}^\tau e^{-\lambda(t - \psi^j_j)} \left( \Phi^{\widetilde{\Psi}}_t - \frac{\lambda}{c} \right) dt \big| \mc{F}^{\Psi^k}_{\psi^j_j} \right] \geq \bm{v}^n_{j,j} \left(0,\Phi^{\Psi^j}_{\psi^j_j} \right)
\]
on the set $\{\psi^j_j < \eta_{j+1}\}$.  Taking the infimum over all $\widetilde{\Psi} \in \bm{\mathfrak{O}}^n_{j,j}(\Psi^j)$, we obtain that $\bm{\gamma}^n_{j,j}(\Psi^j) \geq \bm{v}^n_{j,j} \left(0,\Phi^{\Psi^j}_{\psi^j_j} \right)$ on $\{\psi^j_j < \eta_{j+1}\}$, which establishes the base case in the second induction.  

We now proceed to the inductive step in the second induction.  Our hypothesis is that $\bm{\gamma}^n_{j,k+1} \left(\Psi^{k+1} \right) \geq \bm{v}^n_{j,k+1} \left(\psi^{k+1}_{k+1} \vee \eta_j - \psi^{k+1}_{k+1}, \Phi^{\Psi^{k+1}}_{\psi^{k+1}_{k+1} \vee \eta_j} \right)$ holds for all $\Psi^{k+1} \in \bm{\mathfrak{O}}^{k+1}$ on the set $\left\{\psi^{k+1}_{k+1} < \eta_{j+1} \right \}$.  We wish to show that for any $\Psi^k = \{\psi^k_1,\ldots,\psi^k_k\} \in \bm{\mathfrak{O}}^k$, it is the case that
\[
\bm{\gamma}^n_{j,k}(\Psi^k) \geq \bm{v}^n_{j,k} \left(\psi^k_k \vee \eta_j - \psi^k_k,  \Phi^{\Psi^k}_{\psi^k_k \vee \eta_j} \right) \text{ on } \{\psi^k_k < \eta_{j+1}\}.
\]
So, let $\widetilde{\Psi} = \{\psi^k_1,\ldots,\psi^k_k,\widetilde{\psi}_{k+1},\ldots,\widetilde{\psi}_n\} \in \bm{\mathfrak{O}}^n_{j,k}(\Psi^k)$.  We write, on the set $\{\psi^k_k < \eta_{j+1}\},$
\begin{eqnarray*}
\lefteqn{\underset{\tau \in \mc{T}^{\widetilde{\Psi}}_s, \tau \geq \psi^k_k \vee \eta_j}{\ei} E \left[\int_{\psi^k_k \vee \eta_j}^\tau e^{-\lambda(t - \psi^k_k \vee \eta_j)} \left(\Phi^{\widetilde{\Psi}}_t - \frac{\lambda}{c} \right) \big| \mc{F}^{\widetilde{\Psi}}_{\psi^k_k \vee \eta_j} \right]} \\
&& = \min \Bigg\{0, \underset{\tau \in \mc{T}^{\widetilde{\Psi}}_s, \tau \geq \widetilde{\psi}_{k+1} \wedge \eta_{j+1}}{\ei} E \Bigg[ \int_{\psi^k_k \vee \eta_j}^{\widetilde{\psi}_{k+1} \wedge \eta_{j+1}} e^{-\lambda (t - \psi^k_k \vee \eta_j)} \left( \Phi^{\widetilde{\Psi}}_t - \frac{\lambda}{c} \right) dt  \\ && \ \ \ \ \ + \int_{\widetilde{\psi}_{k+1} \wedge \eta_{j+1}}^\tau e^{-\lambda (t - \psi^k_k \vee \eta_j)} \left( \Phi^{\widetilde{\Psi}}_t - \frac{\lambda}{c} \right) dt \big | \mc{F}^{\widetilde{\Psi}}_{\psi^k_k \vee \eta_j} \Bigg] \Bigg\},
\end{eqnarray*}
using the fact that $\tau \in \mc{T}_s^{\Psi}$, i.e. it does not stop the game between observations while there are remaining observation rights.  This then equals
\begin{eqnarray*}
\lefteqn{\min \Bigg\{0,  E \left[ \int_{\psi^k_k \vee \eta_j}^{\widetilde{\psi}_{k+1} \wedge \eta_{j+1}} e^{-\lambda (t - \psi^k_k \vee \eta_j)} \left( \Phi^{\widetilde{\Psi}}_t - \frac{\lambda}{c} \right) dt \big | \mc{F}^{\widetilde{\Psi}}_{\psi^k_k \vee \eta_j} \right]} \\
&& \ \ \ \ \  + \underset{\tau \in \mc{T}^{\widetilde{\Psi}}_s, \tau \geq \widetilde{\psi}_{k+1} \wedge \eta_{j+1}}{\ei} E \left[ \int_{\widetilde{\psi}_{k+1} \wedge \eta_{j+1}}^\tau e^{-\lambda (t - \psi^k_k \vee \eta_j)} \left( \Phi^{\widetilde{\Psi}}_t - \frac{\lambda}{c} \right) dt \big | \mc{F}^{\widetilde{\Psi}}_{\psi^k_k \vee \eta_j} \right] \Bigg\} \\
&& = \min \Bigg\{0,  E \left[ \int_{\psi^k_k \vee \eta_j}^{\widetilde{\psi}_{k+1} \wedge \eta_{j+1}} e^{-\lambda (t - \psi^k_k \vee \eta_j)} \left( \Phi^{\widetilde{\Psi}}_t - \frac{\lambda}{c} \right) dt \big | \mc{F}^{\widetilde{\Psi}}_{\psi^k_k \vee \eta_j} \right] \\
&& \ \ \ \ \  + \underset{\tau \in \mc{T}^{\widetilde{\Psi}}_s, \tau \geq \widetilde{\psi}_{k+1} \wedge \eta_{j+1}}{\ei} E \Bigg[ 1_{\{\widetilde{\psi}_{k+1} < \eta_{j+1}\}} \int_{\widetilde{\psi}_{k+1}}^\tau e^{-\lambda (t - \psi^k_k \vee \eta_j)} \left( \Phi^{\widetilde{\Psi}}_t - \frac{\lambda}{c} \right) dt  \\
&& \ \ \ \ \ + 1_{\{\eta_{j+1} \leq \widetilde{\psi}_{k+1}\}} \int_{\eta_{j+1}}^\tau e^{-\lambda (t - \psi^k_k \vee \eta_j)} \left( \Phi^{\widetilde{\Psi}}_t - \frac{\lambda}{c} \right) dt \big | \mc{F}^{\widetilde{\Psi}}_{\psi^k_k \vee \eta_j} \Bigg] \Bigg\}
\end{eqnarray*}
\begin{eqnarray}\label{prop2.2.3pair}
\lefteqn{\ \ \ \ \ \ \ \ \ \geq \min \Bigg\{0,  E \left[ \int_{\psi^k_k \vee \eta_j}^{\widetilde{\psi}_{k+1} \wedge \eta_{j+1}} e^{-\lambda (t - \psi^k_k \vee \eta_j)} \left( \Phi^{\widetilde{\Psi}}_t - \frac{\lambda}{c} \right) dt \big | \mc{F}^{\widetilde{\Psi}}_{\psi^k_k \vee \eta_j} \right]} \\
&& \ \ \ \ \  + \underset{\tau \in \mc{T}^{\widetilde{\Psi}}_s, \tau \geq \widetilde{\psi}_{k+1} \wedge \eta_{j+1}}{\ei} E \left[ 1_{\{\widetilde{\psi}_{k+1} < \eta_{j+1}\}} \int_{\widetilde{\psi}_{k+1}}^\tau e^{-\lambda (t - \psi^k_k \vee \eta_j)} \left( \Phi^{\widetilde{\Psi}}_t - \frac{\lambda}{c} \right) dt \big | \mc{F}^{\widetilde{\Psi}}_{\psi^k_k \vee \eta_j} \right] \nonumber\\
&& \ \ \ \ \ + \underset{\tau \in \mc{T}^{\widetilde{\Psi}}_s, \tau \geq \widetilde{\psi}_{k+1} \wedge \eta_{j+1}}{\ei} E \left[ 1_{\{\eta_{j+1} \leq \widetilde{\psi}_{k+1}\}} \int_{\eta_{j+1}}^\tau e^{-\lambda (t - \psi^k_k \vee \eta_j)} \left( \Phi^{\widetilde{\Psi}}_t - \frac{\lambda}{c} \right) dt \big | \mc{F}^{\widetilde{\Psi}}_{\psi^k_k \vee \eta_j} \right] \Bigg\}, \nonumber
\end{eqnarray}
which we claim is equal to
\begin{eqnarray}\label{prop2.2.3}
\lefteqn{ \ \ \ \ \ \ \ \ \ \min \Bigg\{0,  E \left[ \int_{\psi^k_k \vee \eta_j}^{\widetilde{\psi}_{k+1} \wedge \eta_{j+1}} e^{-\lambda (t - \psi^k_k \vee \eta_j)} \left( \Phi^{\widetilde{\Psi}}_t - \frac{\lambda}{c} \right) dt \big | \mc{F}^{\widetilde{\Psi}}_{\psi^k_k \vee \eta_j} \right]} \\
&& \ \ \ \ \ + \underset{\tau \in \mc{T}^{\widetilde{\Psi}}_s, \tau \geq \widetilde{\psi}_{k+1}}{\ei} E \left[ 1_{\{\widetilde{\psi}_{k+1} < \eta_{j+1}\}} \int_{\widetilde{\psi}_{k+1}}^\tau e^{-\lambda (t - \psi^k_k \vee \eta_j)} \left( \Phi^{\widetilde{\Psi}}_t - \frac{\lambda}{c} \right) dt \big | \mc{F}^{\widetilde{\Psi}}_{\psi^k_k \vee \eta_j} \right] \nonumber \\
&& \ \ \ \ \ + \underset{\tau \in \mc{T}^{\widetilde{\Psi}}_s, \tau \geq \eta_{j+1}}{\ei} E \left[ 1_{\{\eta_{j+1} \leq \widetilde{\psi}_{k+1}\}} \int_{\eta_{j+1}}^\tau e^{-\lambda (t - \psi^k_k \vee \eta_j)} \left( \Phi^{\widetilde{\Psi}}_t - \frac{\lambda}{c} \right) dt \big | \mc{F}^{\widetilde{\Psi}}_{\psi^k_k \vee \eta_j} \right] \Bigg\}. \nonumber
\end{eqnarray}

It is trivially true that $\eqref{prop2.2.3pair} \leq \eqref{prop2.2.3}$, since its infimums are over a larger set of stopping times.  We claim that the other inequality also holds.  To see this, let $\tau \geq \widetilde{\psi}_{k+1} \wedge \eta_{j+1}$ be given.  Consider $\tau' = \tau 1_{\{\widetilde{\psi}_{k+1} < \eta_{j+1}\}} + \infty 1_{\{\eta_{j+1} \leq \widetilde{\psi}_{k+1}\}}$, which is a stopping time greater than $\widetilde{\psi}_{k+1}$ and takes the same value as $\tau$ on the set $1_{\{\widetilde{\psi}_{k+1} < \eta_{j+1} \}}$.  The existence of such a $\tau'$ implies that minimizing over $\{\tau \in \mc{T}_s^{\widetilde{\Psi}} : \tau \geq \widetilde{\psi}_{k+1} \wedge \eta_{j+1} \}$ is equivalent to minimizing over $\{\tau \in \mc{T}^{\widetilde{\Psi}}_s : \tau \geq \widetilde{\psi}_{k+1} \}$, on the set $\{\widetilde{\psi}_{k+1} < \eta_{j+1} \}$.  A similar procedure may be done for stopping times on the set $1_{\{\eta_{j+1} \leq \widetilde{\psi}_{k+1}\}}$, to establish the equality of \eqref{prop2.2.3pair} and \eqref{prop2.2.3}.

Next, conditioning the second and third terms of \eqref{prop2.2.3} on, respectively, $\mc{F}^{\widetilde{\Psi}}_{\widetilde{\psi}_{k+1} \vee \eta_j}$ and $\mc{F}^{\widetilde{\Psi}}_{\widetilde{\psi}_k \vee \eta_{j+1}}$, and using the definition of $\bm{\gamma}^n_{j,k+1}(\widetilde{\Psi})$, $\bm{\gamma}^n_{j+1,k}(\widetilde{\Psi})$, we obtain
\begin{eqnarray*}
\lefteqn{\eqref{prop2.2.3}} \\
&& \geq \min \Bigg\{0,  E \left[ \int_{\psi^k_k \vee \eta_j}^{\widetilde{\psi}_{k+1} \wedge \eta_{j+1}} e^{-\lambda (t - \psi^k_k \vee \eta_j)} \left( \Phi^{\widetilde{\Psi}}_t - \frac{\lambda}{c} \right) dt \big | \mc{F}^{\widetilde{\Psi}}_{\psi^k_k \vee \eta_j} \right] \\
&& \ \ \ \ \  + E \left[ 1_{\{\widetilde{\psi}_{k+1} < \eta_{j+1}\}} e^{-\lambda(\widetilde{\psi}_{k+1} - \psi^k_k \vee \eta_j)} \bm{\gamma}^n_{j,k+1}(\widetilde{\Psi}) \big | \mc{F}^{\widetilde{\Psi}}_{\psi^k_k \vee \eta_j} \right] \\
&& \ \ \ \ \ + E \left[ 1_{\{\eta_{j+1} \leq \widetilde{\psi}_{k+1}\}} e^{-\lambda(\eta_{j+1} - \psi^k_k \vee \eta_j)} \bm{\gamma}^n_{j+1,k}(\widetilde{\Psi}) \big | \mc{F}^{\widetilde{\Psi}}_{\psi^k_k \vee \eta_j} \right] \Bigg\},
\end{eqnarray*}
where we have used the fact that $\widetilde{\psi}_{k+1} \geq \eta_j$ (so $\widetilde{\psi}_{k+1} \vee \eta_j = \widetilde{\psi}_{k+1}$) for the second term, and the fact that we are on the set $\{\psi^k_k < \eta_{j+1}\}$ for the third term, so that $\psi^k_k \vee \eta_{j+1} = \eta_{j+1}$.  Now, applying the induction hypothesis, this is greater than or equal to
\begin{eqnarray*}
\lefteqn{\min \Bigg \{0,  E \left[ \int_{\psi^k_k \vee \eta_j}^{\widetilde{\psi}_{k+1} \wedge \eta_{j+1}} e^{-\lambda (t - \psi^k_k \vee \eta_j)} \left( \Phi^{\widetilde{\Psi}}_t - \frac{\lambda}{c} \right) dt \big | \mc{F}^{\widetilde{\Psi}}_{\psi^k_k \vee \eta_j} \right]} \\
&& \ \ \ \ \  + E \left[ 1_{\{\widetilde{\psi}_{k+1} < \eta_{j+1}\}} e^{-\lambda(\widetilde{\psi}_{k+1} - \psi^k_k \vee \eta_j)} \bm{v}^n_{j,k+1}\left(\widetilde{\psi}_{k+1} - \widetilde{\psi}_{k+1}, \Phi^{\widetilde{\Psi}}_{\widetilde{\psi}_{k+1}} \right) \big | \mc{F}^{\widetilde{\Psi}}_{\psi^k_k \vee \eta_j} \right] \\
&& \ \ \ \ \ + E \left[ 1_{\{\eta_{j+1} \leq \widetilde{\psi}_{k+1}\}} e^{-\lambda(\eta_{j+1} - \psi^k_k \vee \eta_j)} \bm{v}^n_{j+1,k} \left(\psi^k_k \vee \eta_{j+1} - \psi^k_k, \Phi^{\widetilde{\Psi}}_{\psi^k_k \vee \eta_{j+1}} \right) \big | \mc{F}^{\widetilde{\Psi}}_{\psi^k_k \vee \eta_j} \right] \Bigg\},
\end{eqnarray*}
equal to
\begin{eqnarray}\label{prop2.2.4}
\lefteqn{\ \ \ \min \Bigg \{0,  E \left[ \int_{\psi^k_k \vee \eta_j}^{\widetilde{\psi}_{k+1} \wedge \eta_{j+1}} e^{-\lambda (t - \psi^k_k \vee \eta_j)} \left( \Phi^{\widetilde{\Psi}}_t - \frac{\lambda}{c} \right) dt \big | \mc{F}^{\widetilde{\Psi}}_{\psi^k_k \vee \eta_j} \right]} \\
&& \ \ \ \ \  + E \left[ 1_{\{\widetilde{\psi}_{k+1} < \eta_{j+1}\}} e^{-\lambda(\widetilde{\psi}_{k+1} - \psi^k_k \vee \eta_j)} \bm{v}^n_{j,k+1}\left(0, \Phi^{\widetilde{\Psi}}_{\widetilde{\psi}_{k+1}} \right) \big | \mc{F}^{\widetilde{\Psi}}_{\psi^k_k \vee \eta_j} \right] \nonumber \\
&& \ \ \ \ \ + E \left[ 1_{\{\eta_{j+1} \leq \widetilde{\psi}_{k+1}\}} e^{-\lambda(\eta_{j+1} - \psi^k_k \vee \eta_j)} \bm{v}^n_{j+1,k} \left(\eta_{j+1} - \psi^k_k, \Phi^{\widetilde{\Psi}}_{\eta_{j+1}} \right) \big | \mc{F}^{\widetilde{\Psi}}_{\psi^k_k \vee \eta_j} \right] \Bigg\}. \nonumber
\end{eqnarray}
Now we make some observations.  First, on the set $\{\psi^k_k < \eta_{j+1}\}$, $\eta_{j+1}-\psi^k_k \vee \eta_j$ is independent of $\mc{F}^{\widetilde{\Psi}}_{\psi^k_k \vee \eta_j}$, and conditioned on this sigma algebra, is distributed as an exponential random variable with parameter $\mu$.  Second, between $\psi^k_k \vee \eta_j$ and $\widetilde{\psi}_{k+1} \wedge \eta_{j+1}$, $\Phi^{\widetilde{\Psi}}$ has deterministic dynamics described by \eqref{stateprocess}.  Third, for some nonnegative $\mc{F}^{\widetilde{\Psi}}_{\psi^k_k \vee \eta_j}$-random variable $R_{j,k}$, $\widetilde{\psi}_{k+1} = \psi^k_k \vee \eta_j + R_{j,k}$, $1_{\{\widetilde{\psi}_{k+1} < \eta_{j+1}\}} = 1_{\{R_{j,k} < \eta_{j+1} - \psi^k_k \vee \eta_j\}}$, and $1_{\{\widetilde{\psi}_{k+1} \geq \eta_{j+1}\}} = 1_{\{R_{j,k} \geq \eta_{j+1} - \psi^k_k \vee \eta_j\}}$.  Fourth, we note that, based upon \eqref{stateprocess} and arguing as in Proposition \ref{dpp1},

\begin{eqnarray*}
\lefteqn{E \left[ \bm{v}^n_{j,k+1} \left(0,\Phi^{\widetilde{\Psi}}_{\widetilde{\psi}_{k+1}} \right) | \mc{F}^{\widetilde{\Psi}}_{\psi^k_k \vee \eta_j} \right]} \\
&&= \bm{K} \bm{v}^n_{j,k+1} \left( \widetilde{\psi}_{k+1} - \psi^k_k, \Phi^{\Psi^k}_{\psi^k_k} \right)  \\
&&= \bm{K} \bm{v}^n_{j,k+1} \left( (\widetilde{\psi}_{k+1} - \psi^k_k \vee \eta_j) + (\psi^k_k \vee \eta_j - \psi^k_k), \Phi^{\Psi^k}_{\psi^k_k} \right)  \\ 
&&=  \bm{K} \bm{v}^n_{j,k+1} \left( (\widetilde{\psi}_{k+1} - \psi^k_k \vee \eta_j) + (\psi^k_k \vee \eta_j - \psi^k_k), \varphi \left(-(\psi^k_k \vee \eta_j - \psi^k_k),\Phi^{\widetilde{\Psi}}_{\psi^k_k \vee \eta_j } \right) \right). 
\end{eqnarray*}
Therefore,
\begin{eqnarray*}
\lefteqn{\eqref{prop2.2.4}} \\
&& = \min \Bigg \{ 0, \int_0^\infty \mu e^{-\lambda u} \Bigg[ \int_0^{R_{j,k} \wedge u} e^{-\lambda t} \left( \varphi \left(t,\Phi^{\widetilde{\Psi}}_{\psi^k_k \vee \eta_j} \right) - \frac{\lambda}{c} \right)dt \\
&& \ \ \ \ \ + 1_{\{R_{j,k} < u\}} e^{-\lambda R_{j,k}} \bm{K} \bm{v}^n_{j,k+1} \left( R_{j,k} + (\psi^k_k \vee \eta_j - \psi^k_k), \varphi \left(-(\psi^k_k \vee \eta_j - \psi^k_k),\Phi^{\widetilde{\Psi}}_{\psi^k_k \vee \eta_j } \right) \right) \\
&& \ \ \ \ \ + 1_{\{u \leq R_{j,k}\}} e^{-\lambda u} \bm{v}^n_{j+1,k} \left(u + (\psi^k_k \vee \eta_j - \psi^k_k), \varphi \left(u, \Phi^{\widetilde{\Psi}}_{\psi^k_k \vee \eta_j } \right) \right) \Bigg] du \Bigg\},
\end{eqnarray*}
which is equal to 
\begin{eqnarray*}
\lefteqn{\min \left \{0, J^+ \left(\bm{v}^n_{j,k+1}, \bm{v}^n_{j+1,k} \right) \left(R_{j,k}, \psi^k_k \vee \eta_j - \psi^k_k, \Phi^{\Psi^k}_{\psi^k_k \vee \eta_j } \right) \right \}} \\
&& \geq J^+_0 \left(\bm{v}^n_{j,k+1}, \bm{v}^n_{j+1,k} \right) \left(\psi^k_k \vee \eta_j - \psi^k_k, \Phi^{\Psi^k}_{\psi^k_k \vee \eta_j } \right) \\
&&= \bm{v}^n_{j,k} \left(\psi^k_k \vee \eta_j - \psi^k_k, \Phi^{\Psi^k}_{\psi^k_k \vee \eta_j } \right).
\end{eqnarray*}
  We have shown that
\begin{eqnarray*}
\lefteqn{\underset{\tau \in \mc{T}^{\widetilde{\Psi}}_s, \tau \geq \psi^k_k \vee \eta_j}{\ei} E \left[ \int_{\psi^k_k \vee \eta_j}^\tau e^{-\lambda(t - \psi^k_k \vee \eta_j)} \left( \Phi^{\widetilde{\Psi}}_t - \frac{\lambda}{c} \right) dt \big| \mc{F}^{\Psi^k}_{\psi^k_k \vee \eta_j} \right]} \\
&&\geq \bm{v}^n_{j,k} \left(\psi^k_k \vee \eta_j - \psi^k_k, \Phi^{\Psi^k}_{\psi^k_k \vee \eta_j } \right) \text{ on } \{\psi^k_k < \eta_{j+1}\}.
\end{eqnarray*}
Thus, after taking the infimum over all $\widetilde{\Psi} \in \bm{\mathfrak{O}}^n_{j,k}(\Psi^k)$, we deduce that
\[
\bm{\gamma}^n_{j,k}(\Psi^k) \geq \bm{v}^n_{j,k} \left(\psi^k_k \vee \eta_j - \psi^k_k, \Phi^{\Psi^k}_{\psi^k_k \vee \eta_j } \right) \text{ on } \{\psi^k_k < \eta_{j+1}\},
\]
as claimed.
\end{proof}

\begin{proof}[Proof of Proposition \ref{prop2.3}] As in Proposition \ref{prop2.2}, the proof is handled by a double backwards induction, first on the number of observation rights received $j$, and then on the number of observation rights spent, $k$.  The proof of the base case, $j=n$, corresponds to times when all possible observation rights have been received, and so as before, this case is handled essentially identically as in Proposition \ref{dpp1}.  Therefore, we move on to the inductive step.  Suppose that the result has been proven for $j+1$ observation rights received; we will prove it for $j$.  Now comes a second induction, on $k$, so we will take up the base case of this, and take $k=j$.  So, let $\Psi^j \in \bm{\mathfrak{O}}^j$.  We have
\begin{eqnarray}\label{prop2.3.4}
\lefteqn{\ \ \ \ \ \ \  \bm{\gamma}^n_{j,j}(\Psi^j)} \\
&& = \underset{\Psi \in \bm{\mathfrak{O}}^n_{j,j}(\Psi^j)}{\ei} \ \underset{\tau \in \mc{T}^\Psi_s, \tau \geq \psi^j_j}{\ei} E \left[ \int_{\psi^j_j}^\tau e^{-\lambda (t-\psi^j_j)} \left( \Phi^\Psi_t - \frac{\lambda}{c} \right) dt \big | \mc{F}^{\Psi^j}_{\psi^j_j} \right] \nonumber \\
&& \ \ \leq E \left[ \int_{\psi^j_j}^{\widehat{\tau}^n_{j,j} \wedge \eta_{j+1}} e^{-\lambda (t-\psi^j_j)} \left( \varphi \left(s - \psi^j_j,\Phi^{\Psi^j}_{\psi^j_j} \right) - \frac{\lambda}{c} \right) ds \big | \mc{F}^{\Psi^j}_{\psi^j_j} \right] \label{prop2.3.4.b} \\
&& \ \ \ \ \ + \underset{\Psi \in \bm{\mathfrak{O}}^n_{j+1,j}(\Psi^j)}{\ei} \ \underset{\tau \in \mc{T}^\Psi_s, \tau \geq \eta_{j+1}}{\ei} E \left[ 1_{\{\widehat{\tau}^n_{j,j} \geq \eta_{j+1}\}} \int_{\eta_{j+1}}^\tau e^{-\lambda (t-\psi^j_j)} \left( \Phi^\Psi_t - \frac{\lambda}{c} \right) dt \big | \mc{F}^{\Psi^j}_{\psi^j_j} \right], \nonumber
\end{eqnarray}
where we have used the following facts: 
\begin{itemize}
\item[(a)] The $j+1^{st}$ observation can not be made prior to the arrival time $\eta_{j+1}$ of the $j+1^{st}$ arrival time, so $\Phi^\Psi$ evolves deterministically between $\psi^j_j$ and $\eta_j$ for all $\Psi \in \bm{\mathfrak{O}}^n_{j,j}(\Psi^j)$.
\item[(b)] For any $\tau \in \mc{T}^\Psi_s$ with $\tau \geq \eta_{j+1}$, we may construct the stopping time $\widetilde{\tau} = \widehat{\tau}^n_{j,j} 1_{\{\widehat{\tau}^n_{j,j} < \eta_{j+1}\}} + \tau 1_{\{\widehat{\tau}^n_{j,j} \geq \eta_{j+1}\}}$ which agrees with $\tau$ on $\{\widehat{\tau}^n_{j,j} \geq \eta_{j+1}\}$, and is an element of $\mc{T}^\Psi_s$ such that $\widetilde{\tau} \geq \psi^j_j$.  Therefore, the infimum over $\tau \geq \psi^j_j$ in \eqref{prop2.3.4} can be replaced with the infimum over $\tau \geq \eta_{j+1}$ in \eqref{prop2.3.4.b}.
\end{itemize}
After conditioning the interior of $E \left[ 1_{\{\widehat{\tau}^n_{j,j} \geq \eta_{j+1}\}} \int_{\eta_{j+1}}^\tau e^{-\lambda (t-\psi^j_j)} \left( \Phi^\Psi_t - \frac{\lambda}{c} \right) dt \big | \mc{F}^{\Psi^j}_{\psi^j_j} \right]$ on $\mc{F}^{\Psi^j}_{\eta_{j+1}}$, we see that \eqref{prop2.3.4.b} is equal, by definition, to
\[
E \left[ \int_{\psi^j_j}^{\widehat{\tau}^n_{j,j} \wedge \eta_{j+1}} e^{-\lambda (t-\psi^j_j)} \left( \varphi \left(s - \psi^j_j,\Phi^{\Psi^j}_{\psi^j_j} \right) - \frac{\lambda}{c} \right) ds + 1_{\{\widehat{\tau}^n_{j,j} \geq \eta_{j+1}\}} \bm{\gamma}^n_{j+1,j}(\Psi^j) \big | \mc{F}^{\Psi^j}_{\psi^j_j} \right],
\]
which by the induction hypothesis is equal to
\begin{multline*}
E \Bigg[ \int_{\psi^j_j}^{\widehat{\tau}^n_{j,j} \wedge \eta_{j+1}} e^{-\lambda (t-\psi^j_j)} \left( \varphi \left(s - \psi^j_j,\Phi^{\Psi^j}_{\psi^j_j} \right) - \frac{\lambda}{c} \right) ds 
\\+ 1_{\{\widehat{\tau}^n_{j,j} \geq \eta_{j+1}\}} \bm{v}^n_{j+1,j} \left(\eta_{j+1} - \psi^j_j, \varphi \left(\eta_{j+1} - \psi^j_j, \Phi^{\Psi^j}_{\psi^j_j} \right) \right) \big | \mc{F}^{\Psi^j}_{\psi^j_j} \Bigg],
\end{multline*}
which is equal to
\begin{multline}\label{prop2.3.5}
\int_0^\infty \mu e^{-\mu u} \Bigg[ \int_0^{s^n_{j,j}\left(\Phi^{\Psi^j}_{\psi^j_j}\right) \wedge u} e^{-\lambda s} \left( \varphi \left(s,\Phi^{\Psi^j}_{\psi^j_j} \right) - \frac{\lambda}{c} \right) ds 
\\+ 1_{\left\{s^n_{j,j}\left(\Phi^{\Psi^j}_{\psi^j_j}\right) \geq u \right \}} \bm{v}^n_{j+1,j}\left(u, \varphi \left(u,\Phi^{\Psi^j}_{\psi^j_j} \right) \right) \Bigg] du,
\end{multline}
as $\widehat{\tau}^n_{j,j} = \psi^j_j + s^n_{j,j}\left(\Phi^{\Psi^j}_{\psi^j_j}\right)$, and using the fact that on the set $\{\psi^j_j < \eta_{j+1}\}$, $\eta_{j+1} - \psi^j_j$ is independent of $\mc{F}^{\Psi^j}_{\psi^j_j}$, and distributed as an exponential random variable with parameter $\mu$.

Now, \eqref{prop2.3.5} is equal to $J^0 \bm{v}^n_{j+1,j}\left(s^n_{j,j}\left(\Phi^{\Psi^j}_{\psi^j_j}\right),0,\Phi^{\Psi^j}_{\psi^j_j} \right)$, which by construction of $s^n_{j,j}$, is equal to $J^0_0 \bm{v}^n_{j+1,j} \left(0,\Phi^{\Psi^j}_{\psi^j_j} \right) = \bm{v}^n_{j,j} \left(0,\Phi^{\Psi^j}_{\psi^j_j} \right)$.

Thus, we have established that $\bm{\gamma}^n_{j,j}(\Psi^j) \leq \bm{v}^n_{j,j} \left(0,\Phi^{\Psi^j}_{\psi^j_j} \right)$ on the set $\{\psi^j_j < \eta_{j+1}\}$.  In light of Proposition \ref{prop2.2}, these quantities are actually equal.  Furthermore, since $\bm{v}^n_{j,j} \left(0,\Phi^{\Psi^j}_{\psi^j_j} \right) \leq \bm{\gamma}^n_{j,j}(\Psi^j) \leq \eqref{prop2.3.4} = \bm{v}^n_{j,j} \left(0,\Phi^{\Psi^j}_{\psi^j_j} \right)$, \eqref{prop2.3.2} is established.

We now proceed to the second inductive step.  Supposing that the result has been proven for $j$ observation rights received and $k+1 \leq j$ observation rights spent, we will prove the result for $j$ observation rights received and $k$ observation rights spent.  So, let $\Psi^k \in \bm{\mathfrak{O}}^k$, and let $\widehat{\Psi}^{k+1} = \{\psi^k_1,\ldots,\psi^k_k,\widehat{\tau}^n_{j,k}\} \in \bm{\mathfrak{O}}^{k+1}_{j,k}(\Psi^k)$, with $\widehat{\tau}^n_{j,k} = \widehat{\tau}^n_{j,k}(\Psi^k)$.  We have, on the set $\{\psi^k_k < \eta_{j+1}\}$,

{\small{
\begin{eqnarray}\label{prop2.3.6}
\lefteqn{\ \ \ \ \ \ \ \ \ \ \ \ \ \ \bm{\gamma}^n_{j,k}(\Psi^k)} \\
&& = \underset{\Psi \in \bm{\mathfrak{O}}^n_{j,k}(\Psi^k)}{\ei} \ \underset{\tau \in \mc{T}^\Psi_s, \tau \geq \psi^k_k \vee \eta_j}{\ei} E \left[ \int_{\psi^k_k \vee \eta_j}^\tau e^{-\lambda (t-\psi^k_k \vee \eta_j)} \left( \Phi^\Psi_t - \frac{\lambda}{c} \right) dt \big | \mc{F}^{\Psi^k}_{\psi^k_k \vee \eta_j} \right] \nonumber \\
&& \leq E \left[ \int_{\psi^k_k \vee \eta_j}^{\widehat{\tau}^n_{j,k} \wedge \eta_{j+1}} e^{-\lambda (s - \psi^k_k \vee \eta_j)} \left (\varphi \left(s - \psi^k_k \vee \eta_j,\Phi^{\Psi^k}_{\psi^k_k \vee \eta_j} \right) - \frac{\lambda}{c} \right) ds \big| \mc{F}^{\Psi^k}_{\psi^k_k \vee \eta_j} \right] \nonumber \\
&& \ \ \ \ \ +  \underset{\Psi \in \bm{\mathfrak{O}}^n_{j,k+1}(\widehat{\Psi}^{k+1})}{\ei} \ \underset{\tau \in \mc{T}^\Psi_s, \tau \geq \widehat{\tau}^n_{j,k}}{\ei} E \left[ e^{-\lambda(\widehat{\tau}^n_{j,k} - \psi^k_k \vee \eta_j)} 1_{\{\widehat{\tau}^n_{j,k} < \eta_{j+1}\}}  \int_{\widehat{\tau}^n_{j,k}}^\tau e^{-\lambda (t-\widehat{\tau}^n_{j,k})} \left( \Phi^\Psi_t - \frac{\lambda}{c} \right) dt \big | \mc{F}^{\Psi^k}_{\psi^k_k \vee \eta_j} \right] \nonumber \\ 
&& \ \ \ \ \ +  \underset{\Psi \in \bm{\mathfrak{O}}^n_{j+1,k}(\Psi^k)}{\ei} \ \underset{\tau \in \mc{T}^\Psi_s, \tau \geq \eta_{j+1}}{\ei} E \left[e^{-\lambda(\eta_{j+1} - \psi^k_k \vee \eta_j)} 1_{\{\widehat{\tau}^n_{j,k} \geq \eta_{j+1}\}} \int_{\eta_{j+1}}^\tau e^{-\lambda (t-\eta_{j+1})} \left( \Phi^\Psi_t - \frac{\lambda}{c} \right) dt \big | \mc{F}^{\Psi^k}_{\psi^k_k \vee \eta_j} \right], \nonumber
\end{eqnarray}
}}
where we used in the first line above the deterministic evolution of $\Phi^{\Psi^k}$ in between observations, and for the second and third lines, the structure of stopping times in this problem, arguing as in $(b)$ above.  Now, \eqref{prop2.3.6} above is equal to

{\small{
\begin{eqnarray*}
\lefteqn{E \Bigg[ \int_{\psi^k_k \vee \eta_j}^{\widehat{\tau}^n_{j,k} \wedge \eta_{j+1}} e^{-\lambda (s - \psi^k_k \vee \eta_j)} \left (\varphi \left(s - \psi^k_k \vee \eta_j,\Phi^{\Psi^k}_{\psi^k_k \vee \eta_j} \right) - \frac{\lambda}{c} \right) ds} \\
&&+ e^{-\lambda(\widehat{\tau}^n_{j,k} - \psi^k_k \vee \eta_j)} 1_{\{\widehat{\tau}^n_{j,k} < \eta_{j+1}\}} \bm{\gamma}^n_{j,k+1}(\widehat{\Psi}^{k+1}) + e^{-\lambda(\eta_{j+1} - \psi^k_k \vee \eta_j)} 1_{\{\widehat{\tau}^n_{j,k} \geq \eta_{j+1}\}} \bm{\gamma}^n_{j+1,k}(\Psi^k) \big| \mc{F}^{\Psi^k}_{\psi^k_k \vee \eta_j} \Bigg],
\end{eqnarray*}
}}

\noindent which by the induction hypotheses (first induction for the third term and second induction for the second term) is equal to
\begin{flalign}\label{prop2.3.7}
& E \Bigg[ \int_{\psi^k_k \vee \eta_j}^{\widehat{\tau}^n_{j,k} \wedge \eta_{j+1}} e^{-\lambda (s - \psi^k_k \vee \eta_j)} \left (\varphi \left(s - \psi^k_k \vee \eta_j,\Phi^{\Psi^k}_{\psi^k_k \vee \eta_j} \right) - \frac{\lambda}{c} \right) ds & \\
&+ e^{-\lambda(\widehat{\tau}^n_{j,k} - \psi^k_k \vee \eta_j)} 1_{\{\widehat{\tau}^n_{j,k} < \eta_{j+1}\}} \bm{v}^n_{j,k+1}\left(0,\Phi^{\widehat{\Psi}^{k+1}}_{\widehat{\tau}^n_{j,k}} \right) & \nonumber \\
& + e^{-\lambda(\eta_{j+1} - \psi^k_k \vee \eta_j)} 1_{\{\widehat{\tau}^n_{j,k} \geq \eta_{j+1}\}} \bm{v}^n_{j+1,k} \Big((\eta_{j+1} - \psi^k_k \vee \eta_j) & \nonumber\\
& \ \ \ \ \ + (\psi^k_k \vee \eta_j - \psi^k_k),\varphi \left(\eta_{j+1} - \psi^k_k \vee \eta_j, \Phi^{\Psi^k}_{\psi^k_k \vee \eta_j} \right) \Big) \big| \mc{F}^{\Psi^k}_{\psi^k_k \vee \eta_j} \Bigg], & \nonumber
\end{flalign}
We will now argue as in \eqref{prop2.2.4} and the  discussion following it.  Using the fact that $\eta_{j+1} - \psi^k_k \vee \eta_j$ on the set $\{\psi^k_k < \eta_{j+1}\}$ is independent of $\mc{F}^{\Psi^k}_{\psi^k_k \vee \eta_j}$ and exponentially distributed with parameter $\mu$, and noting $\widehat{\tau}^n_{j,k} - \psi^k_k \vee \eta_j = o^n_{j,k} = o^n_{j,k} \left(\psi^k_k \vee \eta_j - \psi^k_k,\Phi^{\Psi^k}_{\psi^k_k \vee \eta_j} \right)$, \eqref{prop2.3.7} becomes
\begin{eqnarray*}
\lefteqn{\int_0^\infty \mu e^{-\mu u} \Bigg[ \int_0^{o^n_{j,k} \wedge u} e^{-\lambda s} \left( \varphi \left(s, \Phi^{\Psi^k}_{\psi^k_k \vee \eta_j} \right) - \frac{\lambda}{c} \right) ds}  \\
&&+ e^{-\lambda o^n_{j,k}} 1_{\{o^n_{j,k} < u\}} \bm{K} \bm{v}^n_{j,k+1} \left(o^n_{j,k} + (\psi^k_k \vee \eta_j - \psi^k_k),\varphi \left(-(\psi^k_k \vee \eta_j - \psi^k_k),\Phi^{\Psi^k}_{\psi^k_k \vee \eta_j} \right) \right) \\
&&+ e^{-\lambda u} 1_{\{o^n_{j,k} \geq u\}} \bm{v}^n_{j+1,k} \left(u + (\psi^k_k \vee \eta_j - \psi^k_k),\varphi \left(u, \Phi^{\Psi^k}_{\psi^k \vee \eta_j} \right) \right)\Bigg] du,
\end{eqnarray*}
which is equal to $J^+ \left(\bm{v}^n_{j+1,k},\bm{v}^n_{j,k+1} \right) \left(o^n_{j,k} + \psi^k_k \vee \eta_j - \psi^k_k,\Phi^{\Psi^k}_{\psi^k_k \vee \eta_j} \right)$, which by construction of $o^n_{j,k}$, is equal to 
\[
J^+_0 \left(\bm{v}^n_{j+1,k},\bm{v}^n_{j,k+1} \right) \left( \psi^k_k \vee \eta_j - \psi^k_k, \Phi^{\Psi^k}_{\psi^k_k \vee \eta_j} \right) \\ = \bm{v}^n_{j,k} \left( \psi^k_k \vee \eta_j - \psi^k_k, \Phi^{\Psi^k}_{\psi^k_k \vee \eta_j} \right).
\]
Thus, we have established that 
\[
\bm{\gamma}^n_{j,k}(\Psi^k) \leq \bm{v}^n_{j,k} \left( \psi^k_k \vee \eta_j - \psi^k_k, \Phi^{\Psi^k}_{\psi^k_k \vee \eta_j} \right)
\]
on the set $\{\psi^k_k < \eta_{j+1}\}$.  Equality is now a consequence of Proposition \ref{prop2.2}.  Examining the proof, we immediately deduce \eqref{prop2.3.3} as well.
\end{proof}

%\newpage

%\appendix

%\gdef\thesection{APPENDIX \Alph{section}}

\section{\uppercase{Appendix A: Posterior Dynamics}}\label{jfun}

In this Appendix, we derive the recursive formula \eqref{stateprocess}.  It will be convenient to assume that all observations occur at deterministic times, following \cite{MR2777513}.  The reduction to this case from observations at stopping times follows immediately from iteratively taking conditional expectations.  All of the material in this appendix may be found in \cite{MR2777513}. 

On some probability space $(\widetilde{\Omega}, \widetilde{\mc{F}}, P)$, suppose that $\widetilde{X}$ is a standard Brownian Motion which gains drift $\alpha$ at time $\widetilde{\Theta}$, where $P(\widetilde{\Theta} = 0) = \pi$ and $P(\widetilde{\Theta} \in dt | \Theta > 0) = \lambda e^{-\lambda t} dt$.  Let $0 = t_0 < t_1 < \cdots$ be a fixed infinite sequence of numbers, describing the times at which $\widetilde{X}$ is observed.  Let
\[
L_t(u,x_0,x_1,\ldots) \triangleq \prod_{l \geq 1, t_l \leq t} \frac{1}{\sqrt{2\pi(t_l - t_{l-1})}}t \exp \left \{\frac{[x_l - x_{l-1} - \alpha(t_l - t_{l-1} \vee u)^+]^2}{2(t_l - t_{l-1})} \right \}.
\]

Then 
\[
P(\widetilde{X}_{t_l} \in dx_l \text{ for all } l \geq 1 \text{ s.t. } t_l \leq t) = L_t(\widetilde{\Theta}, x_0, x_1,\ldots) \prod_{l \geq 1, t_l \leq t} d x_l.
\]

Therefore, the conditional likelihood of the observations $\widetilde{X}_{t_0},\widetilde{X}_{t_1},\ldots,$ given $\widetilde{\Theta} = u$, is
\[
\begin{split}
L_t(u) 
& \triangleq L_t(u,\widetilde{X}_{t_0},\widetilde{X}_{t_1}, \ldots)
\\ & = \prod_{l \geq 1, t_l \leq t} \frac{1}{\sqrt{2\pi(t_l - t_{l-1})}}t \exp \left \{\frac{[\widetilde{X}_{t_l} - \widetilde{X}_{t_{l-1}} - \alpha(t_l - t_{l-1} \vee u)^+]^2}{2(t_l - t_{l-1})} \right \}.
\end{split}
\]

To this point, we have already assumed that such a process $\widetilde{X}$ exists.  The actual construction starts from a standard Brownian Motion $X$, and is then achieved via a change of measure.  To that end, let $(\Omega, \mc{F}, P_\infty)$ support a standard Brownian Motion $X$ and an independent random variable $\Theta$ with $P_\infty(\Theta = 0) = \pi$, and $P_\infty(\Theta \in dt | \Theta > 0 ) = \lambda e^{-\lambda t} dt$ for $t > 0$.

Then, $P_\infty \left \{X_{t_l} \in dx_l \text{ for all } l \geq 1, t_l \leq t \right \}$ is
\[
L_t(\infty, x_0, x_1,\ldots) \prod_{l \geq 1, t_l \leq t} dx_l = \prod_{l \geq 1, t_l \leq t} \frac{1}{\sqrt{2\pi(t_l - t_{l-1})}} \exp \left \{ \frac{\left(x_l - x_{l-1} \right)^2}{2(t_l - t_{l-1})} \right \} d x_l
\]
for all $t \geq 0$.  Let $\mb{F}$ be the filtration obtained by observing $X$ at fixed times $0 = t_0 < t_1 < \cdots$, and let $\mb{G} = \left( \mc{G}_t \right)_{t \geq 0}$ be the augmentation of $\mb{F}$ by $\sigma(\Theta)$, so that $\mc{G}_t = \mc{F}_t \wedge \sigma(\Theta)$.

Define $P$ on $\mc{G}_\infty$, locally along the filtration, by

{\small{
\[
\begin{split}
\frac{dP}{dP_\infty} 
& = Z_t(\Theta) \\
& \triangleq \frac{L_t(\Theta)}{L_t(\infty)}
\\& = \exp \left \{ \sum_{l = 1}^\infty 1_{\{t_l \leq t\}} \left[ \frac{\left(X_{t_l} - X_{t_{l-1}} \right)\alpha \left(t_l - \left(\Theta \vee t_{l-1} \right) \right)^+}{t_l - t_{l-1}} - \frac{\alpha^2\left( \left(t_l - \left(\Theta \vee t_{l-1} \right) \right)^+ \right)^2}{2(t_l - t_{l-1})} \right] \right \}.
\end{split}
\]
}}

Under $P$, the random variables $X_{t_l} - X_{t_{l-1}}$, $l \geq 1$, conditionally on $\Theta$, are independent and Gaussian with mean $\alpha \left(t_l - \left(\Theta \vee t_{l-1} \right) \right)^+$ and variance $t_l - t_{l-1}$.

Since $Z_0(\Theta) = 1$, $P$ and $P_\infty$ are identical on $\mc{G}_0 = \sigma(\Theta)$, so that $\Theta$ has the same distribution under $P$ and $P_\infty$.  Under $P$, $X$ has the distribution of a Brownian Motion which gains drift $\alpha$ at time $\Theta$.  We will work under $P$ for the remainder of this Appendix.

Define the conditional odds process
\begin{equation}\label{oddsdef}
\Phi_t \triangleq \frac{P \left(\Theta \leq t | \mc{F}_t \right)}{P \left( \Theta > t | \mc{F}_t \right)} = \frac{E_\infty \left[Z_t(\Theta) 1_{\{\Theta \leq t\}} | \mc{F}_t \right]}{E_\infty \left[ Z_t(\Theta) 1_{\left \{\Theta > t \right\}} | \mc{F}_t \right]},
\end{equation}

\noindent with the second equality following from Bayes' Theorem.  On the set $\{ \Theta > t \}$, $\left( t_l - \left(\Theta \vee t_{l-1} \right) \right)^+ = \left(t_l - \Theta \right)^+ = 0$ for all $l \geq 1, t_l \leq t$.  Therefore, $Z_t(\Theta) 1_{ \left \{\Theta > t \right \}} = 1_{ \left\{\Theta > t \right\}}$.

Thus,
\[
\begin{split}
E_\infty \left[Z_t(\Theta) 1_{\left\{\Theta > t \right\}} | \mc{F}_t \right] & = P_\infty \left(\Theta > t | \mc{F}_t \right) \\
& = P_\infty(\Theta > t) \\
& = (1 - \pi) e^{-\lambda t}.
\end{split}
\]
So, \eqref{oddsdef} becomes 
\begin{equation}\label{odds2}
\frac{E_\infty \left[ Z_t(\Theta) 1_{\left\{\Theta \leq t \right\}} | \mc{F}_t \right]}{(1-\pi)e^{-\lambda t}} = \frac{e^{\lambda t}}{1 - \pi} E_\infty \left[ Z_t(\Theta) 1_{\left\{\Theta \leq t \right\}} | \mc{F}_t \right] .
\end{equation}

We will now focus on this last term in \eqref{odds2}.  Write 
\begin{equation}\label{odds3}
E_\infty \left[ Z_t(\Theta) 1_{\left\{\Theta \leq t \right\}} | \mc{F}_t \right] = \pi Z_t(0) + (1 - \pi) \int_0^t \lambda e^{-\lambda t} Z_t(u) du.
\end{equation}
Suppose that $t_{n-1} \leq t < t_n$ for some $n \geq 1$.  By definition, $Z_t(u) = Z_{t_{n-1}}(u)$ for every $u \geq 0$, and $Z_{t_{n-1}}(u) = 1$ for every $t_{n-1} \leq u < t_n$.  This implies that \eqref{odds3} is
\begin{eqnarray*}
\lefteqn{\pi Z_{t_{n-1}}(0) + (1 - \pi)\int_0^t \lambda e^{-\lambda u} Z_{t_{n-1}}(u) du} \\
&& = \pi Z_{t_{n-1}}(0) + (1 - \pi)\int_0^{t_{n-1}} \lambda e^{-\lambda u} Z_{t_{n-1}}(u) du + (1 - \pi) \int_{t_{n-1}}^t \lambda e^{-\lambda u} Z_{t_{n-1}}(u) du \\
&& = \frac{1 - \pi}{e^{\lambda t_{n-1}}} \Phi_{t_{n-1}} + (1 - \pi) \left(e^{-\lambda t_{n-1}} - e^{-\lambda t} \right).
\end{eqnarray*}
From this, it follows that, for $t_{n-1} \leq t < t_n$, we have that \eqref{odds2} is equal to
\[
e^{\lambda(t - t_{n-1})} \Phi_{t_{n-1}} + e^{\lambda(t - t_{n-1})} - 1 = \varphi \left(t - t_{n-1}, \Phi_{t_{n-1}} \right),
\]
and this establishes the first part of \eqref{stateprocess}.  We now derive the form of $\Phi_{t_n}$, conditionally on $\Phi_{t_{n-1}}$.  Since $Z_{t_{n-1}}(u) = 1$ for $u \geq t_{n-1}$, we have

{\small{
\begin{eqnarray*}
\lefteqn{Z_{t_n}(u)} \\
&& = Z_{t_{n-1}}(u) \exp \left \{ \frac{\left(X_{t_n} - X_{t_{n-1}} \right) \alpha \left( t_n - \left(u \vee t_{n-1} \right) \right)^+}{t_n - t_{n-1}} - \frac{\alpha^2 \left( \left( t_n - \left(u \vee t_{n-1} \right) \right)^+ \right)^2}{2(t_n - t_{n-1})} \right\}, u \geq 0.
\end{eqnarray*}
}}

So, \eqref{odds2} becomes
{\small{
\begin{multline*}
 \frac{e^{\lambda t_n}}{1 - \pi} \Bigg[ \left(\pi Z_{t_{n-1}}(0) + (1 - \pi) \int_0^{t_{n-1}} \lambda e^{-\lambda u} Z_{t_{n-1}}(u) du \right) \exp \left\{ \left(X_{t_n} - X_{t_{n-1}} \right) \alpha - \frac{\alpha^2}{2}(t_n - t_{n-1}) \right\} 
\\ \ \ \ \ \ \ + (1 - \pi) \int_{t_{n-1}}^{t_n} \lambda e^{-\lambda u} \underbrace{Z_{t_{n-1}}(u)}_{ = 1 } \exp \left\{ \frac{\left(X_{t_n} - X_{t_{n-1}} \right)\alpha \left(t_n -u \right)}{t_n - t_{n-1}} - \frac{\alpha^2 \left(t_n - u \right)^2}{2(t_n - t_{n-1})} \right\} du \Bigg],
\end{multline*}
}}

which is equal to
\begin{multline*}
\exp \left\{ \left(X_{t_n} - X_{t_{n-1}} \right) \alpha - \frac{\alpha^2}{2}(t_n - t_{n-1}) \right\} e^{\lambda(t_n - t_{n-1})} \Phi_{t_{n-1}} 
\\ \ \ \ \ \ \ + \int_{t_{n-1}}^{t_n} \lambda e^{\lambda (t_n - u) }  \exp \left\{ \frac{\left(X_{t_n} - X_{t_{n-1}} \right)\alpha \left(t_n - u \right)}{t_n - t_{n-1}} - \frac{\alpha^2\left(t_n - u \right)^2}{2(t_n - t_{n-1})} \right\} du.
\end{multline*}

After making the substitution $w = -(t_n - u)$ for the integral above, we see that this is equal to the second term in \eqref{stateprocess}.

\pagebreak
\section{\uppercase{AppendiX B: Selected Figures}}\label{figs}

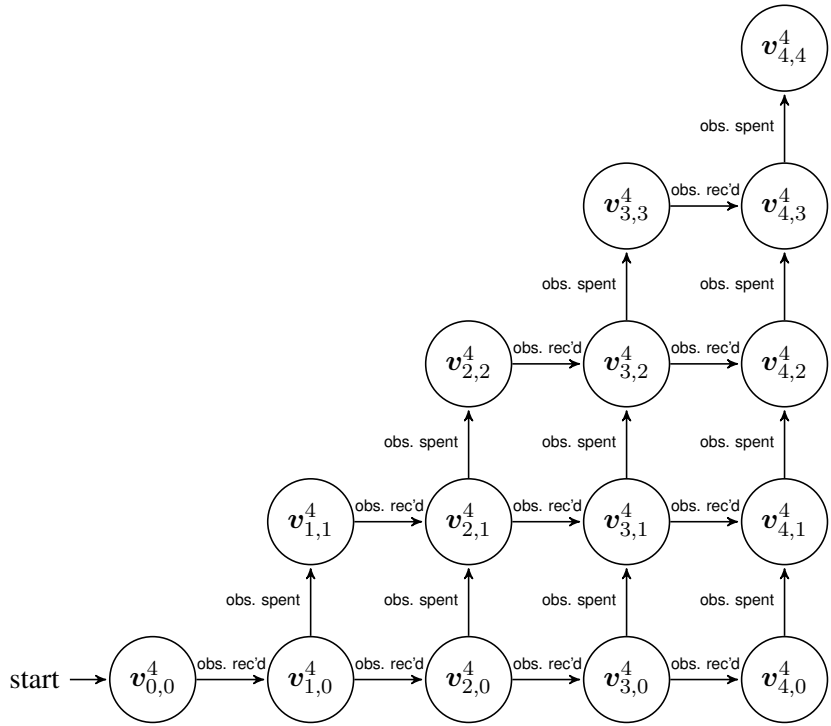
\begin{figure}[h!]

\begin{tikzpicture}[->,>=stealth',shorten >=1pt,auto,node distance=2.1cm,
                    semithick]   
  \tikzstyle{every state}=[fill=white,draw=black,text=black]

  \node[initial,state] (v00)                {$\bm{v}^4_{0,0}$};
  \node[state]         (v10) [right of=v00] {$\bm{v}^4_{1,0}$};
  \node[state]         (v11) [above of=v10] {$\bm{v}^4_{1,1}$};
  \node[state]         (v20) [right of=v10] {$\bm{v}^4_{2,0}$};
  \node[state]         (v21) [above of=v20] {$\bm{v}^4_{2,1}$};
  \node[state]				 (v22) [above of=v21] {$\bm{v}^4_{2,2}$};
  \node[state]	       (v30) [right of=v20] {$\bm{v}^4_{3,0}$};
  \node[state]	       (v31) [above of=v30] {$\bm{v}^4_{3,1}$};
  \node[state]	       (v32) [above of=v31] {$\bm{v}^4_{3,2}$};
  \node[state]	       (v33) [above of=v32] {$\bm{v}^4_{3,3}$};
  \node[state]	       (v40) [right of=v30] {$\bm{v}^4_{4,0}$};
  \node[state]	       (v41) [above of=v40] {$\bm{v}^4_{4,1}$};
  \node[state]	       (v42) [above of=v41] {$\bm{v}^4_{4,2}$};
  \node[state]	       (v43) [above of=v42] {$\bm{v}^4_{4,3}$};
  \node[state]				 (v44) [above of=v43] {$\bm{v}^4_{4,4}$};

  \path[every node/.style={font=\sffamily\tiny}] 
  (v00) edge node [above] {obs. rec'd} (v10)        
  (v10) edge node [left] {obs. spent} (v11)
  (v10) edge node [above] {obs. rec'd} (v20)
  (v11) edge node [above] {obs. rec'd} (v21)
  (v20) edge node [left] {obs. spent} (v21)
  (v20) edge node [above] {obs. rec'd} (v30)
  (v21) edge node [left] {obs. spent} (v22)
  (v21) edge node [above] {obs. rec'd} (v31)
  (v22) edge node [above] {obs. rec'd} (v32)
  (v30) edge node [left] {obs. spent} (v31)
  (v30) edge node [above] {obs. rec'd} (v40)
  (v31) edge node [left] {obs. spent} (v32)
  (v31) edge node [above] {obs. rec'd} (v41)
  (v32) edge node [left] {obs. spent} (v33)
  (v32) edge node [above] {obs. rec'd} (v42)
  (v33) edge node [above] {obs. rec'd} (v43)
  (v40) edge node [left] {obs. spent} (v41)
  (v41) edge node [left] {obs. spent} (v42)
  (v42) edge node [left] {obs. spent} (v43)
  (v43) edge node [left] {obs. spent} (v44)
;

\end{tikzpicture}
\caption{Schematic of Stochastic Arrival Problem, $n=4$}\label{st_scheme}

\end{figure}

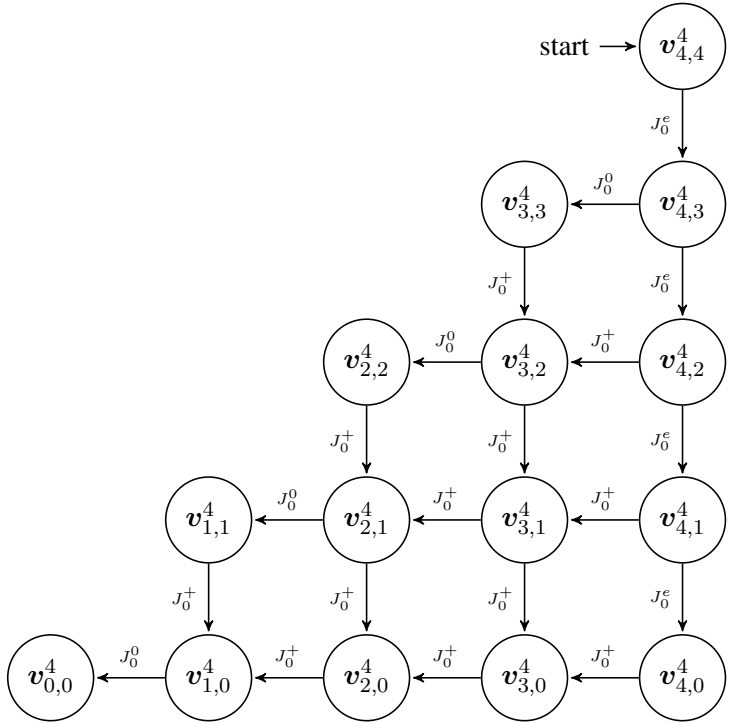
\begin{figure}[h!]

\begin{tikzpicture}[->,>=stealth',shorten >=1pt,auto,node distance=2.1cm,
                    semithick]   
  \tikzstyle{every state}=[fill=white,draw=black,text=black]

  \node[state] (v00)                {$\bm{v}^4_{0,0}$};
  \node[state]         (v10) [right of=v00] {$\bm{v}^4_{1,0}$};
  \node[state]         (v11) [above of=v10] {$\bm{v}^4_{1,1}$};
  \node[state]         (v20) [right of=v10] {$\bm{v}^4_{2,0}$};
  \node[state]         (v21) [above of=v20] {$\bm{v}^4_{2,1}$};
  \node[state]				 (v22) [above of=v21] {$\bm{v}^4_{2,2}$};
  \node[state]	       (v30) [right of=v20] {$\bm{v}^4_{3,0}$};
  \node[state]	       (v31) [above of=v30] {$\bm{v}^4_{3,1}$};
  \node[state]	       (v32) [above of=v31] {$\bm{v}^4_{3,2}$};
  \node[state]	       (v33) [above of=v32] {$\bm{v}^4_{3,3}$};
  \node[state]	       (v40) [right of=v30] {$\bm{v}^4_{4,0}$};
  \node[state]	       (v41) [above of=v40] {$\bm{v}^4_{4,1}$};
  \node[state]	       (v42) [above of=v41] {$\bm{v}^4_{4,2}$};
  \node[state]	       (v43) [above of=v42] {$\bm{v}^4_{4,3}$};
  \node[initial,state]				 (v44) [above of=v43] {$\bm{v}^4_{4,4}$};

  \path[every node/.style={font=\sffamily\tiny}] 
  (v10) edge node [above] {$J^0_0$} (v00)        
  (v11) edge node [left] {$J^+_0$} (v10)
  (v20) edge node [above] {$J^+_0$} (v10)
  (v21) edge node [above] {$J^0_0$} (v11)
  (v21) edge node [left] {$J^+_0$} (v20)
  (v30) edge node [above] {$J^+_0$} (v20)
  (v22) edge node [left] {$J^+_0$} (v21)
  (v31) edge node [above] {$J^+_0$} (v21)
  (v32) edge node [above] {$J^0_0$} (v22)
  (v31) edge node [left] {$J^+_0$} (v30)
  (v40) edge node [above] {$J^+_0$} (v30)
  (v32) edge node [left] {$J^+_0$} (v31)
  (v41) edge node [above] {$J^+_0$} (v31)
  (v33) edge node [left] {$J^+_0$} (v32)
  (v42) edge node [above] {$J^+_0$} (v32)
  (v43) edge node [above] {$J^0_0$} (v33)
  (v41) edge node [left] {$J^e_0$} (v40)
  (v42) edge node [left] {$J^e_0$} (v41)
  (v43) edge node [left] {$J^e_0$} (v42)
  (v44) edge node [left] {$J^e_0$} (v43)
;

\end{tikzpicture}
\caption{Recursive Computation of Value Functions, $n=4$}\label{st_comp}

\end{figure}

\newpage

\section*{\uppercase{Acknowledgements}}

This work is supported by the National Science Foundation under grant DMS-1118673. We are grateful to the editor Professor Nitis Mukhopadhyay for a very timely handling of our paper.

\renewcommand{\refname}{{\large \bf REFERENCES}}

%\bibliographystyle{apacite}

%\bibliographystyle{plain}

%\bibliography{q.d.fbib}

\begin{thebibliography}{23}
\providecommand{\natexlab}[1]{#1}
\providecommand{\url}[1]{\texttt{#1}}
\expandafter\ifx\csname urlstyle\endcsname\relax
  \providecommand{\doi}[1]{doi: #1}\else
  \providecommand{\doi}{doi: \begingroup \urlstyle{rm}\Url}\fi

\bibitem[Aldous(1981)]{aldous1981weak}
Aldous, D. (1981).
\newblock {Weak Convergence and the General Theory of Processes}.
\newblock \emph{Incomplete draft of a monograph. Available at
  http://www.stat.berkeley.edu/~aldous/Papers/weak-gtp.pdf}.

\bibitem[Antelman and Savage(1965)]{MR0189833}
Antelman, G. and Savage, R. (1965).
\newblock {Surveillance Problems: {W}iener Processes},
\newblock \emph{Naval Research Logistics Quarterly} 12: 35--55.

\bibitem[Balmer(1975)]{MR0431582}
Balmer, D. (1975).
\newblock {On a Quickest Detection Problem with Costly Information},
\newblock \emph{Journal of Applied Probability} 12: 87--97.

\bibitem[Banerjee and Veeravalli(2012)]{banerjee2012data}
Banerjee ,T. and Veeravalli, V. (2012).
\newblock {Data-Efficient Quickest Change Detection with On--Off Observation
  Control},
\newblock \emph{Sequential Analysis} 31: 40--77.

\bibitem[Banerjee and Veeravalli(2013)]{banerjee2012data2}
Banerjee, T. and Veeravalli, V (2013).
\newblock {Data-Efficient Quickest Change Detection in Minimax Settings},
\newblock \emph{IEEE Transactions on Information Theory} 59: 6917--6931.

\bibitem[Bather(1973)]{MR0402919}
Bather, J. (1973).
\newblock {An Optimal Stopping Problem with Costly Information} 45: 9--24.

\bibitem[Bayraktar and Dayanik(2006)]{MR2233993}
Bayraktar, E. and Dayanik, S. (2006).
\newblock {Poisson Disorder Problem with Exponential Penalty for Delay},
\newblock \emph{Mathematics of Operations Research} 31: 217--233.

\bibitem[Bayraktar and Fahim(2014)]{BF11}
Bayraktar, E. and Fahim, A. (2014).
\newblock A stochastic approximation for fully nonlinear free boundary
  parabolic problems,
\newblock \emph{Numerical Methods for Partial Differential Equations}
  30: 902--929.

\bibitem[Bayraktar et~al.(2005)Bayraktar, Dayanik, and Karatzas]{MR2158013}
Bayraktar, E., Dayanik, S.,  and Karatzas, I. (2005).
\newblock {The Standard {P}oisson Disorder Problem Revisited},
\newblock \emph{Stochastic Processes and their Applications} 115: 1437--1450.

\bibitem[Bayraktar et~al.(2006)Bayraktar, Dayanik, and Karatzas]{MR2260062}
Bayraktar, E., Dayanik, S.,  and Karatzas, I. (2006).
\newblock {Adaptive {P}oisson Disorder Problem}.
\newblock \emph{The Annals of Applied Probability} 16: 1190--1261.

\bibitem[Bayraktar et~al.(2012)Bayraktar, Geng, and Lai]{bayraktar-lai}
Bayraktar, E., Geng, J., and Lai, L. (2012).
\newblock {Quickest Change Point Detection with Sampling Right Constraints},
\newblock in \emph{Proceedings of the 50th Annual Allerton Conference on
  Communication, Control, and Computing}, pp. 874 - 881.

\bibitem[Beibel(1997)]{MR1483699}
Beibel, M. (1997).
\newblock {Sequential Change-Point Detection in Continuous Time When The
  Post-Change Drift is Unknown},
\newblock \emph{Bernoulli} 3: 457--478.

\bibitem[Beibel(2000)]{MR1835037}
Beibel, M. (2000).
\newblock {A Note on Sequential Detection with Exponential Penalty for the
  Delay},
\newblock \emph{The Annals of Statistics} 28: 1696--1701.

\bibitem[Dalang and Shiryaev(2014)]{DalangShiryaev}
Dalang, R. and Shiryaev, A. (2014).
\newblock A quickest detection problem with an observation cost.
\newblock \emph{Annals of Applied Probability, in press}.

\bibitem[Dayanik(2010)]{MR2777513}
Dayanik, S. (2010).
\newblock {Wiener Disorder Problem with Observations at Fixed Discrete Time
  Epochs},
\newblock \emph{Mathematics of Operations Research} 35: 756--785.

\bibitem[Gapeev and Peskir(2006)]{MR2307058}
Gapeev, P. and Peskir, G. (2006).
\newblock {The {W}iener Disorder Problem with Finite Horizon},
\newblock \emph{Stochastic Processes and their Applications} 116: 1770--1791.

\bibitem[Geng et~al.(2014)Geng, Bayraktar, and Lai]{6862025}
Geng, J., Bayraktar, E., and Lai, L. (2014).
\newblock {Bayesian Quickest Change-Point Detection with Sampling Right
  Constraints}.
\newblock \emph{IEEE Transactions on Information Theory} 60: 6474--6490.

\bibitem[Higham et~al.(2011)Higham, Mao, Roj, Song, and Yin]{higham2011mean}
Higham, D., Mao, X.,  Roj, M,  Song, Q., and  Yin, G. (2013). 
\newblock {Mean Exit Times and the {Multi-Level Monte Carlo Method}},
\newblock \emph{SIAM/ASA Journal on Uncertainty Quantification} 1: 2--18.

\bibitem[Peskir and Shiryaev(2002)]{MR1929384}
Peskir, G. and Shiryaev, A. (2002).
\newblock {Solving the Poisson Disorder Problem},
\newblock in K.~Sandmann and P.~Schonbucher, editors, \emph{Advances in Finance
  and Stochastics}, pp. 295--312, Berlin: Springer. 

\bibitem[Peskir and Shiryaev(2006)]{MR2256030}
Peskir, G. and Shiryaev, A. (2006).
\newblock \emph{{Optimal Stopping and Free-Boundary Problems}},
 Basel: Birkh\"auser.

\bibitem[Shiryaev(1963)]{shiryaev1963optimum}
 Shiryaev, A. (1963).
\newblock {On Optimum Methods in Quickest Detection Problems},
\newblock \emph{Theory of Probability and Its Applications} 8: 22--46.

\bibitem[Shiryaev(2008)]{MR2374974}
Shiryaev, A. (2008).
\newblock \emph{{Optimal Stopping Rules}},
\newblock Berlin: Springer. 

\bibitem[Wagner(1977)]{MR0486391}
Wagner, D. (1977)
\newblock {Survey of Measurable Selection Theorems},
\newblock \emph{SIAM Journal on Control and Optimization}, 15:  859--903.

\end{thebibliography}

\end{document}